\documentclass[oneside]{article}

\usepackage[english]{babel}
\usepackage[utf8]{inputenc}
\usepackage[T1]{fontenc}

\usepackage{csquotes}
\usepackage[backend=biber,style=alphabetic]{biblatex}
\DeclareSourcemap{
	\maps[datatype=bibtex]{
		\map{
			\step[fieldset=urldate,null]
		}
	}
}

\usepackage{amsmath, amstext, amsthm,amssymb,enumitem}
\usepackage{mathrsfs} % Package for mathscr characters
\usepackage{dsfont} % Package for \ind below
\usepackage{amsbsy} % Package for bold symbols
\usepackage{mathtools} % Pour mathclap

\usepackage{tikz} % Package for use of tickz pitures
\usepackage{graphicx}
\usepackage{caption}
\usepackage{subcaption}
\usepackage[all]{xy}

\usepackage{hyperref}
\hypersetup{colorlinks = true, urlcolor = blue, citecolor=blue}
\numberwithin{equation}{section}
\usepackage{cleveref}

\theoremstyle{plain}
\newtheorem{theorem}{Theorem}

\newtheorem{prop}{Proposition}
\newtheorem{lemma}{Lemma}[section]
\newtheorem{hypo}{Hypothesis}

\newtheorem*{lemma*}{Lemma}

\renewenvironment{proof}{\noindent\textbf{Proof }}{\hspace*{\fill}$\Box$\medskip}

\def \ind {\mathds{1}}
\def \nug {\boldsymbol{\nu}}
\def \pg {\textbf{p}}
\def \Pg {\textbf{P}}
\def \dsp {\overline{u}}
\def \Ldsp {{\Lcc}}
\def \dspD {\overline{u}^\delta}
\def \LdspD {{L^\delta}}

\def \hg {\textbf{h}}
\def \QD {Q^\delta}
\def \Rayon {R}

\def \C {\mathbb{C}}

\def \E {\mathbb{E}}

\def \N {\mathbb{N}}

\def \R {\mathbb{R}}
\def \S {\mathbb{S}}

\def \Z {\mathbb{Z}}

\def \Cc {\mathcal{C}}

\def \Fc {\mathcal{F}}

\def \Lc {\mathcal{L}}

\def \Nc {\mathcal{N}}
\def \Oc {\mathcal{O}}

\def \Tc {\mathcal{T}}
\def \Uc {\mathcal{U}}

\def \Ccc {\mathscr{C}}

\def \Gcc {\mathscr{G}}

\def \Lcc {\mathscr{L}}

\begin{document}
	\begin{center}
		{\large \bfseries Nonlinear orbital stability of stationary discrete shock profiles for scalar conservation laws}
	\end{center}
	
	\begin{center}
		Lucas \textsc{Coeuret}\footnote{Dipartimento di Matematica “Tullio Levi-Civita”, Università di Padova, Via Trieste 63, 35121 Padova, Italy. ORCID: 0009-0009-6746-4786.  Research of the author was supported by the Agence Nationale de la Recherche project Indyana (ANR-21-CE40-0008), as well as by the Labex Centre International de Mathématiques et Informatique de Toulouse under grant agreement ANR-11-LABX-0040. The author also would like to thank the Italian Ministry of University and Research (MUR) to support this research with funds coming from PRIN Project 2022 PNRR (No. 2022XJ9SX entitled “Heterogeneity on the road - Modeling, analysis, control”). E-mail: lucas.coeuret@math.unipd.it}
	\end{center}
	
		\vspace{5mm}
	
	\begin{center}
		\textbf{Abstract}
	\end{center}
	
	\noindent For scalar conservation laws, we prove that spectrally stable stationary Lax discrete shock profiles are nonlinearly stable in some polynomially-weighted $\ell^1$ and $\ell^\infty$ spaces. In comparison with several previous nonlinear stability results on discrete shock profiles, we avoid the introduction of any weakness assumption on the amplitude of the shock and apply our analysis to a large family of schemes that introduce some artificial possibly high-order viscosity. The proof relies on a precise description of the Green's function of the linearization of the numerical scheme about spectrally stable discrete shock profiles obtained in \cite{Coeuret2024d}. The present article also pinpoints the ideas for a possible extension of this nonlinear orbital stability result for discrete shock profiles in the case of systems of conservation laws.
	
	\vspace{5mm}
	
	
	\vspace{5mm}
	
	
	\section*{Notations}
	
	For $1\leq r <+\infty$, we let $\ell^r(\Z)$ denote the Banach space of complex valued sequences indexed by $\Z$ and such that the norm:
	$$\left\|u\right\|_{\ell^r}:=\left(\sum_{j\in\Z}|u_j|^r\right)^\frac{1}{r}$$
	is finite. We also let $\ell^\infty(\Z)$ denote the Banach space of bounded complex valued sequences indexed by $\Z$ equipped with the norm
	$$\left\|u\right\|_{\ell^\infty}:=\sup_{j\in \Z}|u_j|.$$
	Within the article, we will also introduce some polynomially-weighted $\ell^r$-spaces defined by \eqref{def:EspAPoids}.
	
	Furthermore, for $E$ a Banach space, we denote $\Lc(E)$ the space of bounded operators acting on $E$ and $\left\|\cdot\right\|_{\Lc(E)}$ the operator norm. For $T$ in $\Lc(E)$, the notation $\sigma(T)$ stands for the spectrum of the operator $T$ and $\rho(T)$ denotes the resolvent set of $T$.
	
	Finally, we may use the notation $\lesssim$ to express an inequality up to a multiplicative constant. Eventually, we let $C$ (resp. $c$) denote some large (resp. small) positive constants that may vary throughout the text (sometimes within the same line). Furthermore, we use the usual Landau notation $O(\cdot)$ to introduce a term uniformly bounded with respect to the argument.
	
	\tableofcontents

	\section{Introduction and main result}
	
	This introductory section is separated in four parts. First, in Section \ref{subsec:PosProb}, we will recall the notions of \textit{(stationary) discrete shock profiles} and of \textit{nonlinear orbital stability}. We will also present a quick state of the art on the subject of the stability of discrete shock profiles. Sections \ref{subsec:Lin} and \ref{subsec:Green} will be dedicated to recall the notion of \textit{spectral stability} of a discrete shock profile as well as to recall the result \cite[Theorem 1]{Coeuret2024d} which will play a central role in the present article. Finally, the main result Theorem \ref{Th} of the article will be stated in Section \ref{subsec:MainRes}. It is a nonlinear orbital stability result for spectrally stable stationary discrete shock profiles of scalar conservation laws. We will also present the improvements and limitations of Theorem \ref{Th} with respect to the aforementioned state of the art.
	
	\subsection{Position of the problem}\label{subsec:PosProb}
	
	\vspace{0.1cm}\textbf{\underline{Definition of the scalar conservation law and the stationary Lax shock}}\vspace{0.1cm}
	
	We consider a one-dimensional scalar conservation law
	\begin{equation}\label{def:EDP}
		\begin{array}{c}
			\partial_t u+\partial_x f(u)=0, \quad t\in\R_+, x\in\R,\\
			u:\R_+\times\R\rightarrow \Uc,
		\end{array}
	\end{equation}
	where the space of states $\Uc$ is an open set of $\R$ and the flux $f:\Uc\rightarrow \R$ is a $\Ccc^\infty$ function. Solutions of conservation laws tend to have discontinuities, even for smooth initial data. When considering approximations of solutions of conservation laws by finite difference methods, one of the central questions concerns the ability of approaching discontinuities. Indeed, finite difference methods are usually deduced using Taylor expansions and are thus essentially more adapted to approximate smooth solutions. Thus, a first step is to answer whether shocks are well approximated by finite difference schemes. The notion of discrete shock profiles defined below is a central tool for this purpose.
	
	In the present paper, we focus our attention on the approximation of stationary Lax shocks by finite difference methods. We consider two states $u^-,u^+\in\Uc$ such that:
	\begin{equation}\label{cond:RK}
		f(u^-)=f(u^+).
	\end{equation}
	This corresponds to the well-known Rankine-Hugoniot condition which allows us to conclude that the standing shock $u$ defined by:
	\begin{equation}\label{def:choc}
		\forall t\in\R_+,\forall x\in\R,\quad u(t,x):=\left\{ \begin{array}{cc}u^- & \text{ if }x<0, \\ u^+ & \text{ else},\end{array}\right.
	\end{equation}
	is a weak solution of \eqref{def:EDP}. We also introduce the so-called Lax shock condition (see \cite{Lax1957}):
	\begin{equation}\label{cond:Lax}
		f^\prime(u^+)<0<f^\prime(u^-).
	\end{equation}
	Let us point out that starting by studying the discrete approximation of Lax shocks is fairly logical. Indeed, for convex fluxes $f$, Lax shocks correspond exactly to the entropic shocks. Furthermore, for more general fluxes, weak genuinely nonlinear Lax shocks are also entropic solutions of \eqref{def:EDP}.
	
	\vspace{0.1cm}\noindent\textbf{\underline{Definition of the numerical scheme and of the discrete shock profiles}}\vspace{0.1cm}
	
	We consider a constant $\nug>0$ and introduce a space step $\Delta x>0$ and a time step $\Delta t:=\nug \Delta x>0$. We impose the following CFL condition on the constant $\nug$:\footnote{Up to considering that the space of state $\Uc$ is a close neighborhood of the SDSPs $\dspD$ defined underneath in Hypothesis \ref{H:SDSP}, we should be able to satisfy such a condition.}
	\begin{equation}\label{cond:CFL}
		\forall u\in\Uc,\quad -q< \nug f^\prime(u)<p.
	\end{equation}	
	We introduce the \textit{nonlinear} discrete evolution operator $\Nc: \Uc^\Z \rightarrow \R^\Z$ defined for $u=(u_j)_{j\in\Z}\in\Uc^\Z$ as:
	\begin{equation}\label{def:evolOpe}
		\forall j\in \Z, \quad (\Nc (u))_j := u_j- \nug\left(F\left(\nug;u_{j-p+1},\hdots, u_{j+q}\right)-F\left(\nug;u_{j-p},\hdots, u_{j+q-1}\right)\right),
	\end{equation}
	where $p,q\in\N\backslash\lbrace0\rbrace$ and the numerical flux $F:(\nu;u_{-p},\hdots, u_{q-1})\in]0,+\infty[\times \Uc^{p+q}\rightarrow \R$ is a $\Cc^\infty$ function. We are interested in solutions of the conservative one-step explicit finite difference scheme defined by: 
	\begin{equation}\label{def:SchemeNum}
		\forall n\in\N,\quad u^{n+1}=\Nc (u^n)
	\end{equation}
	with the initial condition $u^0\in\Uc^\Z$. 
	
	We assume that the numerical scheme satisfies the following usual consistency condition with regards to the conservation law \eqref{def:EDP}:
	\begin{equation}\label{cond:consistency}
		\forall u \in\Uc, \quad F(\nug;u,\hdots,u)=f(u).
	\end{equation}
	Let us point out that we will also introduce an additional hypothesis later on (Hypothesis \ref{H:F}) on the numerical schemes considered in this article which essentially states that the scheme introduces some artificial viscosity.
	
	Let us now introduce the so-called \textit{discrete shock profiles} which will be the main mathematical object of the article. The presentation done here will be fairly brief and we invite the interested reader to take a closer look at \cite{Serre2007} for a general state of the art on discrete shock profiles.
	
	In a general setting (i.e. for moving shocks), a discrete shock profile is a solution of the numerical scheme \eqref{def:SchemeNum} which is a traveling wave linking the two states of the shock. We pinpoint the fact that the consistency condition \eqref{cond:consistency} satisfied by the numerical scheme implies that the existing discrete shock profiles must travel at the same speed as the shock they are associated with while recalling that this speed is deduced using the well-known Rankine-Hugoniot condition (see \cite{Serre2007} for a detailed proof). In the case of stationary shocks which is studied in the present paper, stationary discrete shock profiles (SDSPs) are thus stationary solution of the numerical scheme \eqref{def:SchemeNum} linking the two states $u^-$ and $u^+$ of the shock.
	
	\begin{figure}
		\centering
		\includegraphics[width=0.9\textwidth]{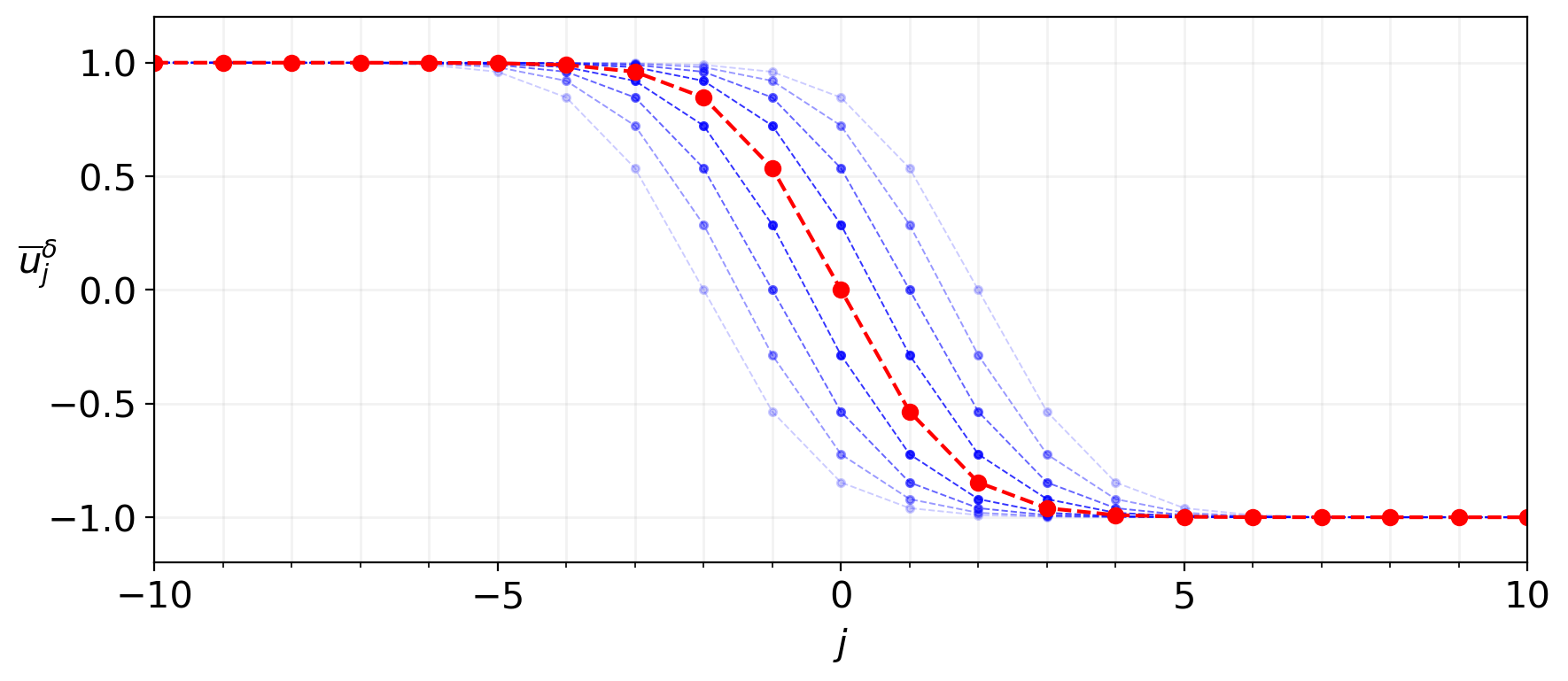}
		\caption{An example of stationary discrete shock profiles (SDSPs). The chosen scalar conservation law is Burger's equation $f(u):=\frac{u^2}{2}$. The numerical scheme is the modified Lax-Friedrichs scheme defined below as \eqref{def:MLF} with $\nug=0.5$ and $D=\frac{0.4}{\nug}$ and the shock considered is the one associated with the states $u^-=1$ and $u^+=-1$. The red solution is a SDSP as well as all the blue ones. There exists a continuum of discrete shock profiles as stated in Hypothesis \ref{H:SDSP}. We also observe that Hypotheses \ref{H:CVExpo} and \ref{H:identification} seem to be verified.}
		\label{fig:SDSP}
	\end{figure}
	
	\begin{hypo}[Existence of a family of stationary discrete shock profiles (SDSPs)]\label{H:SDSP}
		There exists an open interval $I$ containing $0$ and a continuous family of sequences $(\dspD)_{\delta\in I} \in\left(\Uc^\Z \right)^I$ that satisfy for each $\delta\in I$: 
		$$\Nc(\dspD)=\dspD \quad \text{and} \quad \dspD_j\underset{j\rightarrow \pm\infty}\rightarrow u^\pm.$$
	\end{hypo}

	 The sequences $\dspD$ are stationary discrete shock profiles (SDSPs) associated with the shock \eqref{def:choc} we introduced earlier. We refer to Figure \ref{fig:SDSP} for a visual representation of such SDSPs. The existence theory of discrete shock profiles, and thus the verification of Hypothesis \ref{H:SDSP}, has been a longstanding question \cite{Jennings1974,Majda1979,Michelson1984,Smyrlis1990,Serre2004}. We refer the interested reader to \cite{Serre2007} for a general overview on the existence theory. To quickly summarize, most existence results tend to necessitate fairly strict additional hypotheses. For instance, in \cite{Jennings1974}, one can find an existence result for discrete shock profiles associated with Lax shocks for \textit{monotone} schemes approximating \textit{scalar} conservation laws. On the other hand, existence results of discrete shock profiles for \textit{systems }of conservation laws such as \cite{Majda1979,Michelson1984} tend to introduce a \textit{weakness} assumption on the amplitude of the approximated shocks.
	 
	 Let us point out that, in the present article, one of the discrete shock profile will be fixed as a reference discrete shock profile compared to the others and will have to satisfy additional hypotheses, in particular a \textit{spectral stability} assumption (see Hypotheses \ref{H:spec} and \ref{H:Evans} below). We thus assumed in Hypothesis \ref{H:SDSP} that $0$ belongs to the interval $I$ and we will use the notation:
	 \begin{equation}\label{def:DSP}
	 	\dsp:=\overline{u}^0
	 \end{equation}
	 to identify this discrete shock profile more clearly from now on. The additional assumptions on the discrete shock profile $\dsp$ that we will introduce aim at allowing the use of the result \cite[Theorem 1]{Coeuret2024d} in Section \ref{subsec:Green}.
	 
	 We now introduce two additional assumptions (Hypotheses \ref{H:CVExpo} and \ref{H:identification}) on the family $(\dspD)_{\delta\in I}$ of SDSPs.
	
	\begin{hypo}[Exponential localization]\label{H:CVExpo}
		We assume that there exist two positive constants $C,c$ such that:
		\begin{subequations}\label{DSP_CV_ExpoEtatLim}
			\begin{align}
				\forall j \in\N,\quad & \left|\dsp_j-u^+\right|\leq Ce^{-cj},\label{DSP_CV_Expo+}\\
				\forall j \in\N,\quad & \left|\dsp_{-j}-u^-\right|\leq Ce^{-cj}.\label{DSP_CV_Expo-}
			\end{align}
		\end{subequations}
		Up to considering a smaller interval $I$ for the family of SDSPs $(\dspD)_{\delta \in I}$, we will also assume that there exist two positive constants $C,c$ such that:
		\begin{equation}
			\forall \delta\in I,\forall j \in\Z,\quad \left|\dspD_j-\dsp_j\right| \leq C\left|\delta\right|e^{-c|j|}\label{DSP_CV_ExpoDelta}
		\end{equation}
		and that the set 
		\begin{equation}\label{set:CT}
			\left\{\dspD_j, \quad(\delta,j )\in I\times \Z \right\}
		\end{equation}
		is relatively compact in the space of states $\Uc$.
	\end{hypo}
	
	The inequality \eqref{DSP_CV_ExpoEtatLim} corresponds to an exponential localization of the transition zone of the sequence $\dsp$ between the endstates $u^+$ and $u^-$ while the inequality \eqref{DSP_CV_ExpoDelta} is essentially linked to some Lipschitz continuity property of the family of SDSPs $(\dspD)_{\delta \in I}$. We claim that those two inequalities might be obtained fairly similarly and should be a consequence of the shock being non-characteristic. Indeed, one can find proofs of similar properties on viscous approximations of Lax shocks (see \cite[Corollary 1.2]{Zumbrun1998}) or on semi-discrete approximations of moving Lax shocks (see \cite[Proposition 2.4]{BenzoniGavage2003} and \cite[Lemma 1.1]{Beck2010}). Let us point out that the exponential localization \eqref{DSP_CV_ExpoEtatLim} is central in \cite{Coeuret2024d} to study the spectrum of the operator $\Ldsp:=L^0$ defined below by \eqref{def:linearizedScheme} which corresponds to the linearization of the numerical scheme about the discrete shock profile $\dsp$. 
	
	With regards to the relative compactness of the set \eqref{set:CT}, we claim that it is actually a consequence of \eqref{DSP_CV_ExpoEtatLim} and \eqref{DSP_CV_ExpoDelta}. Indeed, since the space of states $\Uc$ is open and using \eqref{DSP_CV_ExpoEtatLim} and \eqref{DSP_CV_ExpoDelta}, there exist a radius $\Rayon>0$ and an integer $j_0\in\N$ such that:
	$$ \left\{\dspD_j, \quad (\delta,j)\in I\times \Z \text{ with } |j|> j_0\right\}\subset B(u^+,\Rayon)\cup B(u^-,\Rayon) \subset \Uc.$$
	One can then easily obtain the relative compactness of \eqref{set:CT} up to introducing a smaller open interval $I$ such that $0$ belongs to it and the sets $\left\{\dspD_j,\delta\in I\right\}$ are relatively compact in $\Uc$ for $j \in \left\{-j_0,\hdots, j_0\right\}$.

	We observe that the inequality \eqref{DSP_CV_ExpoDelta} allows us to define the real valued continuous "mass function" $M$ defined by:
	\begin{equation}\label{def:M}
		\forall \delta\in I,\quad M(\delta):= \sum_{j\in\Z} \dspD_j-\dsp_j.
	\end{equation}
	The \textit{mass} function $M$ corresponds to summing the differences between the elements of a discrete shock profile $\dspD$ and the reference discrete shock profile $\dsp:=\dsp^0$. Let us introduce a hypothesis related to the function $M$.	
	
	\begin{hypo}[Identification by mass]\label{H:identification}
		Up to considering a smaller open interval $I$ containing $0$ for the family of SDSPs $(\dspD)_{\delta \in I}$, we will assume that the mass function $M$ is injective and that there exists a constant $C_\delta>0$ such that:
		\begin{equation}\label{in:M}
			\forall \delta \in I,\quad |\delta|\leq C_\delta |M(\delta)|. 
		\end{equation}		
	\end{hypo}
	
	The injectivity condition we impose on the mass function $M$ implies that each discrete shock profile $\dspD$ for $\delta \in I$ is identified by its mass difference $M(\delta)$ with the reference profile $\dsp$. A similar result was already obtained in \cite[Theorem 2.1]{Smyrlis1990} for specific discrete shock profiles of the Lax-Wendroff scheme. Let us add that the mass function $M$ shares several links with the so-called $Y$ function introduced in \cite{Serre2007} and that the analysis of the $Y$ function of Serre tends to supports the claim of a linear behavior of the mass function $M$.
	
	The mass function $M$ will be central in the nonlinear orbital stability result Theorem \ref{Th} of the present article, as we will explain in the following paragraph. Let us point out that if we were able to prove that the mass function $M$ was of class $\Cc^1$ on $I$ and that we have:
	\begin{equation}\label{MPrimeneq0}
		M^\prime(0)\neq0,
	\end{equation}
	then we immediately obtain Hypothesis \ref{H:identification}. Up to having sufficient regularity on the family of discrete shock profiles $\left(\dspD\right)_{\delta\in I}$ and on the function $M$, we will show at the end of Section \ref{subsec:Green} that we can actually obtain \eqref{MPrimeneq0} and thus the verification of Hypothesis \ref{H:identification} using the result \cite[Theorem 1]{Coeuret2024d}. 
	
	\vspace{0.1cm}\noindent\textbf{\underline{Stability theory for discrete shock profiles and goal of the present article}}\vspace{0.1cm}

	In the following paragraph, we will present the stability theory for discrete shock profiles and where Theorem \ref{Th} of the present article stands with respect to the state of the art. We point out that, even though this article focuses on discrete shock profiles for \textit{scalar} conservation laws, a crucial point is that most of the discussion on the stability of discrete shock profiles that follows also holds for and tackles the case of \textit{systems} of conservation laws.
	
	\begin{figure}
		\begin{center}
			\includegraphics[width=0.33\textwidth]{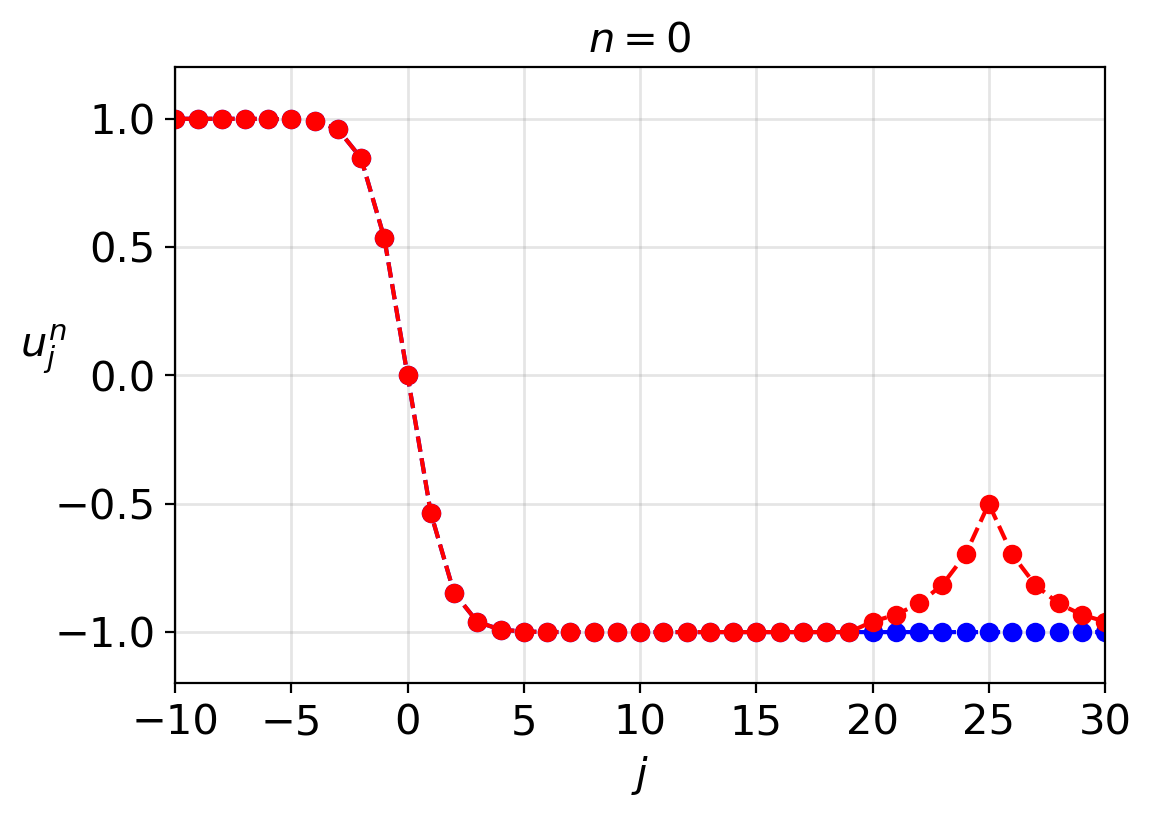} \includegraphics[width=0.33\textwidth]{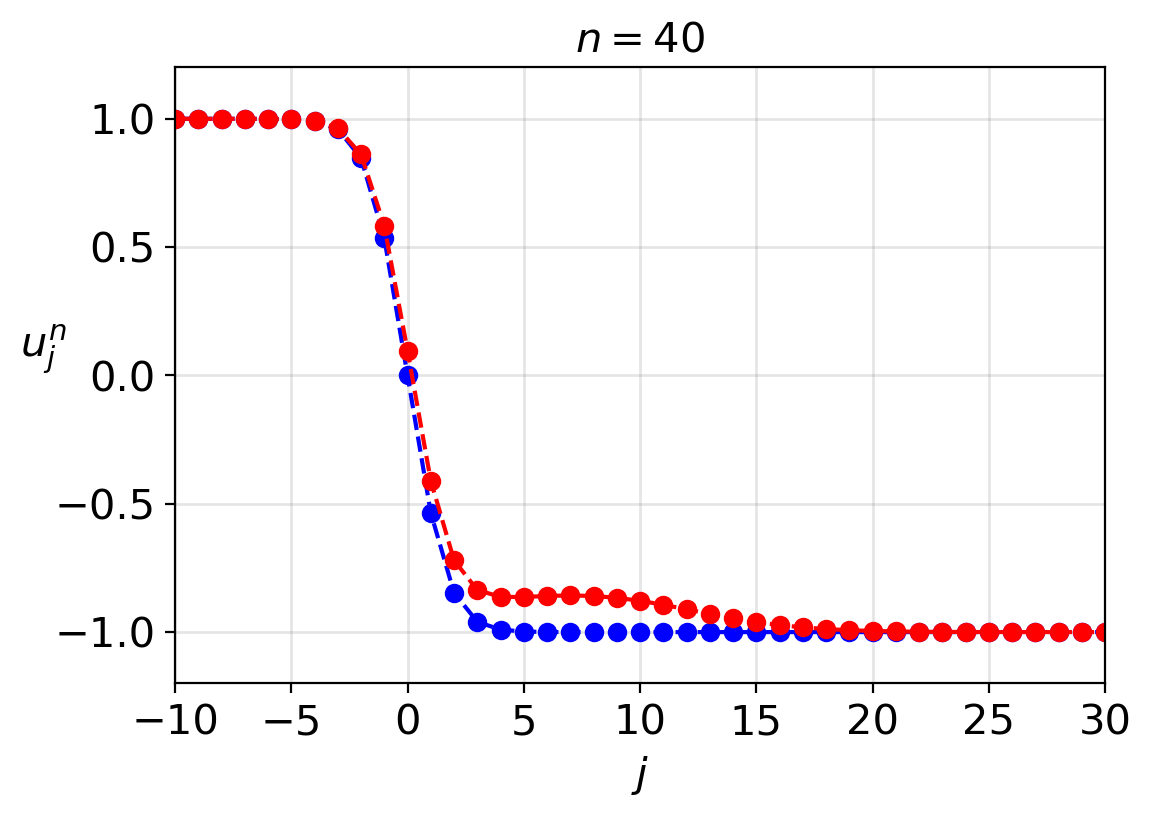}\includegraphics[width=0.33\textwidth]{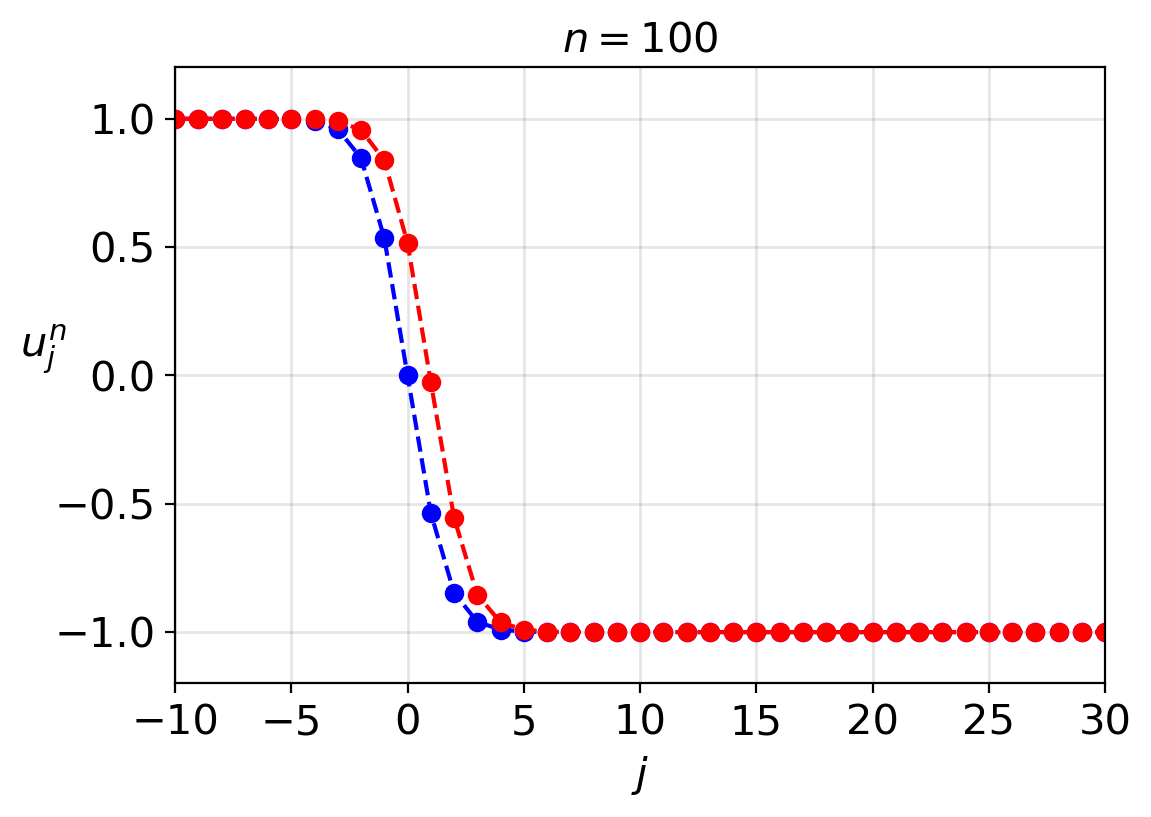}
		\end{center}
		\caption{An example of nonlinear orbital stability of a discrete shock profile. On the left figure, we represent in blue a discrete shock profile $\dsp$ and in red a perturbation of the discrete shock profile $u^0:=\dsp+\hg$. The figures in the middle and on the right represent the evolution in time of the solution $u^n$ of the numerical scheme \eqref{def:SchemeNum} with the initial condition $u^0$. We see that, in the long run, the sequence $(u^n)_{n\in\N}$ converges towards some other discrete shock profile $\dspD$ associated with the same shock. More precisely, the solution $u^n$ seems to converge towards the member of the family $(\dspD)_{\delta\in I}$ defined by the equality \eqref{IDMasseFinal} since the sum of the differences between the elements of the red and blue sequences is conserved in time.}
		\label{fig:perturb}
	\end{figure}
	
	The main subject of this article handles \textit{nonlinear orbital stability} of stationary discrete shock profiles. It amounts to finding two suitable Banach spaces $X$ and $Y$ (for instance weighted $\ell^r$-spaces) such that we can obtain a result of the following form, where the continuous one-parameter family of discrete shock profiles $(\dspD)_{\delta\in I}$ associated with $\dsp$ appears:
	\begin{center}
		\noindent\fbox{\begin{minipage}{0.93\textwidth}
				\noindent There exists a positive constant $\varepsilon>0$ such that for an initial pertubation $\hg\in X$ such that:
				$$\left\|\hg\right\|_X<\varepsilon, $$
				then the solution $(u^n)_{n\in\N}$ of the numerical scheme \eqref{def:SchemeNum} associated with the initial condition $u^0:=\overline{u}+\hg$ is defined for all time $n\in\N$ and we have:
				\begin{equation}\label{StabNonlin}
					\inf_{\delta\in I}\left\|u^n-\overline{u}^\delta\right\|_{Y}\underset{n\rightarrow +\infty}\rightarrow  0.
				\end{equation}
		\end{minipage}}
	\end{center}
	Figure \ref{fig:perturb} gives a representation of this result. Essentially, the nonlinear orbital stability states that, for small perturbations $u^0:= \dsp+\hg$ of a discrete shock profile $\dsp$, the solution $(u^n)_{n\in\N}$ of the numerical scheme \eqref{def:SchemeNum} converges towards the family $\left(\overline{u}^\delta\right)_{\delta\in I}$ of discrete shock profiles associated with the initial discrete shock profile $\dsp$. Let us make some important observations. 
	
	First, in the statement of nonlinear orbital stability results, one can actually hope that the solution $(u^n)_{n\in\N}$ converges towards a specific element of the family $\left(\overline{u}^\delta\right)_{\delta\in I}$, i.e. one might prove the following assertion rather than \eqref{StabNonlin}:
	\begin{equation}\label{StabNonlin_EtatFin}
		\exists \delta \in I,\quad \left\|u^n-\dspD\right\|_{Y}\underset{n\rightarrow +\infty}\rightarrow  0.
	\end{equation}
	Furthermore, due to the conservative nature of the numerical scheme \eqref{def:SchemeNum}, we claim that the solution $(u^n)_{n\in\N}$ of the numerical scheme \eqref{def:SchemeNum} for an initial condition $u^0:=\dsp+\hg$ satisfies:
	\begin{equation}\label{IDMasseTemp}
		\forall n\in\N, \quad \sum_{j\in\Z}u^n_j-\dsp_j =  \sum_{j\in\Z}u^0_j-\dsp_j= \sum_{j\in\Z}\hg_j.
	\end{equation}
	As a consequence, if the solution $(u^n)_{n\in\N}$ of the numerical scheme converges for instance in the $\ell^1$-norm towards a specific discrete shock profile $\dspD$ with $\delta \in I$, then \eqref{IDMasseTemp} implies that:
	\begin{equation}\label{IDMasseFinal}
		M(\delta)=\sum_{j\in\Z}\hg_j
	\end{equation}
	where the mass function $M$ is defined by \eqref{def:M}. Since Hypothesis \ref{H:identification} implies that the function $M$ is injective, there exists a unique choice of $\delta\in I$ such that the solution $(u^n)_{n\in\N}$ of the numerical scheme \eqref{def:SchemeNum} with the initial condition $u^0:=\dsp+\hg$ converges towards $\dspD$ and the choice of $\delta$ is determined by \eqref{IDMasseFinal}. We observe that Hypothesis \ref{H:identification} is therefore crucial. Indeed, if the function $M$ was not injective, then one could not identify a unique discrete shock profile towards which the solution $(u^n)_{n\in\N}$ should converge and proving a nonlinear orbital stability result would be difficult, if not impossible. We recall that we will discuss on the verification of Hypothesis \ref{H:identification} at the end of Section \ref{subsec:Green}.
	
	Let us make a second observation. Within the definition of the nonlinear stability results, it is not clear and easy to prove that the solution $(u^n)_{n\in\N}$ of the numerical scheme \eqref{def:SchemeNum} associated with an initial condition $u^0:=\dsp+\hg$ is defined for all time $n\in\N$. Indeed, it is possible that the solution $u^n$ leaves the domain of definition $\Uc^\Z$ of the nonlinear evolution operator $\Nc$ defined by \eqref{def:evolOpe}, i.e. there can be a time $n\in\N$ and an integer $j\in\Z$ such that $u^n_j$ leaves the space of states $\Uc$. This is also a important point to prove.
	
	Several articles have proved nonlinear orbital stability results for discrete shock profiles associated with stationary (and moving) Lax shocks (see \cite{Jennings1974,Smyrlis1990,Liu1993a,Liu1993,Ying1997,Liu1999a,Michelson2002}). However, those results have recurring limitations that one would want to avoid:
	\begin{itemize}
		\item Some stability results are strictly restricted for discrete shock profiles of \textit{scalar} conservation laws. It is for instance the case of \cite{Jennings1974} which uses intensively a \textit{monotonicity} assumption on the choice of numerical schemes studied, which can quite obviously not be generalized to the case of \textit{systems} of conservation laws. However, those results tend to be less subject to the following limitations.
		\item In the case of \textit{systems} of conservation laws, known stability results for discrete shock profiles impose a weakness assumption on the amplitude of the shock associated with the discrete shock profiles (see for instance \cite{Liu1993a,Liu1993,Ying1997,Liu1999a,Michelson2002}).
		\item Several result are applied to fairly restrictive families of numerical schemes, for instance, \textit{monotone} conservative finite difference schemes \cite{Jennings1974}, the Lax-Wendroff scheme for \cite{Smyrlis1990}  and the Lax-Friedrichs scheme in \cite{Liu1993a,Liu1993,Ying1997}.  
		\item Some results impose strong restrictions on the initial perturbations $\hg$ considered (i.e. on the choice of the vector space $X$ in the statement of the nonlinear stability result above). For instance, most result consider the initial perturbation $\hg$ to be small in some weighted $\ell^r$-spaces. However, there can be significant differences in the choices of the weights. For instance, the initial perturbation in \cite[Theorem 4.1]{Smyrlis1990} belong to an exponentially-weighted $\ell^2$-space whereas the main results of \cite{Liu1993a,Liu1993} only consider small initial perturbations in a polynomially-weighted $\ell^2$-space. Another example of restriction that can be imposed is a zero-mass assumption on the initial perturbation $\hg$. We observe in this context that the identification equation \eqref{IDMasseFinal} implies that the solution $(u^n)_{n\in\N}$ should then converge towards the initial discrete shock profile $\dsp:=\dsp^0$ since $M(0)=0$. This is for instance the case in \cite{Liu1993a,Liu1993}, though those results were generalized for nonzero-mass perturbations in \cite{Ying1997}.
		\item The results presented previously do not have any decay rates for the convergence \eqref{StabNonlin_EtatFin}, which could be a desirable feature.
	\end{itemize}
	
	One would want to prove a nonlinear stability result that discards or replaces several limitations of the previously presented results. Most importantly, we wish to avoid the introduction of a weakness assumption. An option in this direction, which is presented in \cite[Open Question 5.3]{Serre2007}, would be to generalize in the case of discrete shock profiles the stability analysis initiated for viscous and relaxation approximations of shocks in \cite{Zumbrun1998,Zumbrun2000,Mascia2002,Mascia2003}. The main idea would be essentially to prove that the \textit{spectral stability} of a discrete shock profile implies its nonlinear orbital stability. This \textit{spectral stability} assumption corresponds to a hypothesis on the point spectrum of the linearization of the numerical scheme $\Nc$ about a discrete shock profile. It will be defined more precisely below in Hypotheses \ref{H:spec} and \ref{H:Evans}. Considering this spectral stability assumption avoids the introduction of a weakness assumption on the strength of the approximated shocks. Furthermore, the analysis can be carried for a large family of numerical schemes contrarily to many known results. The stability analysis described in \cite{Zumbrun1998,Zumbrun2000,Mascia2002,Mascia2003} was already generalized for semi-discrete approximations of shocks in \cite{BenzoniGavage2003,Beck2010}. Up to the author's knowledge, the combination of \cite{Coeuret2024d} and the present article correspond to the first complete generalization of these techniques for discrete shock profiles, i.e. in a fully discrete setting, to conclude a nonlinear orbital stability result (Theorem \ref{Th}), at least for scalar conservation laws, and thus offer a partial answer to the \cite[Open Question 5.3]{Serre2007}. 
	
	\subsection{Linearizations of the numerical scheme about discrete shock profiles and spectral stability}\label{subsec:Lin}
	
	One of the central element that we will use in the present article is an in-depth comprehension of the linearization of the numerical scheme about the discrete shock profile $\dsp$. This was the main subject of a previous paper \cite{Coeuret2024d} by the author. We will start by introducing the linearization $\LdspD$ about the discrete shock profile $\dspD$, discuss about the spectral properties of the operator $\Ldsp:=L^0$ and finally introduce the so-called \textit{spectral stability} assumption that is imposed on the point spectrum of the operator $\Ldsp$ (see Hypotheses \ref{H:spec} and \ref{H:Evans} below).
	
	First, for $\delta\in I$, we linearize the discrete evolution operator $\Nc$ about the discrete shock profile $\dspD$ and thus introduce the bounded \textit{linear} operator $\LdspD$ acting on $\ell^r(\Z)$ with $r\in[1,+\infty]$ defined by:
	\begin{equation}\label{def:linearizedScheme}
		\forall h\in \ell^r(\Z),\forall j\in \Z, \quad (\LdspD h)_j := \sum_{k=-p}^{q}a^\delta_{j,k}h_{j+k},
	\end{equation}
	where for $j\in\Z$ and $k\in\lbrace-p,\hdots,q-1\rbrace$, we first define the scalar:
	\begin{equation}\label{def:Bjk}
		b^\delta_{j,k}:=\displaystyle\nug \partial_{u_k}F\left(\nug;\dspD_{j-p},\hdots,\dspD_{j+q-1}\right)\in\C 
	\end{equation}
	and for $j\in\Z$ and $k\in\lbrace-p,\hdots,q\rbrace$:
	\begin{equation}\label{def:Ajk}
		a^\delta_{j,k}:=
		\left\{\begin{array}{cc}
			-b^\delta_{j+1,q-1} & \text{ if }k=q,\\
			b^\delta_{j,-p} & \text{ if }k=-p,\\
			\delta_{k,0} +b^\delta_{j,k} - b^\delta_{j+1,k-1} & \text{ otherwise,}
		\end{array}\right.
	\end{equation}
	and $\delta_k:=\left(\delta_{k,0}\right)$ is the Dirac mass sequence equal to $1$ at the index $k$ and $0$ elsewhere. 
	
	Though the linear operators $\LdspD$ will appear later on in the article, the one that will be more important to us is the operator $L^0$ which corresponds to the linearization of the numerical scheme about the discrete shock profile $\dsp:= \dsp^0$. Just like we used the notation $\dsp$ to denote the discrete shock profile $\overline{u}^0$, we introduce the notation: 
	\begin{equation}\label{def:Ldsp}
		\Ldsp:=L^0
	\end{equation}
	to describe the linearization of the numerical scheme about $\dsp$. As claimed previously, we need to study the spectrum of the operator $\Ldsp$ and also introduce the so-called \textit{spectral stability} assumption. This corresponds to one of the main parts of the article \cite{Coeuret2024d} and is heavily inspired by \cite{LafitteGodillon2001,Godillon2003,Serre2007}. We will now recall the main spectral results.	
	
	For $k\in\lbrace-p,\hdots,q-1\rbrace$, we define the scalars:
	\begin{equation}\label{def:Bk}
		b_k^\pm:=\displaystyle\nug \partial_{u_k}F\left(\nug;u^\pm,\hdots,u^\pm\right) \in\C
	\end{equation}
	and then for $k\in\lbrace-p,\hdots,q\rbrace$, we let:
	\begin{equation}\label{def:Ak}
		a_k^\pm:=
		\left\{\begin{array}{cc}
			-b_{q-1}^\pm & \text{ if }k=q,\\
			b_{-p}^\pm & \text{ if }k=-p,\\
			\delta_{k,0} +b_k^\pm - b_{k-1}^\pm & \text{ else.}
		\end{array}\right.
	\end{equation}
	We then define the meromorphic function $\Fc^\pm$ on $\C\backslash\lbrace0\rbrace$ by:
	\begin{equation}\label{def:Fc}
		\forall \kappa\in\C\backslash\lbrace0\rbrace, \quad \Fc^\pm(\kappa):=\sum_{k=-p}^q a_k^\pm\kappa^k\in\C.
	\end{equation}
	The definition \eqref{def:Ak} of the scalars $a_k^\pm$ and the consistency condition \eqref{cond:consistency} imply that
	\begin{equation}\label{eg:FcFin}
		\Fc^\pm(1)=1 \quad \text{and} \quad \alpha^\pm:=-{\Fc^\pm}^\prime(1)=\nug f^\prime(u^\pm)\neq 0.
	\end{equation}
	We additionally consider that the following assumption is verified, which corresponds to \cite[Hypothesis 5]{Coeuret2024d} and where the complex unit circle is denoted $\S^1$.
	\begin{hypo}\label{H:F}
		We have:
		$$\forall \kappa\in \S^1\backslash \lbrace1\rbrace, \quad |\Fc^\pm(\kappa)|<1.\quad \text{(Dissipativity condition)}$$
		Moreover, we suppose that there exists an integer $\mu\in \N\backslash\lbrace0\rbrace$ and a complex number $\beta^\pm$ with positive real part such that:
		\begin{equation}\label{F}
			\Fc^\pm(e^{i\xi})\underset{\xi\rightarrow0}= \exp(-i\alpha^\pm\xi - \beta^\pm \xi^{2\mu} + O(|\xi|^{2\mu+1})).\quad \text{(Diffusivity condition)}
		\end{equation}
	\end{hypo}
	
	This assumption is inspired by the fundamental contribution \cite{Thomee1965} which studies the $\ell^\infty$-stability\footnote{The $\ell^r$-stability defined here correponds to the power boundedness of the linearization about constant states when it acts on $\ell^r$.} of finite difference schemes for the transport equation. We first observe that the dissipativity condition implies, via the well-known Von Neumann condition, the $\ell^2$-stability of the linearization of the numerical scheme $\Nc$ about the constant states $u^\pm$. Adding the diffusivity condition \eqref{F} implies that the linearization about the constant states $u^\pm$ is also $\ell^r$-stable for all $r\in[1,+\infty]$. The numerical schemes that satisfy the diffusivity condition correspond to the numerical scheme which introduce some numerical viscosity, i.e. odd ordered schemes. Even ordered schemes, like the Lax-Wendroff scheme, do not satisfy this condition.
	
	One of the main result of \cite{Godillon2003,Serre2007,Coeuret2024d} is that the elements of the complex plane which belong to the unbounded connected component $\Oc$ of $\C\backslash\left(\Fc^+(\S^1)\cup\Fc^-(\S^1)\right)$ are either in the resolvent set of the operator $\Ldsp$ or are eigenvalues of $\Ldsp$ (see Figure \ref{fig:spec}). 
	
	\begin{figure}
		\begin{center}
			\begin{tikzpicture}[scale=2]
				\fill[color=gray!20] (-1.5,-1.5) -- (-1.5,1.5) -- (1.5,1.5) -- (1.5,-1.5) -- cycle;
				\fill[color=white] plot [samples = 100, domain=0:2*pi] ({1/4+3*cos(\x r)/4},{sin(\x r)/3}) -- cycle ;
				\fill[color=white] plot [samples = 100, domain=0:2*pi] ({1/2.5+1.5*cos(\x r)/2.5},{sin(\x r)/2}) -- cycle ;
				
				\draw (0,0) circle (1);
				
				\draw[thick,red] plot [samples = 100, domain=0:2*pi] ({(1/4+3*cos(\x r)/4},{sin(\x r)/3});
				\draw[thick,red] plot [samples = 100, domain=0:2*pi] ({1/2.5+1.5*cos(\x r)/2.5},{sin(\x r)/2});
				\draw[blue] (-0.7,0.2) node {$\bullet$};
				\draw[blue](0.5,0.6) node {$\bullet$};
				\draw[blue] (-0.1,-0.7) node {$\bullet$};
				
				\draw (45:1.2) node {$\S^1$};
				\draw[color=black!70] (-1,1.2) node {$\Oc$};
				\draw[red] (-0.2,0.55) node {$\Fc^\pm(\S^1)$};
				\draw[blue](1,0) node {$\bullet$} node[right] {$1$};
			\end{tikzpicture}
			\caption{In red, we have the curves $\Fc^\pm(\S^1)$. In gray, we represent the set $\Oc$ which corresponds to the unbounded component of $\C\backslash(\Fc^+(\S^1)\cup\Fc^-(\S^1))$. The elements of the set $\Oc$ are either eigenvalues of the operator $\Ldsp$ (represented in blue) or belong to the resolvent set $\rho(\Ldsp)$. We know that $1$ is an eigenvalue of $\Ldsp$. The first part of the spectral stability assumption described in Hypothesis \ref{H:spec} implies that the eigenvalues of $\Ldsp$ in $\Oc$ are located within the open unit disk.}
			\label{fig:spec}
		\end{center}
	\end{figure}
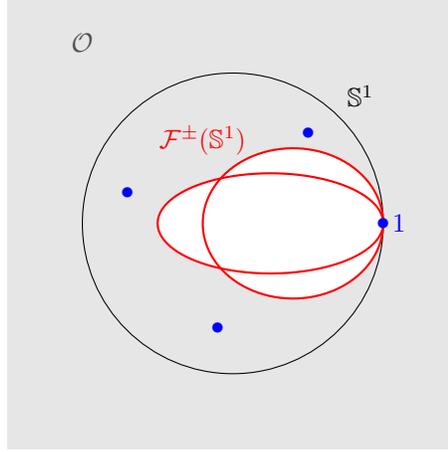
	
	We introduce the first part of the \textit{spectral stability} assumption, which corresponds to \cite[Hypothesis 6]{Coeuret2024d}.
	\begin{hypo}\label{H:spec}
		The operator $\Ldsp$ has no eigenvalue of modulus equal or larger than $1$ other than $1$.
	\end{hypo}
	Combining Hypothesis \ref{H:spec} with the statements above on the spectrum $\sigma(\Ldsp)$, we can then conclude that every complex scalar different than $1$ and of modulus equal or larger than $1$ is included in the resolvent set $\rho(\Ldsp)$. 
	
	The complex scalar $1$ plays a particular role in the spectral analysis of the operator $\Ldsp$. First, we have that $1$ belongs to the curves $\Fc^\pm(\S^1)$ and thus to the essential spectrum of the operator $\Ldsp$. However, the situation is even more complicated. In \cite{Godillon2003,Serre2007,Coeuret2024d}, we can find the construction of a so-called Evans function noted $\mathrm{Ev}$ in a neighborhood of $1$. This Evans function $\mathrm{Ev}$ is a holomorphic function that vanishes at eigenvalues of the operator $\Ldsp$, i.e. it acts like a characteristic polynomial. The articles  \cite{Godillon2003,Serre2007,Coeuret2024d} also prove that the Evans function $\mathrm{Ev}$ vanishes at $1$, i.e. $1$ is an eigenvalue of the operator $\Ldsp$. This is essentially the consequence of the existence of a continuous one-parameter family of SDSPs $(\dspD)_{\delta\in I}$ associated with $\dsp$.
	
	We now introduce the second part of the \textit{spectral stability} assumption, which is linked with the Evans function $\mathrm{Ev}$. It corresponds to \cite[Hypothesis 7]{Coeuret2024d}.
	\begin{hypo}\label{H:Evans}
		We have that $1$ is a simple zero of the Evans function $\mathrm{Ev}$, i.e.
		$$\mathrm{Ev}(1)=0 \quad \text{ and } \quad \frac{\partial \mathrm{Ev}}{\partial z}(1)\neq 0.$$
	\end{hypo}
	A consequence of Hypothesis \ref{H:Evans} proved in \cite{Coeuret2024d} is that the eigenspace of the operator $\Ldsp$ associated with the eigenvalue $1$ is of dimension $1$. More precisely, there exists a sequence $V\in \ell^2(\Z)\backslash\lbrace0\rbrace$ and two positive constants $C,c$ such that:
	\begin{equation}\label{decExpoV}
		\forall j\in \Z, \quad |V(j)|\leq Ce^{-c|j|}
	\end{equation}
	and:
	\begin{equation}\label{egV}
		\ker(Id_{\ell^2}-\Ldsp)=\mathrm{Span} V.
	\end{equation}
	This sequence $V$ will come back below in the description of the Green's function of the operator $\Ldsp$. Let us make an additional interesting observation. If we assume that the one-parameter family of SDSPs $\left(\dspD\right)_{\delta\in I}$ introduced in Hypothesis \ref{H:SDSP} is differentiable, then we claim that the sequence $\frac{\partial \dspD}{\partial\delta}\big|_{\delta=0}$ would belong to the eigenspace associated with the eigenvalue $1$ of the operator $\Ldsp$, which implies that the sequence $V$ would be collinear with the sequence $\frac{\partial \dspD}{\partial\delta}\big|_{\delta=0}$. This fact will be used later on at the end of Section \ref{subsec:Green} when discussing about the identification by mass of the discrete shock profiles $\dspD$ introduced in Hypothesis \ref{H:identification}.
	
	\vspace{0.1cm}\noindent\textbf{\underline{Additional technical assumptions}}\vspace{0.1cm}
	
	In the present article, the main result \cite[Thereom 1]{Coeuret2024d} which provides a precise description of the Green's function of the operator $\Ldsp$ will play a central role. Several of the hypotheses we introduced before on the discrete shock profile $\dsp$ are also part of the article \cite{Coeuret2024d}. However, there are a few additional technical assumptions introduced in \cite{Coeuret2024d} which will not be central in the understanding of the present article but are necessary to use \cite[Thereom 1]{Coeuret2024d}.  We will introduce them here without much context and invite the interested reader to take a look at \cite{Coeuret2024d} for more details on their role.
	
	\begin{itemize}
		\item \cite[Hypothesis 4]{Coeuret2024d} is always verified in the case of scalar conservation laws.
		
		\item We consider that the following assumption, which corresponds to \cite[Hypothesis 8]{Coeuret2024d}, is verified.
		\begin{hypo}\label{H:inv}
			The scalars $a^0_{j,-p}=b^0_{j,-p}$ and $a^0_{j,q}=-b^0_{j+1,q-1}$  for all $j\in\Z$ as well as the scalars $a_{-p}^\pm=b_{-p}^\pm$ and $a_q^\pm=-b_{q-1}^\pm$ are all different from zero.
		\end{hypo}
		
		\item We consider that the following assumption, which corresponds to \cite[Hypothesis 9]{Coeuret2024d}, is verified.
		\begin{hypo}\label{H:Mpm1}
			The equation 
			\begin{equation}\label{eg:Fcl}
				\Fc^\pm(\kappa)=1
			\end{equation}
			has $p+q$ distinct solutions $\kappa\in \C\backslash\lbrace0\rbrace$.
		\end{hypo}
	\end{itemize}
	
	\subsection{Green's function and linear stability estimates}\label{subsec:Green}
	
	Now that we have introduced the spectral stability assumption on the operator $\Ldsp$, our goal is to translate those spectral information into decay estimates the semigroup $(\Ldsp^n)_{n\in\N}$ as well as for other related families of operators. This will be achieved below in Proposition \ref{prop:Est}. Those estimates will be central in the proof of Theorem \ref{Th}. The proof of Proposition \ref{prop:Est} relies on a precise pointwise description of the Green's function associated with the operator $\Ldsp$.
	
	\vspace{0.1cm}\noindent\textbf{\underline{Definition of the Green's function of $\Ldsp$ and statement of \cite[Theorem 1]{Coeuret2024d}}}\vspace{0.1cm}
	
	For $j_0\in\Z$, we define the Green's function of the operator $\Ldsp$ recursively as:
	\begin{equation}\label{def:GreenTempo}
		\begin{array}{cc}
			&\Gcc(0,j_0,\cdot) : = \delta_{j_0},\\
			\forall n\in\N, & \Gcc(n+1,j_0,\cdot) : = \Ldsp  \Gcc(n,j_0,\cdot) ,
		\end{array}
	\end{equation}
	where the Dirac mass $\delta_{j_0}$ is a sequence defined by:
	$$\delta_{j_0}:=\left(\delta_{j_0,j}\right)_{j\in\Z}\in\ell^2(\Z).$$
	The main consequence of the introduction of the Green's function is that for all $h\in \ell^r(\Z)$ with $r\in[1,+\infty]$, we have:
	\begin{equation}\label{lienLccGreenV1}
		\forall n\in\N,\forall j\in\Z,\quad (\Ldsp^nh)_j = \sum_{j_0\in\Z}\Gcc(n,j_0,j)h_{j_0}.
	\end{equation}
	Thus, a precise description of the Green's function is sufficient to understand the action of the semigroup $(\Ldsp^n)_{n\in\N}$ associated with the operator $\Ldsp$.
	
	The following lemma proved via a simple recurrence is a direct consequence of the definition \eqref{def:GreenTempo} of the Green's function and the finite speed propagation of the linearized scheme.
	\begin{lemma}\label{lemGreenTempo}
		For $n\in \N$ and $j_0\in \Z$, the Green's function $\Gcc(n,j_0,\cdot)$ is finitely supported. More precisely, for $j\in\Z$, we have that:
		$$j-j_0\notin \lbrace-nq,\hdots,np\rbrace\quad \Rightarrow \quad\Gcc(n,j_0,j)=0.$$
	\end{lemma}
	The goal of \cite{Coeuret2024d} was to describe the pointwise asymptotic behavior of the Green's function $\Gcc(n,j_0,j)$ when $j-j_0\in\lbrace-nq,\hdots,np\rbrace$. To present \cite[Theorem 1]{Coeuret2024d}, for $\beta \in \C$ with positive real part, we define the functions $H_{2\mu}(\beta;\cdot),E_{2\mu}(\beta;\cdot): \R\rightarrow \C$ by:
	\begin{align}\label{def:H2mu_et_E2mu}
		\begin{split}
			\forall x \in \R,\quad &H_{2\mu}(\beta;x) := \frac{1}{2\pi} \int_\R e^{ixu}e^{-\beta u^{2\mu}}du,\\
			\forall x \in \R,\quad &E_{2\mu}(\beta;x) := \int_x^{+\infty} H_{2\mu}(\beta;y)dy,
		\end{split}
	\end{align}
	where we recall that the integer $\mu$ is defined in Hypothesis \ref{H:F}. Let us point out that Lemma \ref{lemHE} below implies that the function $E_{2\mu}$ is well-defined. We call the functions $H_{2\mu}$ generalized Gaussians and the functions $E_{2\mu}$ generalized Gaussian error functions since for $\mu=1$, we have 
	$$\forall x \in \R,\quad H_{2}(\beta;x)=\frac{1}{\sqrt{4\pi\beta}}e^{-\frac{x^2}{4\beta}}.$$
	Noticing that the function $H_{2\mu}$ is an even function and that it is the inverse Fourier transform of $u\mapsto e^{-\beta u^{2\mu}}$, we observe that:
	\begin{subequations}
		\begin{align}
			\lim_{x\rightarrow -\infty}E_{2\mu}(\beta;x)=\int_{-\infty}^{+\infty} H_{2\mu}(\beta;y)dy= 1 \label{eq:E_en_-infty}\\
			\forall x\in \R, \quad E_{2\mu}(\beta,-x)= 1-E_{2\mu}(\beta,x).
		\end{align}
	\end{subequations}
	
	The following lemma introduces some useful inequalities on the functions $H_{2\mu}$ and $E_{2\mu}$ defined by \eqref{def:H2mu_et_E2mu}. We refer to \cite[Lemma 1.2]{Coeuret2024d} for the proof.
	\begin{lemma}\label{lemHE}
		Let us consider a compact subset $A$ of $\lbrace z\in\C,\Re(z)>0\rbrace$ and integers $\mu,m\in\N\backslash\lbrace0\rbrace$. There exist two positive constants $C,c$ such that for all $\beta\in A$:
		\begin{subequations}
			\begin{align}
				\forall x\in\R,\quad & |\partial_x^mH_{2\mu}(\beta;x)|\leq C\exp(-c|x|^\frac{2\mu}{2\mu-1}),\label{inH}\\
				\forall x\in]0,+\infty[,\quad & |E_{2\mu}(\beta;x)|\leq C\exp(-c|x|^\frac{2\mu}{2\mu-1}),\label{inE+}\\
				\forall x\in]-\infty,0[, \quad& |1-E_{2\mu}(\beta;x)|\leq C\exp(-c|x|^\frac{2\mu}{2\mu-1}).\label{inE-}
			\end{align}
		\end{subequations}
	\end{lemma}
	
	Hypotheses \ref{H:SDSP}-\ref{H:Mpm1} (though the verification of Hypothesis \ref{H:identification} is not necessary) allow us to use the results of \cite{Coeuret2024d}. First, the \textit{Liu-Majda condition} \cite[Lemma 4.6]{Coeuret2024d} implies in the scalar case that:
	\begin{equation}\label{Liu-Majda_condition}
		\sum_{j\in\Z}V(j)\neq0
	\end{equation}
	where the sequence $V$ is defined by \eqref{egV}. Furthermore, \cite[Theorem 1]{Coeuret2024d}  states that there exists some small constant $c>0$ such that for $j_0\in\N$, $n\in\N\backslash\lbrace0\rbrace$ and $j\in\Z$ verifying $j-j_0\in\lbrace-nq,\hdots,np\rbrace$, we have:
	\begin{subequations}
		\begin{multline}\label{decompoGreen}
			\Gcc(n,j_0,j) = C_EE_{2\mu}\left(\beta^+;\frac{n\alpha^+ +j_0}{n^\frac{1}{2\mu}}\right)V(j)+\ind_{j\geq0}O\left(\frac{1}{n^\frac{1}{2\mu}}\exp\left(-c\left(\frac{\left|n-\left(\frac{j-j_0}{\alpha^+}\right)\right|}{n^\frac{1}{2\mu}}\right)^\frac{2\mu}{2\mu-1}\right)\right)\\+O\left(\frac{e^{-c|j|}}{n^\frac{1}{2\mu}}\exp\left(-c\left(\frac{\left|n+\frac{j_0}{\alpha^+}\right|}{n^\frac{1}{2\mu}}\right)^\frac{2\mu}{2\mu-1}\right)\right)+O(e^{-cn})
		\end{multline}
		and for $j_0\in \N\backslash\left\{0\right\}$, $n\in\N\backslash\lbrace0\rbrace$ and $j\in\Z$ verifying $j-j_0\in\lbrace-nq-1,\hdots,np\rbrace$, we have:
		\begin{multline}\label{decompoDerGreen}
			\Gcc(n,j_0,j) - \Gcc(n,j_0-1,j) \\=\ind_{j\geq0}\exp\left(-c\left(\frac{\left|n-\left(\frac{j-j_0}{\alpha^+}\right)\right|}{n^\frac{1}{2\mu}}\right)^\frac{2\mu}{2\mu-1}\right)\left(O\left(\frac{e^{-c|j|}}{n^\frac{1}{2\mu}}\right)+O\left(\frac{e^{-c|j_0|}}{n^\frac{1}{2\mu}}\right)+O\left(\frac{1}{n^\frac{1}{\mu}}\right)\right)\\+O\left(\frac{e^{-c|j|}}{n^\frac{1}{2\mu}}\exp\left(-c\left(\frac{\left|n+\frac{j_0}{\alpha^+}\right|}{n^\frac{1}{2\mu}}\right)^\frac{2\mu}{2\mu-1}\right)\right)+O(e^{-cn}).
		\end{multline}
	\end{subequations}
	where the constant $C_E$ is defined by\footnote{We clarify that this expression of the constant $C_E$ is an immediate consequence of the analysis of the mass of the Green's function \cite[(5.39)]{Coeuret2024d} in the scalar case.}:
	\begin{equation*}
		C_E=\frac{1}{\sum_{j\in\Z} V(j)}
	\end{equation*}
	and we recall that we introduced the usual Landau notation $O(\cdot)$. These pointwise descriptions of the Green's function and its discrete derivative have equivalent expressions when the location $j_0$ of the initial Dirac mass for the Green's function is negative.
	
	Let give a brief description on the behavior of the Green's function $\Gcc(n,j_0,j)$ for $j_0$ positive based on the expression \eqref{decompoGreen}. A similar description holds for the Green's function when $j_0$ is negative, as well as for the discrete derivative $\Gcc(n,j_0,j)-\Gcc(n,j_0-1,j)$ of the Green's function using \eqref{decompoDerGreen}. We also refer the interested reader to \cite{Coeuret2024d} for more details and numerical representations of the Green's function.
	
	For $j_0\in \N$, we recall that the definition of the Green's function \eqref{def:GreenTempo} implies that at the initial time $n=0$, the Green's function is just a Dirac mass localized at $j_0$, essentially on the right-side of the shock, assuming that the shock is located $0$. The expression \eqref{decompoGreen} describing the long-time behavior of the Green's function is composed of four terms, the last one being just a small remainder. 
	\begin{itemize}
		\item The second term
		$$\ind_{j\geq0}O\left(\frac{1}{n^\frac{1}{2\mu}}\exp\left(-c\left(\frac{\left|n-\left(\frac{j-j_0}{\alpha^+}\right)\right|}{n^\frac{1}{2\mu}}\right)^\frac{2\mu}{2\mu-1}\right)\right)$$
		describes a generalized Gaussian wave originating from $j_0$ and traveling at the speed $\alpha^+:=\nug f^\prime(u^\pm)<0$. It moves along the characteristic of the right state $u^+$, towards the shock, and disappears when it reaches it.
		\item The first term
		$$C_EE_{2\mu}\left(\beta^+;\frac{n\alpha^+ +j_0}{n^\frac{1}{2\mu}}\right)V(j)$$
		describes the progressive activation of the exponential profile $V$ that spans the eigenspace $\ker(Id_{\ell^2}-\Ldsp)$. The activation of the profile $V$ happens when the generalized Gaussian wave discussed above reaches the shock.
		\item Finally, the third term
		$$O\left(\frac{e^{-c|j|}}{n^\frac{1}{2\mu}}\exp\left(-c\left(\frac{\left|n+\frac{j_0}{\alpha^+}\right|}{n^\frac{1}{2\mu}}\right)^\frac{2\mu}{2\mu-1}\right)\right)$$
		is a remainder term which describes an additional exponential profile that is activated when the generalized Gaussian wave reaches the shock and then disappears.
	\end{itemize} 
	
	\vspace{0.1cm}\noindent\textbf{\underline{Decay estimates for the families of operators $\left(\Ldsp^n\right)_{n\in\N}$, $\left(\Ldsp^n(Id-\Tc)\right)_{n\in\N}$ and $\left(\Ldsp^n(\LdspD-\Ldsp)\right)_{n\in\N}$}}\vspace{0.1cm}
	
	The decompositions \eqref{decompoGreen} and \eqref{decompoDerGreen} are fundamental to prove the following proposition which states sharp bounds on the families of operators $\left(\Ldsp^n\right)_{n\in\N}$, $\left(\Ldsp^n(Id-\Tc)\right)_{n\in\N}$ as well as $\left(\Ldsp^n(\LdspD-\Ldsp)\right)_{n\in\N}$ acting on polynomially weighted $\ell^r$-spaces, where the shift operator $\Tc$ is defined by:
	\begin{equation}\label{def:Tc}
		\Tc: (h_j)_{j\in\Z}\in\C^\Z\mapsto (h_{j+1})_{j\in\Z}\in\C^\Z
	\end{equation}
	and the operator $\LdspD$ defined by \eqref{def:linearizedScheme} corresponds to the linearized operator of the numerical scheme about the discrete shock profile $\dspD$.
	
	For $\gamma\in[0,+\infty[$, we introduce the polynomial-weighted spaces $\ell^r_\gamma$ and their norms, as well as the space $\E_\gamma$ of zero-mass elements of $\ell^1_\gamma$ :
	\begin{subequations}\label{def:EspAPoids}
		\begin{align}
			\forall r\in[1,+\infty],\quad& \ell^r_\gamma:= \lbrace (h_j)_{j\in\Z}\in\C^\Z,\quad \left((1+|j|)^\gamma h_j\right)_{j\in\Z}\in\ell^r(\Z)\rbrace, \\
			\forall r\in[1,+\infty],\forall h \in \ell^r_\gamma,\quad& \left\|h\right\|_{\ell_\gamma^r} := \left\|\left((1+|j|)^\gamma h_j\right)_{j\in\Z}\right\|_{\ell^r},\\
			&\E_\gamma := \left\{ h\in \ell^1_\gamma ,\quad \sum_{j\in\Z}h_j=0\right\}.\label{def:Elamb}
		\end{align}
	\end{subequations}
	
	\begin{prop}\label{prop:Est}
		For any $0\leq \gamma\leq \Gamma$, there exists a constant $C_\Lcc(\gamma,\Gamma)>0$ such that we have the following estimates on the semigroup $(\Ldsp^n)_{n\in\N}$:
		\begin{subequations}\label{prop:Est:all}
			\begin{align}
				\forall n\in\N,\forall h\in \E_{\Gamma}, \quad& \left\|\Ldsp^nh\right\|_{\ell^1_{\gamma}}\leq \frac{C_\Lcc(\gamma,\Gamma)}{(n+1)^{\Gamma-\gamma}} \left\|h\right\|_{\ell^1_{\Gamma}},\label{prop:Est:L^n:1}\\
				\forall n\in\N,\forall h\in \E_{\Gamma}, \quad& \left\|\Ldsp^nh\right\|_{\ell^\infty_{\gamma}}\leq \frac{C_\Lcc(\gamma,\Gamma)}{(n+1)^{\Gamma-\gamma+\min\left(\gamma,\frac{1}{2\mu}\right)}} \left\|h\right\|_{\ell^1_{\Gamma}},\label{prop:Est:L^n:infty}
			\end{align}
			and the following estimates on the family of operators $(\Ldsp^n(Id-\Tc))_{n\in\N}$:
			\begin{align}
				\forall n\in\N,\forall h\in \ell^1_{\Gamma}, \quad& \left\|\Ldsp^n(Id-\Tc)h\right\|_{\ell^1_{\gamma}}\leq \frac{C_\Lcc(\gamma,\Gamma)}{(n+1)^{\Gamma-\gamma+\frac{1}{2\mu}}} \left\|h\right\|_{\ell^1_{\Gamma}},\label{prop:Est:L^n(id-T):1,1}\\
				\forall n\in\N,\forall h\in \ell^1_{\Gamma}, \quad& \left\|\Ldsp^n(Id-\Tc)h\right\|_{\ell^\infty_{\gamma}}\leq \frac{C_\Lcc(\gamma,\Gamma)}{(n+1)^{\Gamma-\gamma+\frac{1}{2\mu}+\min\left(\gamma,\frac{1}{2\mu}\right)}}\left\|h\right\|_{\ell^1_{\Gamma}},\label{prop:Est:L^n(id-T):infty,1}\\
				\forall n\in\N,\forall h\in \ell^\infty_{\Gamma}, \quad& \left\|\Ldsp^n(Id-\Tc)h\right\|_{\ell^\infty_{\gamma}}\leq \frac{C_\Lcc(\gamma,\Gamma)}{(n+1)^{\Gamma-\gamma+\min\left(\gamma,\frac{1}{2\mu}\right)}} \left\|h\right\|_{\ell^\infty_{\Gamma}},\label{prop:Est:L^n(id-T):infty,infty}
			\end{align}
			Furthermore, for any $\gamma,p\in[0,+\infty[$, there exists a constant $C_L(\gamma,p)$ such that we have the following estimates on the family of operators $\left(\Ldsp^n\left(\LdspD-\Ldsp\right)\right)_{n\in\N}$:
			\begin{align}
				\forall \delta \in I,\forall n\in\N,\forall h\in \ell^\infty(\Z), \quad& \left\|\Ldsp^n(\LdspD-\Ldsp)h\right\|_{\ell^1_{\gamma}}\leq \frac{C_L(\gamma,p)|\delta|}{(n+1)^{p}} \left\|h\right\|_{\ell^\infty},\label{prop:Est:L^n(LD-L)}
			\end{align}
		\end{subequations}
	\end{prop}
	
	Proposition \ref{prop:Est} can be seen as an improvement of the linear stability result \cite[Theorem 2]{Coeuret2024d} and will play a central role on proving the main result of this paper, that is Theorem \ref{Th}.
	
	\vspace{0.1cm}\noindent\textbf{\underline{A slight detour: On the mass function $M$ and the identifiction by mass of SDSPs}}\vspace{0.1cm}
	
	 We allow ourselves a slight detour in the content of the paper to discuss on the identification by mass of spectrally stable SDSPs and the verification of Hypothesis \ref{H:identification}. Indeed, as hinted when we introduced Hypothesis \ref{H:identification}, we claim that under regularity assumptions on the family of discrete shock profiles $\left(\dspD\right)_{\delta\in I}$, one could prove the injectivity of the mass function $M$ defined by \eqref{def:M} as well as the verification of the inequality \eqref{in:M} using the \textit{Liu-Majda condition} \eqref{Liu-Majda_condition}. 
	
	Indeed, let us consider that Hypotheses \ref{H:SDSP}-\ref{H:Mpm1} \textit{except} Hypothesis \ref{H:identification} are verified. Let us furthermore assume that, in Hypothesis \ref{H:SDSP}, the family of SDSPs $(\dsp^\delta)_{\delta\in I}$ is differentiable and that the \textit{mass} function $M$ defined by \eqref{def:M} is of class $\Cc^1$. We point out that this additional hypothesis of regularity on the family of discrete shock profiles $\left(\dspD\right)_{\delta\in I}$ seems fairly legitimate in practice, but known existence results on discrete discrete shock profiles do not yet allow to prove it. As we explained under \eqref{egV}, the sequence $V$ defined by \eqref{egV} is collinear to the sequence $\frac{\partial \dspD}{\partial\delta}\big|_{\delta=0}$. Thus, the \textit{Liu-Majda condition} \eqref{Liu-Majda_condition} implies that:
	$$\sum_{j\in\Z} \left(\frac{\partial \dspD}{\partial\delta}\Big|_{\delta=0}\right)_j\neq 0$$
	and thus 
	$$M^\prime(0)\neq 0.$$
	This would allow us to conclude on the verification of Hypothesis \ref{H:identification}.
	
	\subsection{Main result of the article: nonlinear orbital stability in the scalar case}\label{subsec:MainRes}
	
	Before stating Theorem \ref{Th}, we need to introduce some particular conditions. First, we claim that the triplet of positive constants $(a,b,c)\in[0,+\infty[^3$ satisfies the condition \eqref{cond:H} when:
	\begin{align}\label{cond:H}\tag{H}
		\begin{split}
			1-a\leq b-c& \quad \text{ if } a\in[0,1[,\\ 
			0 < b-c & \quad \text{ if } a=1,\\
			0 \leq b-c & \quad \text{ if } a>1.
		\end{split}
	\end{align}
	The condition \eqref{cond:H} appears in a technical lemma below (Lemma \ref{lem:InSum}). Using the newly introduced condition \eqref{cond:H}, we introduce for quadruplets of constants $(\gamma_1,\gamma_\infty,\pg_1,\pg_\infty)\in[0,+\infty[^4$ the following conditions \ref{Th:C1}-\ref{Th:C4} where the constant $\mu$ is defined in Hypothesis \ref{H:F}:
	\begin{enumerate}[label=(C\arabic*)]
		\item \label{Th:C1} The triplet $\left(\pg_1+\pg_\infty,\, \gamma_\infty+\frac{1}{2\mu},\, \pg_1\right)$ satisfies the condition \eqref{cond:H}.
		\item \label{Th:C2} The triplet $\left(\gamma_\infty+\frac{1}{2\mu},\, \pg_1+\pg_\infty,\, \pg_1\right)$ satisfies the condition \eqref{cond:H}.
		\item \label{Th:C3} At least one of the two triplets:
		$$\left(\pg_1+\pg_\infty,\, \gamma_1+\frac{1}{2\mu}+\min\left(\gamma_\infty,\, \frac{1}{2\mu}\right),\, \pg_\infty\right) \text{ or }\left(2\pg_\infty,\, \gamma_\infty+\min\left(\gamma_\infty,\, \frac{1}{2\mu}\right),\, \pg_\infty\right) $$ 
		verifies the condition \eqref{cond:H}.
		\item \label{Th:C4} At least one of the two triplets:
		$$\left(\gamma_1+\frac{1}{2\mu}+\min\left(\gamma_\infty,\, \frac{1}{2\mu}\right),\, \pg_1+\pg_\infty,\, \pg_\infty\right)\text{ or }\left(\gamma_\infty+\min\left(\gamma_\infty,\, \frac{1}{2\mu}\right),\, 2\pg_\infty,\, \pg_\infty\right) $$
		verifies the condition \eqref{cond:H}.
	\end{enumerate}
		
	We will now state the main result of this article. We recall that the polynomially-weighted spaces $\ell^r_\gamma$ and their norms appearing in Theorem \ref{Th} are defined by \eqref{def:EspAPoids}. 
	
	\begin{theorem}\label{Th}
		We assume that Hypotheses \ref{H:SDSP}-\ref{H:Mpm1} are verified and we consider four constants $(\gamma_1,\gamma_\infty,\pg_1,\pg_\infty)\in[0,+\infty[^4$ that satisfy the conditions \ref{Th:C1}-\ref{Th:C4}. We define the constant: 
		\begin{equation}\label{def:Gamma}
			\Gamma:= \max\left(\pg_1+\gamma_1, \, \pg_\infty+\gamma_\infty-\frac{1}{2\mu},\, \pg_\infty,\, \gamma_\infty\right).
		\end{equation}
		Then, there exist two constants $\varepsilon,C_0\in[0,+\infty[$ such that, for any initial perturbation $\hg\in\ell^1_{\Gamma}$ verifying:
		$$\left\|\hg\right\|_{\ell^1_{\Gamma}}<\varepsilon $$
		the following assertions are verified:
		\begin{itemize}
			\item There exists a unique constant $\delta \in I$ such that:
			\begin{equation}\label{Th:IDdelta}
				\sum_{j\in\Z}\hg_j = M(\delta)
			\end{equation}
			where the mass function $M$ is defined by \eqref{def:M}.
			\item The solution $(u^n)_{n\in\N}$ of the numerical scheme \eqref{def:SchemeNum} with the initial condition $u^0:=\dsp+\hg$ is well-defined for all $n\in\N$.
			\item If we introduce the sequences $(h^n)_{n\in\N}$ defined by:
			$$\forall n\in\N,\quad h^n:= u^n-\dspD,$$
			then the sequence $h^n$ belongs to $\ell^1_{\gamma_1}\cap \ell^\infty_{\gamma_\infty}$ for all $n\in\N$ and:
			\begin{equation}\label{Th:in}
				\forall n\in\N ,\quad \left\|h^n\right\|_{\ell^1_{\gamma_1}}\leq \frac{C_0}{(n+1)^{\pg_1}}\left\|\hg\right\|_{\ell^1_{\Gamma}}\quad \text{ and } \quad  \left\|h^n\right\|_{\ell^\infty_{\gamma_\infty}}\leq \frac{C_0}{(n+1)^{\pg_\infty}}\left\|\hg\right\|_{\ell^1_{\Gamma}}. 
			\end{equation}
			Up to having $\pg_1>0$ (resp. $\pg_\infty>0$), this implies that the solution $(u^n)_{n\in\N}$ of the numerical scheme converges towards $\dspD$ in the $\ell^1_{\gamma_1}$-norm (resp. $\ell^\infty_{\gamma_\infty}$-norm).
		\end{itemize}
	\end{theorem}
	
	Theorem \ref{Th} provides a first partial answer to \cite[Open Question 5.3]{Serre2007} in the case of \textit{scalar} conservation laws. Let us clarify some details in the statement of Theorem \ref{Th}. The role of the constants $\gamma_1$, $\gamma_\infty$, $\pg_1$ and $\pg_\infty$ appears clearly in \eqref{Th:in}. The constants $\gamma_1$ and $\gamma_\infty$ correspond to choices of weights of spaces in which the sequences $h^n$ are evaluated and the constants $\pg_1$ and $\pg_\infty$ correspond to decay rates. Thus, if Hypotheses \ref{H:SDSP}-\ref{H:Mpm1} and in particular the spectral stability assumption (Hypotheses \ref{H:spec} and \ref{H:Evans}) are verified, Theorem \ref{Th} allows to find sets of parameters $\gamma_1,\gamma_\infty,\pg_1,\pg_\infty\in[0,+\infty[$ such that the nonlinear stability estimates \eqref{Th:in} with \textit{explicit} decay rates would be satisfied for initial perturbations $\hg$ small in some polynomially-weighted space $\ell^1_\Gamma$. The only conditions that the constants $\gamma_1$, $\gamma_\infty$, $\pg_1$ and $\pg_\infty$ must verify are \ref{Th:C1}-\ref{Th:C4}. We also propose an explicit choice of constant $\Gamma$, defined by \eqref{def:Gamma}. Let us observe that the choice of $\delta$ determined by \eqref{Th:IDdelta} corresponds to \eqref{IDMasseFinal} introduced earlier in the state of the art. In Section \ref{sec:Num}, we will present specific choices of constants $\gamma_1,\gamma_\infty,\pg_1,\pg_\infty\in[0,+\infty[$ that satisfy \ref{Th:C1}-\ref{Th:C4} in order to have more concrete uses of Theorem \ref{Th} on which one could rely. We will also numerically test Theorem \ref{Th} with those cases. 
	
	Theorem \ref{Th} and its proof can be seen to some extent as an adaptation in the fully discrete setting of the nonlinear orbital stability result for semi-discrete shock profile \cite[Theorem 5.1]{BenzoniGavage2003} which is itself inspired by the articles \cite{Zumbrun2000,Mascia2002}. To briefly summarize, we will use a proof by induction. For all $n\in\N$, we will find a way to express the perturbation $h^{n+1}$ using the perturbations $h^m$ and the operators $\Ldsp^m$, $\Ldsp^m(Id-\Tc)$ and $\Ldsp^m(\LdspD-\Ldsp)$ for $m\in\lbrace0,\hdots n\rbrace$. The decomposition of the Green's function deduced in \cite{Coeuret2024d} allowed us to prove sharp bounds on those families of operators (Proposition \ref{prop:Est}). This will allow us to prove inequalities on the sequences $h^n$ by induction. Let us however point out that, contrarily with the proofs of nonlinear orbital stability in \cite{Zumbrun2000,Mascia2002,BenzoniGavage2003}, we do not try to approach the shock location to improve the estimates \eqref{Th:in} by introducing a sequence $(\delta_n)_{n\in\N}$ such that we would try to find estimates on $h^n:=u^n-\overline{u}^{\delta_n}$. We claim that, for several reasons related to the discrete nature of the problem in space and in time, this would be far more difficult to consider. This would however be an interesting direction to follow to improve Theorem \ref{Th}.
	
	\vspace{0.1cm}\noindent\textbf{\underline{Novelties and limitations of Theorem \ref{Th}}}\vspace{0.1cm}
	
	Let us make some observations on Theorem \ref{Th} in order to compare it with the state of the art surrounding the stability theory of discrete shock profiles presented above and discuss on possible future improvements. 
	
	\begin{itemize}
		\item The first main observation is that, due to the adaptation of the stability analysis performed in \cite{Zumbrun1998}, we did not have to introduce any weakness assumption of the shock \eqref{def:choc} but rather a spectral stability assumption on the discrete shock profile $\dsp$ (Hypotheses \ref{H:spec} and \ref{H:Evans}).
		\item With regards to the restriction on the numerical schemes considered, we impose on the numerical schemes that they must introduce artificial possibly high-order viscosity (Hypothesis \ref{H:F}). For instance, the main result of the present article does not apply for the Lax-Wendroff scheme since it displays a dispersive behavior. This "diffusive" limitation is necessary due to the result \cite[Theorem 1]{Coeuret2024d} which only applies for numerical schemes displaying a parabolic behavior. However, this still allows us to consider all first order schemes and also higher odd ordered schemes. Let us however point out that, in a recent preprint \cite{Coulombel2024a}, Coulombel and Faye proved a nonlinear orbital stability result for stationary discrete shock profiles and for the Lax-Wendroff scheme. Similarly as in the present article, the proof relies crucially on the analysis of the Green's function of a reference discrete shock profile to obtain a linear stability result. The proof of the nonlinear stability is then deduced using the same strategy and the lemmas that will be proved below.
		\item The initial perturbations $\hg$ have to be small in some polynomially-weighted space $\ell^1_\Gamma$. This is an improvement compared to results like \cite{Smyrlis1990} which impose for the initial perturbations to be small in some exponentially-weighted space. Furthermore, compared to previous nonlinear orbital stability results, Theorem \ref{Th} proves decay rates \eqref{Th:in} on the sequences $h^n$. As will be clearer in Section \ref{sec:Num} when we will give concrete choices of parameters $\gamma_1,\gamma_\infty,\pg_1,\pg_\infty\in[0,+\infty[$ to choose in Theorem \ref{Th}, the definition \eqref{def:Gamma} of the constant $\Gamma$ implies that one could choose the decay rates $\pg_1$ and $\pg_\infty$ to be as high as one wishes un to considering a larger constant $\Gamma$ that parametrizes the weights on the initial perturbations $\hg$.
		 \item Theorem \ref{Th} is restricted to the case of \textit{scalar} conservation laws. However, we claim that one could hope for a generalization of such a nonlinear orbital stability result for discrete shock profiles of \textit{systems} of conservation laws. Let us present some observations towards this direction:
		 \begin{itemize}[label=$\star$]
		 	\item One would need to investigate the identification by mass introduced in Hypothesis \ref{H:identification} since the function $M$ defined by \eqref{def:M} would no longer be scalar valued. Such a discussion is held in \cite{Serre2007}.
		 	\item  The result \cite[Theorem 1]{Coeuret2024d} which provides the long-time behavior of the Green's function associated with the linearization of the numerical scheme about spectrally stable discrete shock profiles also holds in the system case.
		 	\item In the case of systems of conservation laws, the behavior of the Green's function associated with the spectrally stable discrete shock profile is far more complicated. Indeed, contrarily to the \textit{scalar} case presented above where the long time asymptotic behavior of the Green's function is to concentrate its mass at the shock location by activating the profile $V$, the Green's function in the \textit{system} case displays generalized Gaussian waves going towards $\pm \infty$ for instance via reflected and transmitted waves (see \cite{Coeuret2024d} for more details). This relates back to the identification by mass issue we mentioned above and also implies the necessity of a careful analysis for the choices of the norms to obtain a generalization of the linear decay estimates of Proposition \ref{prop:Est}.
		 \end{itemize}
		 The difficult case of adapting Theorem \ref{Th} for systems of conservation laws could be handled in a future article.
	\end{itemize}

	\subsection{Plan of the article}
	
	The rest of the article is separated in four parts. First, in Section \ref{sec:Num}, we will present concrete choices for the constants $\gamma_1$, $\gamma_\infty$, $\pg_1$ and $\pg_\infty$ appearing in Theorem \ref{Th} and use those examples to numerically test Theorem \ref{Th}. In Section \ref{sec:PreuveTh}, we will assume that Proposition \ref{prop:Est} is already proved and focus on the main matter of this paper, the proof of Theorem \ref{Th}. Finally, in Section \ref{sec:Preuve_prop:Est}, we use the decompositions \eqref{decompoGreen} and \eqref{decompoDerGreen} of the Green's function and its discrete derivative to prove the estimates on the operators $\Ldsp^n$, $\Ldsp^n(Id-\Tc)$ and $\Ldsp^n(\LdspD-\Ldsp)$ claimed in Proposition \ref{prop:Est}.
	
	\section{Concrete choices of parameters in Theorem \ref{Th} and numerical tests}\label{sec:Num}
	
	The expression of the conditions \ref{Th:C1}-\ref{Th:C4} makes it difficult to quickly identify convenient choices of parameter $\gamma_1$, $\gamma_\infty$, $\pg_1$ and $\pg_\infty$ to satisfy them, and thus to use in Theorem \ref{Th}. In this section, we will present two specific sets of parameters which satisfy those conditions \ref{Th:C1}-\ref{Th:C4} and then numerically test Theorem \ref{Th} with those choices of parameters. In Section \ref{subsec:NumVer}, we will present our methodology to numercially test Theorem \ref{Th}. Section \ref{subsec:ChoixParam} will then be dedicated to the presentation of choices of parameters $\gamma_1$, $\gamma_\infty$, $\pg_1$ and $\pg_\infty$ and of the results of the numerical tests.
	
	\subsection{Methodology of the numerical tests}\label{subsec:NumVer}
	
	\vspace{0.1cm}\noindent\textbf{\underline{Choices of conservation law, shock and numerical scheme}}\vspace{0.1cm}
	
	Let us here present our methodology for the numerical tests of Theorem \ref{Th}. First, we will have to fix a choice of conservation law, of stationary Lax shock and of numerical scheme. Those choices will be the same as in \cite[Section 6]{Coeuret2024d}. We will consider the Burgers equation for the scalar conservation law \eqref{def:EDP}, which is defined by the choice of flux $f$:
	$$\forall u \in \R,\quad f(u):= \frac{u^2}{2}.$$
	For the stationary Lax shock, we consider the two states:
	\begin{equation}\label{Num:Shock}
		u^-=1\quad \text{ and } \quad u^+=-1
	\end{equation}
	which satisfy the Rankine-Hugoniot condition \eqref{cond:RK} as well as the Lax shock condition \eqref{cond:Lax}. Finally, with regards to the choice of numerical scheme, we consider the modified Lax-Friedrichs scheme for which the numerical flux is defined by:
	$$\forall \nu\in]0,+\infty[,\forall u_{-1},u_0\in\R, \quad F(\nu;u_{-1},u_0):=\frac{f(u_{-1})+f(u_0)}{2}+D(u_{-1}-u_0)$$
	where $D$ is a positive constant. We immediately observe that the consistency condition \eqref{cond:consistency} is verified. The discrete evolution operator $\Nc$ is defined for $u\in\R^\Z$ by:
	\begin{equation}\label{def:MLF}
		\forall j\in\Z,\quad (\Nc(u))_j := u_j - \nug\left(\frac{f(u_{j+1})-f(u_{j-1})}{2}+D\left(-u_{j+1}+2u_j-u_{j-1}\right)\right).
	\end{equation}
	Throughout the rest of the numerical test, we will consider that $\nug:=0.5$ and $D:=\frac{0.4}{\nug}$. Observing that:
	$$\nug<2D\nug<1,$$
	we claim that Hypotheses \ref{H:F}, \ref{H:inv} and \ref{H:Mpm1} are verified (see \cite[Section 6]{Coeuret2024d} for details). Let us observe that, for the Lax-Friedrichs scheme, the constant $\mu$ is equal to $1$ in Hypothesis \ref{H:F} as it is a first order scheme.
	
	\vspace{0.1cm}\noindent\textbf{\underline{Construction of the SDSPs $\dspD$}}\vspace{0.1cm}
	
	For the choices of conservation law and scheme introduced above, we claim that we can numerically observe the existence of a family of SDSPs $(\dspD)_{\delta\in \R}$ associated with the shock \eqref{Num:Shock}. Indeed, for $\delta\in \R$, if we consider the solution $(u^{n,\delta})_{n\in \N}$ of the numerical scheme \eqref{def:MLF} using the initial condition $u^{0,\delta}$ defined by:
	\begin{equation}\label{Num:CI}
		u^{0,\delta}_j:=\left\{\begin{array}{cc}
		1 & \quad \text{ if } j<\lfloor\frac{\delta+1}{2}\rfloor,\\ 
		\delta -2\lfloor\frac{\delta+1}{2}\rfloor & \quad \text{ if } j=\lfloor\frac{\delta+1}{2}\rfloor,\\ 
		-1 & \quad \text{ if } j>\lfloor\frac{\delta+1}{2}\rfloor,\\ 
	\end{array}\right.
	\end{equation}
	then we can numerically observe that the solution $(u^{n,\delta})_{n\in \N}$ converges towards a discrete shock profile $\dspD$, i.e. a stationary solution of the numerical scheme $\Nc$ linking the two states $u^-$ and $u^+$. We represent on Figure \ref{fig:SDSP} some of those limits $\dspD$. We have a thus a continuum of SDSPs $(\dspD)_{\delta \in \R}$ associated with our shock, as stated in Hypothesis \ref{H:SDSP}, and they seem to satisfy the estimates of Hypothesis \ref{H:CVExpo}. Furthermore, the conservative nature of the numerical scheme considered and the definition \eqref{Num:CI} of the initial conditions $u^{0,\delta}$ imply that:
	$$\forall \delta \in \R,\forall n\in \N,\quad \sum_{j\in \Z} u^{n,\delta}_j-u^{n,0}_j=\sum_{j\in \Z} u^{0,\delta}_j-u^{0,0}_j=\delta.$$
	Therefore, passing to the limit in time $n$, the mass function $M$ defined by \eqref{def:M} verifies:
	$$\forall \delta \in \R,\quad M(\delta)=\delta.$$
	Thus, Hypothesis \ref{H:identification} would be verified in this case. We can even say that we have a parametrization by mass of the discrete shock profiles $(\dspD)_{\delta\in \R}$.
	
	\vspace{0.1cm}\noindent\textbf{\underline{On the verification of the estimates \eqref{Th:in} of Theorem \ref{Th}}}\vspace{0.1cm}
	
	We will assume that the discrete shock profile $\dsp:=\dsp^0$ is spectrally stable and thus assume that we can use Theorem \ref{Th}. Let us consider a set of parameters $(\gamma_1,\gamma_\infty,\pg_1,\pg_\infty)\in[0,+\infty[^4$ that satisfy conditions \ref{Th:C1}-\ref{Th:C4} and define the constant $\Gamma$ using \eqref{def:Gamma}.  We want to verify if the inequality \eqref{Th:in} is verified. To do so, we consider the following family of initial perturbation $(\hg_J)_{J\in \N\backslash\left\{0\right\}}$ defined by:
	$$\forall J\in \N\backslash\left\{0\right\},\forall j\in \Z, \quad (\hg_{J})_j:= \left\{ \begin{array}{cc}
		-\displaystyle\frac{1}{1+(1+|J|)^\Gamma}& \quad \text{ if } j=0,\\
		\displaystyle\frac{1}{1+(1+|J|)^\Gamma}& \quad \text{ if } j=J,\\
		0 & \quad \text{ else.}
	\end{array} \right. $$
	We observe that we constructed those initial perturbations $\hg_J$ so that:
	\begin{equation*}
		\forall J\in \N\backslash\left\{0\right\}, \quad \left\|\hg_J\right\|_{\ell^1_\Gamma}=1\quad \text{ and } \quad \sum_{j\in\Z}(\hg_J)_j=0.
	\end{equation*}
	We construct the solution $(u^n_J)_{n\in\N}$ of the modified Lax-Friedrichs scheme \eqref{def:MLF} with the initial condition $u^0_J:= \dsp+\hg_J$. Since the sequences $\hg_J$ are zero-mass perturbations, the choice of $\delta$ in Theorem \ref{Th} defined by \eqref{Th:IDdelta} is necessarily $\delta=0$ and therefore the solution $(u^n_J)_{n\in\N}$ converges for long time $n$ towards the initial discrete shock profile $\dsp$. If we define:
	$$\forall J\in\N\backslash\left\{0\right\},\forall n\in\N,\quad h^n_J:= u^n_J-\dsp, $$
	Theorem \ref{Th} implies that the sequences $h^n_J$ should belong to $\ell^1_{\gamma_1}\cap \ell^\infty_{\gamma_\infty}$ and there should exist a constant $C_0>0$ such that:
	\begin{subequations}\label{Num:in}
		\begin{align}
			\forall J\in\N\backslash\left\{0\right\},\forall n\in\N,\quad&  \left\|h^n_J\right\|_{\ell^1_{\gamma_1}}\leq \frac{C_0}{(n+1)^{\pg_1}}\left\|\hg_J\right\|_{\ell^1_{\Gamma}},\label{Num:in1}\\
			\forall J\in\N\backslash\left\{0\right\},\forall n\in\N,\quad&  \left\|h^n_J\right\|_{\ell^\infty_{\gamma_\infty}}\leq \frac{C_0}{(n+1)^{\pg_\infty}}\left\|\hg_J\right\|_{\ell^1_{\Gamma}}.\label{Num:inInf}
		\end{align}
	\end{subequations}
	We fix two constants $n_{max},J_{max}\in\N\backslash\left\{0\right\}$ to numerically test \eqref{Num:in}. We will compute for times $n\in \left\{0,\hdots,n_{max}\right\}$ the values:
	\begin{equation}\label{Num:Sup}
		\sup_{J\in \left\{1,\hdots,J_{max}\right\}}\log\left(\frac{\left\|h^n_J\right\|_{\ell^1_{\gamma_1}}}{\left\|\hg_J\right\|_{\ell^1_{\Gamma}}}\right) \quad \text{ and } \quad \sup_{J\in\left\{1,\hdots,J_{max}\right\}}\log\left(\frac{\left\|h^n_J\right\|_{\ell^\infty_{\gamma_\infty}}}{\left\|\hg_J\right\|_{\ell^1_{\Gamma}}}\right).
	\end{equation}
	We will then compute linear regressions of the logarithm of those values \eqref{Num:Sup} with respect to $\log(n)$ and observe if the slope obtained via this linear regression is indeed inferior to the slope expected via \eqref{Num:in}, i.e. $-\pg_1$ or $-\pg_\infty$ depending on what is computed. We will also display figures that represent for all $J\in\left\{1,\hdots,J_{max}\right\}$ the logarithm of the ratio of $\left\|h^n_J\right\|_{\ell^1_{\gamma_1}}$ or $\left\|h^n_J\right\|_{\ell^\infty_{\gamma_\infty}}$ with $\left\|\hg_J\right\|_{\ell^1_{\Gamma}}$. In our numerical tests, we choose $J_{max}=500$ and $n_{max}$ large enough for the linear regression depending on our choice of parameters $\gamma_1$, $\gamma_\infty$, $\pg_1$ and $\pg_\infty$. 
	
	Let us observe that the choice of initial perturbations $(\hg_J)_{J\in\N\backslash\left\{0\right\}}$ defined by \eqref{Num:CI} can seem fairly arbitrary. We point out that we tried to add some other families of initial perturbations $\hg_J$ to test if this changed anything. However, it seems as though the supremum \eqref{Num:Sup} was achieved for this choice of initial perturbation or its opposite, at least in our tests.
	
	\subsection{Choices of parameters \texorpdfstring{$\gamma_1$, $\gamma_\infty$, $\pg_1$ and $\pg_\infty$}{gamma1, gammainfty, p1, pinfty}}\label{subsec:ChoixParam}
	
	Let us start this section by describing two choices of parameters $(\gamma_1,\gamma_\infty,\pg_1,\pg_\infty)\in[0,+\infty[^4$ satisfying conditions \ref{Th:C1}-\ref{Th:C4} and we will then present the results obtained via the methodology presented above.  
		
	\textbf{\underline{Choice 1 of parameters}}
	
	We consider a positive constant $\pg\in[0,+\infty[$ such that:
	\begin{equation*}
		\pg \geq \frac{1}{2}\left(1-\frac{1}{\mu}\right).
	\end{equation*}
	Furthermore, for $\mu=1$ (i.e. schemes of order $1$), this inequality must be strict. Then, Theorem \ref{Th} holds for:
	\begin{equation}\label{NumChoix1}
		\gamma_1:=\pg,\quad \gamma_\infty:=\pg+\frac{1}{2\mu},\quad \pg_1:=\pg,\quad \pg_\infty:=\pg
	\end{equation}
	which satisfy conditions \ref{Th:C1}-\ref{Th:C4}. We then have that the constant $\Gamma$ defined by \eqref{def:Gamma} satisfies:
	\begin{equation*}
		\Gamma:= \pg+\max\left(\pg,\frac{1}{2\mu}\right). 
	\end{equation*}
	
	\textbf{\underline{Choice 2 of parameters}}
	
	We consider a positive constant $\pg\in[0,+\infty[$ such that:
	\begin{equation}\label{Num:CondChoix2}
		\pg \geq \max\left(\frac{1}{2}\left(1-\frac{1}{\mu}\right),\frac{1}{2\mu}\right).
	\end{equation}
	Then, Theorem \ref{Th} holds for:
	\begin{equation}\label{NumChoix2}
		\gamma_1:=\pg,\quad \gamma_\infty:=\pg,\quad \pg_1:=\pg,\quad \pg_\infty:=\pg+\frac{1}{2\mu}
	\end{equation}
	which satisfy conditions \ref{Th:C1}-\ref{Th:C4}. We then have that the constant $\Gamma$ defined by \eqref{def:Gamma} satisfies:
	\begin{equation*}
		\Gamma:= 2\pg.
	\end{equation*}
	
	\textbf{\underline{Numerical results}}
	
	\begin{table}
		\begin{center}
			\begin{tabular}{|c|c||c|c||c|c|}
				\hline Choice of parameters & $\pg$ & Slope obtained for \eqref{Num:in1} & $-\pg_1$ & Slope obtained for \eqref{Num:inInf}& $-\pg_\infty$ \\
				\hline 1 - \eqref{NumChoix1} & $0.3$ & $-0.518270$ & $-0.3$ & $-0.601739$ & $-0.3$ \\ 
				\hline 1 - \eqref{NumChoix1} & $0.5$ & $-0.555493$ & $-0.5$ & $-0.620238$ & $-0.5$ \\ 
				\hline 1 - \eqref{NumChoix1} & $1$ & $-1.062254$ & $-1$ & $-1.057359 $ & $-1$ \\ 
				\hline 2 - \eqref{NumChoix2} & $0.5$ & $-0.555493$ & $-0.5$ & $-0.933069 $ & $-1$ \\ 
				\hline 2 - \eqref{NumChoix2} & $1$ & $-1.062254$ & $-1$ & $ -1.548798$ & $-1.5$ \\ 
				\hline
			\end{tabular}
		\end{center}
		\caption{Slopes obtained and expected for \eqref{Num:in} via the methodology mentioned previously. We expect the values of the third and fifth column to be inferior respectively to the one in the fourth and sixth column if Theorem \ref{Th} is verified.}
		\label{Num:TableSlopes}
	\end{table}
	
	We now apply the methodology presented above to numerically test \eqref{Th:in}. The Table \ref{Num:TableSlopes} presents the slopes obtained and expected for the choices of parameters \eqref{NumChoix1} and \eqref{NumChoix2}. Figures \ref{Num:FigV1} and \ref{Num:FigV2} present the linear regressions in some specific cases.
	
	\begin{figure}
		\centering
		\begin{subfigure}{0.49\textwidth}
			\includegraphics[width=\textwidth]{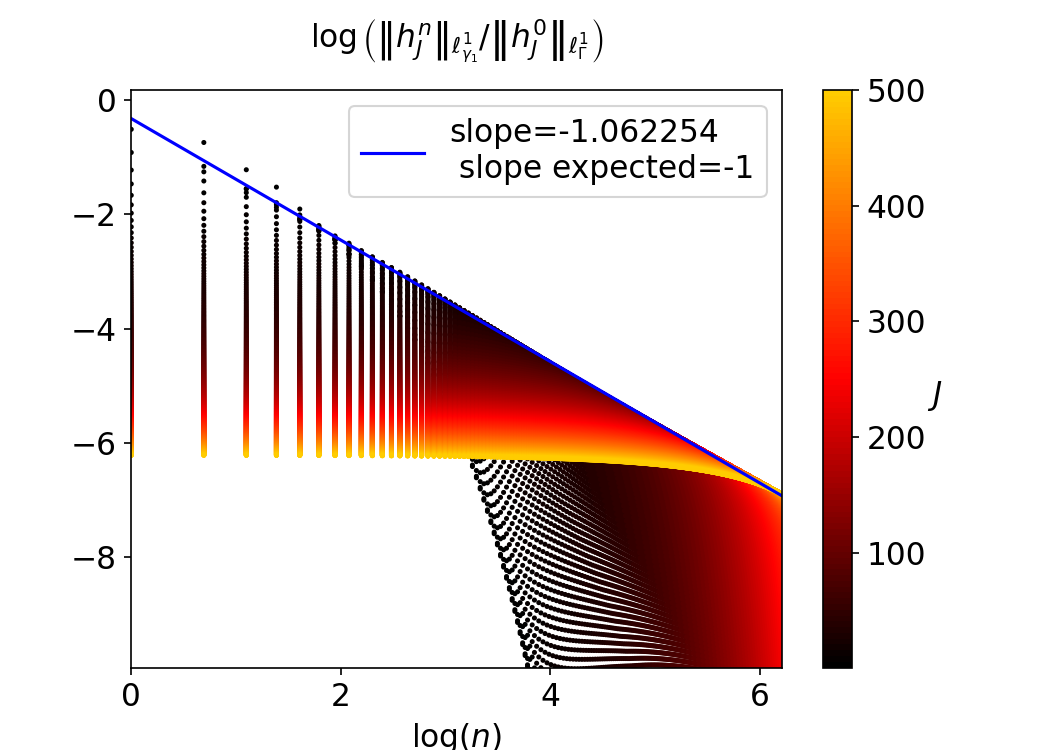}
		\end{subfigure}
		\hfill
		\centering
		\begin{subfigure}{0.49\textwidth}
			\includegraphics[width=\textwidth]{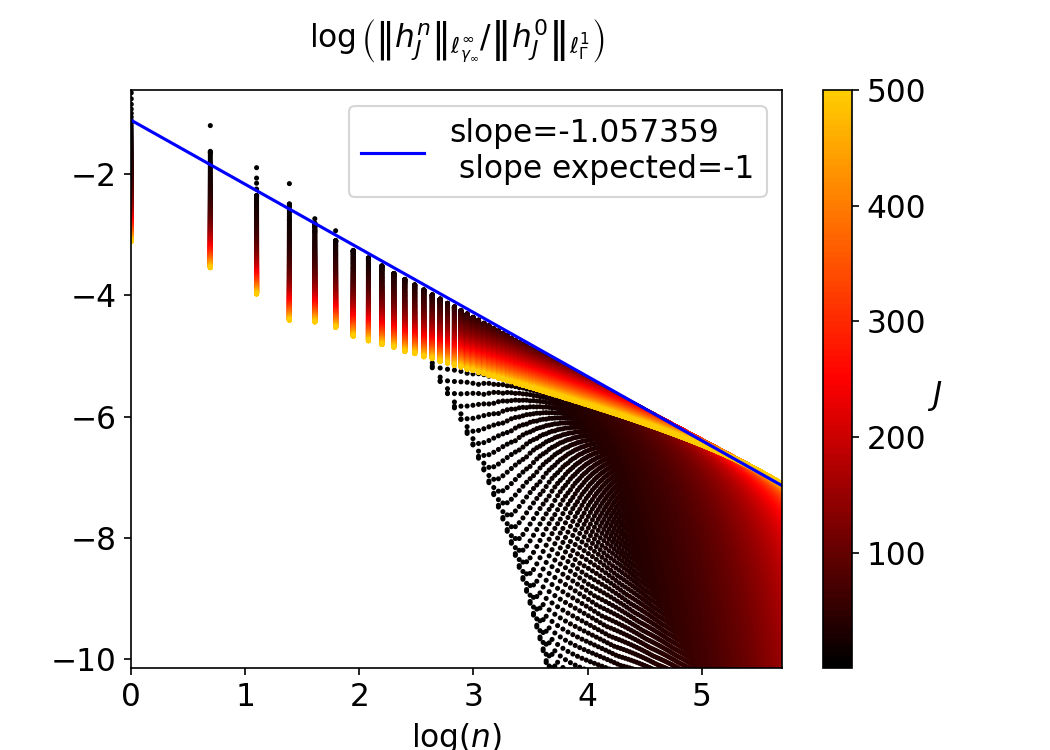}
		\end{subfigure}
		\caption{Representation of the values \eqref{Num:Sup} for the choice of parameters \eqref{NumChoix1} with $\pg=1$. The slopes that we obtain numerically are close to the expected slope, pointing to the fact that \eqref{Th:in} seems sharp in this case.}
		\label{Num:FigV1}
	\end{figure}
	
	The results of Table \ref{Num:TableSlopes} tend to prove that the inequality \eqref{Th:in} of Theorem \ref{Th} is verified. Indeed, we expect the values in the third and fifth columns to be respectively inferior to the values in the fourth and sixth column. We observe that this seems to be verified in all cases presented in the Table \ref{Num:TableSlopes}, except maybe for the slopes in the fifth column in the case of the parameters \eqref{NumChoix2} with $\pg=0.5$.  However, we claim that the slopes that would be obtained in this case when performing the calculations for larger values of $J_{max}$ and $n_{max}$ would be closer to the expected slope $-\pg_\infty=-1$. This case is represented on Figure \ref{Num:FigV2}.
	
	We also observe that several slopes computed in Table \ref{Num:TableSlopes} tend to be close to the expected ones. This points to the estimations \eqref{Th:in} being at least fairly sharp for some choices of parameters, for instance when we choose $\gamma_1$, $\gamma_\infty$, $\pg_1$ and $\pg_\infty$ defined by \eqref{NumChoix1} and \eqref{NumChoix2} with $\pg$ large.

	\begin{figure}
		\centering
		\begin{subfigure}{0.49\textwidth}
			\includegraphics[width=\textwidth]{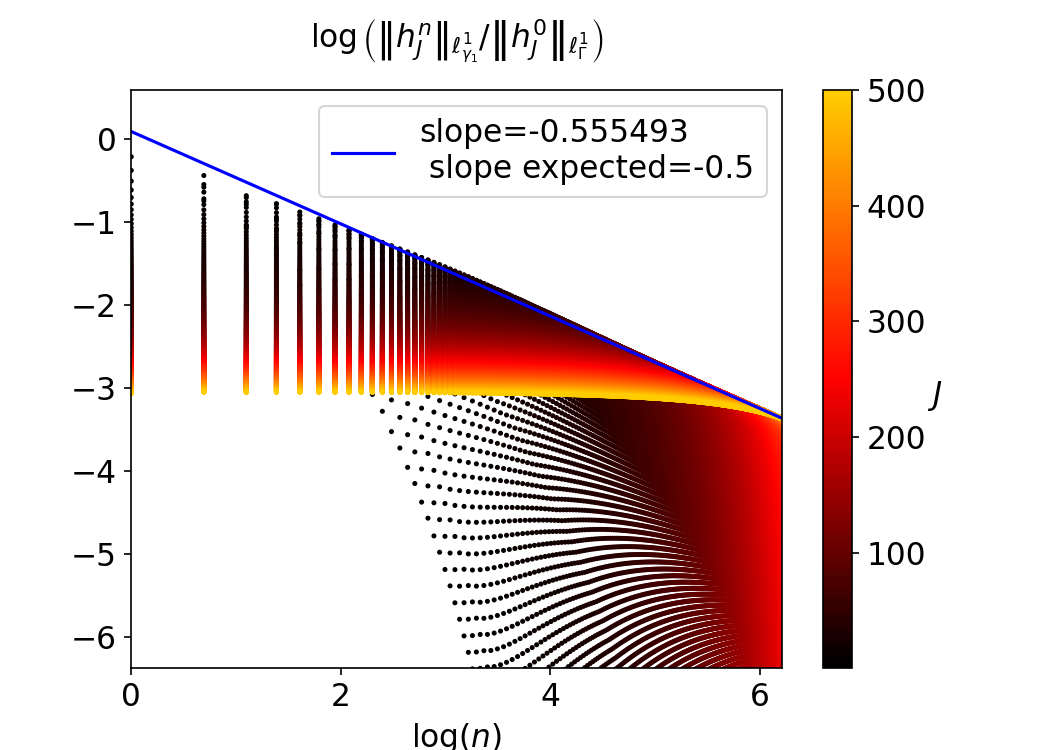}
		\end{subfigure}
		\hfill
		\centering
		\begin{subfigure}{0.49\textwidth}
			\includegraphics[width=\textwidth]{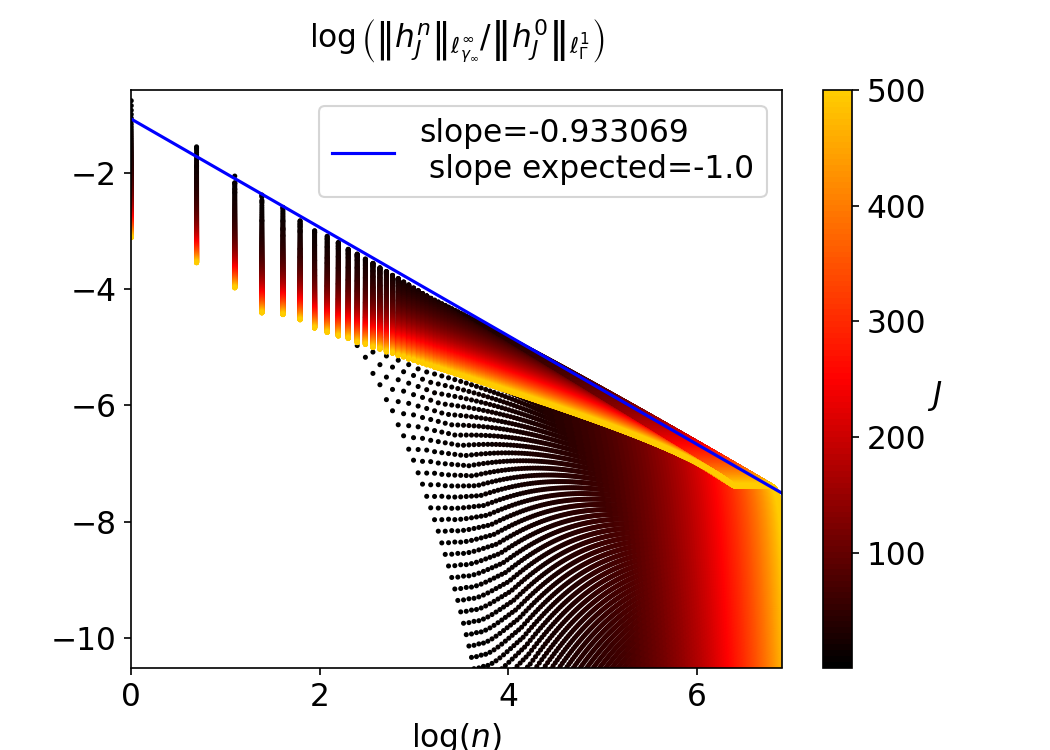}
		\end{subfigure}
		\caption{Representation of the values \eqref{Num:Sup} for the choice of parameters \eqref{NumChoix2} with $\pg=0.5$. For the figure on the right-side, the slope obtained is larger than one expected. However, if we performed the same calculations for $J_{max}$ and $n_{max}$ larger, the slope obtained would be closer to the expected one.}
		\label{Num:FigV2}
	\end{figure}
	
	\section{Nonlinear orbital stability in the scalar case}\label{sec:PreuveTh}
	
	\subsection{Necessary preliminary observations}
	
	The main goal of this section is the proof of Theorem \ref{Th}. We will consider that Proposition \ref{prop:Est} has been proved. Its proof is presented in Section \ref{sec:Preuve_prop:Est} below. Let us first start by introducing some useful lemmas and constants that will appear in the proof.
	
	\subsubsection*{Definition of a neighborhood of the states of the SDSP $\dspD$}
	
	An important part of the proof of Theorem \ref{Th} will rely on proving that, for small enough initial perturbations $\hg$, if we define the initial condition $u^0:=\dsp+\hg$, the sequences $u^n$ constructed using the numerical scheme 
	\begin{equation}\label{Th:defUn}
		\forall n\in\N,\quad u^{n+1}:= \Nc(u^n).
	\end{equation}
	are actually defined for all $n\in\N$. This is nontrivial as there could be a time $n\in\N$ and an integer $j\in\Z$ such that the scalar $u^n_j$ would not belong to the space of states $\Uc$ of the conservation law \eqref{def:EDP} and thus the solution $u^n$ of the numerical scheme would have left the domain of definition $\Uc^\Z$ of the operator $\Nc$ defined by \eqref{def:evolOpe}. We recall that we do not make any monotonicity assumption on the numerical scheme we consider here, which could allow to give a rapid and easy answer to this type of question. Since the set \eqref{set:CT} is relatively compact, there exists a radius $\Rayon>0$ such that:
	\begin{equation}\label{def:Rayon}
			\overline{\bigcup_{(\delta,j)\in I\times \Z}  B(\dspD_j,\Rayon)}\subset\Uc.
	\end{equation}
	The set defined in \eqref{def:Rayon} thus contains a neighborhood of the states of the SDSPs $\dspD$ and is also included in the space of states $\Uc$ of the scalar conservation law \eqref{def:EDP}. The definition \eqref{def:Rayon} of the radius $\Rayon$ implies that for $h\in\ell^\infty(\Z)$ such that $\left\|h\right\|_{\ell^\infty}<\Rayon$, we have that:
	$$\forall \delta\in I,\forall j \in\Z, \quad \dspD_j+h_j\in  \Uc,$$
	i.e. the perturbation $h$ is small enough so that the elements of the sequence $\dspD+h$ remain in the domain of definition $\Uc$ of the numerical scheme. Therefore, coming back to our initial issue of constructing the solution $(u^n)_{n\in\N}$ of the numerical scheme \eqref{Th:defUn}, if for some integer $n\in\N$ we can define the sequence $u^n$ and that there exists $\delta\in I$ such that
	$$\left\|u^n-\dspD\right\|_{\ell^\infty}<\Rayon, $$
	then the definition \eqref{def:Rayon} of the radius $\Rayon$ implies that $u^n$ belongs to $\Uc^\Z$ and that we can construct the sequence $u^{n+1}:=\Nc(u^n)$. This will give us later on a way to prove recursively that the solution $(u^n)_{n\in\N}$ of the numerical scheme \eqref{Th:defUn} is well-defined up to any time $n\in\N$.
	
	\subsubsection*{Decomposition near the SDSP $\dspD$ of the operator $\Nc$ in linear and nonlinear parts}
	
	Let us consider a choice of $\delta\in I$ and a sequence $h\in \ell^\infty(\Z)$ such that $\left\|h\right\|_{\ell^\infty}<\Rayon$ where the radius $\Rayon$ is defined by \eqref{def:Rayon}. We recall that \eqref{def:Rayon} implies that the elements of the sequence $\dspD+h$ belong to the space of states $\Uc$. Using the definition \eqref{def:evolOpe} of the nonlinear evolution operator $\Nc$, we thus have that for all $j\in\Z$:
	\begin{multline}\label{Th:EgInterm}
		\Nc(\dspD+h)_j = \dspD_j+h_j \\ - \nug\left(F\left(\nug;\dspD_{j-p+1}+h_{j-p+1}, \hdots, \dspD_{j+q}+h_{j+q}\right) - F\left(\nug;\dspD_{j-p}+h_{j-p}, \hdots, \dspD_{j+q-1}+h_{j+q-1}\right)\right).
	\end{multline}
	Let us now introduce the sequence $\QD(h)$ defined for $\delta \in I$ and $h\in \ell^\infty(\Z)$ such that $\left\|h\right\|_{\ell^\infty}<\Rayon$ by:
	\begin{equation}
		\forall j\in\Z,\quad \QD(h)_j:= \nug F(\nug; \dspD_{j-p}+h_{j-p}, \hdots, \dspD_{j+q-1}+h_{j+q-1}) - \nug F(\nug; \dspD_{j-p}, \hdots, \dspD_{j+q-1}) -\sum_{k=-p}^{q-1} b^\delta_{j,k} h_{j+k}
	\end{equation}
	where the scalars $b^\delta_{j,k}$ defined by \eqref{def:Bjk} are equal to
	$$ b^\delta_{j,k}:= \nug \partial_{u_k}F(\nug; \dspD_{j-p}, \hdots, \dspD_{j+q-1}).$$
	We observe that, since the sequence $\dspD$ is a fixed point of the nonlinear evolution operator $\Nc$ defined by \eqref{def:evolOpe}, the sequence
	$$\left(F(\nug; \dspD_{j-p},\hdots,\dspD_{j+q-1})\right)_{j\in\Z}$$
	is constant. Thus, the equality \eqref{Th:EgInterm} can be rewritten as:
	\begin{equation}\label{Th:EgInterm1}
		\forall j\in\Z,\quad \Nc(\dspD+h)_j = \dspD_j+(\LdspD h)_j -\QD(h)_{j+1}+\QD(h)_j,
	\end{equation}
	where the operator $\LdspD$ is defined by \eqref{def:linearizedScheme} and corresponds to the linearization of $\Nc$ about the SDSP $\dspD$. Recalling that the shift operator $\Tc$ is defined by \eqref{def:Tc}, the equality \eqref{Th:EgInterm1} implies that:
	\begin{equation}\label{decompoNcLcc}
		\forall \delta\in I,\forall h\in\ell^\infty(\Z),\quad \left\|h\right\|_{\ell^\infty} <\Rayon \quad \Rightarrow\quad \Nc(\dspD+h) = \dspD +\LdspD h + (Id-\Tc)\QD(h).
	\end{equation}
	
	The sequence $\QD(h)$ should be thought of as a nonlinear quadratic remainder term. Indeed, the following lemma, which will be proved in the Appendix of the paper (Section \ref{sec:Appendix}), allows us to obtain sharp and useful bounds for the sequences $\QD(h)$. We recall that the vector spaces $\ell^1_\gamma$ and $\ell^\infty_\gamma$ for $\gamma\in[0,+\infty[ $ are defined by \eqref{def:EspAPoids}.
	\begin{lemma}\label{lem:inQ}
		Let us consider two constants $\gamma_1,\gamma_\infty\in[0,+\infty[$. 
		\begin{subequations}
			\begin{itemize}
				\item There exists a constant $C_{Q,1}(\gamma_1,\gamma_\infty)>0$ such that for any $\delta \in I$ and sequence $h\in \ell^1_{\gamma_1}\cap\ell^\infty_{\gamma_\infty}$ which verifies:
				$$\left\|h\right\|_{\ell^\infty}<\Rayon,$$
				then the sequence $\QD(h)$ belongs to $\ell^1_{\gamma_1+\gamma_\infty}$ and:
				\begin{equation}
					\left\|\QD(h)\right\|_{\ell^1_{\gamma_1+\gamma_\infty}} \leq C_{Q,1}(\gamma_1,\gamma_\infty)\left\|h\right\|_{\ell^1_{\gamma_1}}\left\|h\right\|_{\ell^\infty_{\gamma_\infty}}.\label{lem:inQ:ell1}
				\end{equation}
				\item There exists a constant $C_{Q,\infty}(\gamma_\infty)>0$ such that for any $\delta \in I$ and sequence $h\in \ell^\infty_{\gamma_\infty}$ which verifies:
				$$\left\|h\right\|_{\ell^\infty}<\Rayon,$$
				then the sequence $\QD(h)$ belongs to $\ell^\infty_{2\gamma_\infty}$ and:
				\begin{equation}
					\left\|\QD(h)\right\|_{\ell^\infty_{2\gamma_\infty}}\leq C_{Q,\infty}(\gamma_\infty){\left\|h\right\|_{\ell^\infty_{\gamma_\infty}}}^2.\label{lem:inQ:ellinfty}
				\end{equation}
			\end{itemize}
		\end{subequations}
	\end{lemma}
	The introduction of the sequences $\QD(h)$ and of the estimates of Lemma \ref{lem:inQ} will be central in the proof of Theorem \ref{Th}. Indeed, let us now present some formal calculations before starting the proof of Theorem \ref{Th}. We consider an initial perturbation $\hg\in\ell^1(\Z)$ such that there exists a constant $\delta\in I$ that satisfies:
	\begin{equation}\label{eg:IDTemp}
		\sum_{j\in\Z}\hg_j = M(\delta).
	\end{equation}
	Let us assume that the solution $(u^n)_{n\in\N}$ of the numerical scheme \eqref{Th:defUn} for the initial condition $u^0:=\dsp+\hg$ is well-defined for all $n\in\N$ (which, we recall, is \textit{nontrivial} and would need to be proved). Assume furthermore that, for all $n\in\N$, the sequence $h^n:=u^n-\dspD$ satisfies:
	$$\left\|h^n\right\|_{\ell^\infty}<\Rayon.$$
	Then, using \eqref{decompoNcLcc}, the equality :
	$$\forall n\in\N,\quad u^{n+1}=\Nc(u^n)$$
	can be rewritten as:
	$$\forall n\in\N,\quad h^{n+1}=\LdspD h^n +(Id-\Tc)\QD(h^n)$$
	and thus
	$$\forall n\in\N,\quad h^{n+1}=\Ldsp h^n +\left(\LdspD-\Ldsp\right)h^n+(Id-\Tc)\QD(h^n)$$
	where the operator $\Ldsp:=L^0$ is defined by \eqref{def:Ldsp}. Therefore, Duhamel's formula implies the following expression for the sequences $h^n$:
	$$\forall n\in\N,\quad h^n = \Ldsp^n h^0 + \sum_{m=0}^{n-1}\Ldsp^{n-1-m}\left(\LdspD-\Ldsp\right)h^m +\sum_{m=0}^{n-1}\Ldsp^{n-1-m}(Id-\Tc)\QD(h^m).$$
	It is quite apparent in this expression how one could use the estimates of Proposition \ref{prop:Est} on the families of operators $(\Ldsp^n)_{n\in\N}$, $(\Ldsp^n(Id-\Tc))_{n\in\N}$ and $(\Ldsp^n\left(\LdspD-\Ldsp\right))_{n\in\N}$ as well as the estimates of Lemma \ref{lem:inQ} on the sequences $\QD(h^n)$ to hopefully obtain decay estimates on the sequences $h^n$. We also observe that the condition \eqref{eg:IDTemp} on the initial perturbation $\hg$ implies that the sequence $h^0$ has a null mass, i.e. it satisfies: 
	$$\sum_{j\in\Z}h^0_j=0.$$
	This is central for the use of the estimates \eqref{prop:Est:L^n:1} and \eqref{prop:Est:L^n:infty} on the semigroup $(\Lcc^n)_{n\in\N}$. The proof of Theorem \ref{Th} will essentially use the calculations above while taking into account that we need to prove the definition of the solution $(u^n)_{n\in\N}$.
	
	\subsubsection*{A useful technical lemma}
	
	We finally introduce a useful technical lemma that will be used in the proof of Theorem \ref{Th}. It is a discrete version of \cite[Lemma 2.3]{Xin1992}.
	\begin{lemma}\label{lem:InSum}
		We consider a triplet of positive constants $(a,b,c)\in[0,+\infty[^3$ that satisfies the condition \eqref{cond:H}. There exists a constant $C_I(a,b,c)>0$ such that, for all $n\in\N$, we have that:
		$$\sum_{m=0}^{\lfloor\frac{n+1}{2}\rfloor}\frac{1}{(m+1)^a(n+1-m)^b} \leq\frac{C_I(a,b,c)}{(n+2)^c}. $$ 
	\end{lemma}
	The proof is quite immediate and will be given in the Appendix (Section \ref{sec:Appendix}).
	
	\subsection{Definition of the constants \texorpdfstring{$C_0$ and $\varepsilon$}{C0 and epsilon} appearing in Theorem \texorpdfstring{\ref{Th}}{1}}
	
	We start the proof of Theorem \ref{Th} by fixing some of the constants that will appear. As stated in Theorem \ref{Th}, we consider four constants $\gamma_1,\gamma_\infty,\pg_1,\pg_\infty\in[0,+\infty[$  which verify the conditions \ref{Th:C1}-\ref{Th:C4}. We then also define 
	\begin{equation}\label{def:GammaPreuve}
		\Gamma:= \max\left(\pg_1+\gamma_1,\pg_\infty+\gamma_\infty -\frac{1}{2\mu},\pg_\infty,\gamma_\infty\right).
	\end{equation}
	
	Let us first observe that, using the inequality \eqref{DSP_CV_ExpoDelta}, there exists a constant $C>0$ such that:
	$$\forall \delta \in I,\quad \left\|\dsp-\dspD\right\|_{\ell^1_\Gamma}\leq C|\delta|.$$
	As a consequence, using \eqref{in:M}, there exists a constant $C_M>0$ such that:
	\begin{equation}\label{in:CM}
		\forall \delta \in I,\quad \left\|\dsp-\dspD\right\|_{\ell^1_\Gamma}\leq C_M|M(\delta)|.
	\end{equation}
	Since we will later on choose $\delta\in I$ such that:
	$$M(\delta)=\sum_{j\in\Z}\hg_j,$$
	inequality \eqref{in:CM} will allow us to bound the difference $\dsp-\dspD$ by the $\ell^1$-norm of the initial perturbation $\hg$.
	
	We now introduce the constant:
	\begin{equation}\label{def:C0}
		C_0:= \left(\max(C_\Lcc(\gamma_1,\Gamma),C_\Lcc(\gamma_\infty,\Gamma)) +1\right)\left(C_M+1\right)
	\end{equation}
	where the constant $C_\Lcc(\cdot,\cdot)$ is defined via Proposition \ref{prop:Est} and the constant $C_M$ is defined by \eqref{in:CM}. This constant $C_0$ is a choice that can correspond to the one appearing in Theorem \ref{Th}. There remains to define the parameter $\varepsilon>0$ that appears in Theorem \ref{Th}. We consider $\varepsilon$ to be small enough to satisfy some conditions \eqref{set:ImM}, \eqref{in:varepsDelta} and \eqref{cond:C1C2} stated below:
	
	$\bullet$ Since the function $M$ defined by \eqref{def:M} is continuous on the open interval $I$ and $M(0)=0$, we observe that $0$ belongs to the interior of the set $M(I)$. We assume that $\varepsilon$ is small enough so that:
	\begin{equation}\label{set:ImM}
		]-\varepsilon,\varepsilon[\subset M(I).
	\end{equation}
		
	$\bullet$ We consider that $\varepsilon$ to be small enough so that the following inequality is verified, where the radius $\Rayon$ is defined by \eqref{def:Rayon}:
	\begin{equation}\label{in:varepsDelta}
		C_0\varepsilon<\Rayon.
	\end{equation}
	
	$\bullet$ We recall that the constants $C_\Lcc(\cdot,\cdot)$ and $C_L(\cdot,\cdot)$ appear in Proposition \ref{prop:Est}, the constants $C_{Q,1}(\cdot,\cdot)$ and $C_{Q,\infty}(\cdot)$ in Lemma \ref{lem:inQ}, the constant $C_I(\cdot,\cdot,\cdot)$ appears in Lemma \ref{lem:InSum} and the constants $C_\delta$, $C_M$, $C_0$ are respectively defined by \eqref{in:M}, \eqref{in:CM} and \eqref{def:C0}. Let us now notate some combinations of those constants that will appear later on in the proof of Theorem \ref{Th}. The most important constants below are $C_1$ and $C_2$ defined respectively by \eqref{def:C1} and \eqref{def:C2}. Indeed, they will appear on the condition \eqref{cond:C1C2} below that we will impose on the constant $\varepsilon$. 
	
	We start be introducing a first set of positive constants $C_1^A$, $C_1^B$, $C_1^C$, $C_1^D$ and $C_1$:
	\begin{subequations}
		\begin{align}
			C_1^A&:=C_0C_L(\gamma_1,2+\pg_1)C_\delta C_I(\pg_1,2+\pg_1,\pg_1) ,\label{def:C1A}\\
			C_1^B&:= C_0C_L(\gamma_1,2+\pg_1)C_\delta  C_I(2+\pg_1,\pg_1,\pg_1), \label{def:C1B}\\
			C_1^C&:= C_\Lcc(\gamma_1,\gamma_1+\gamma_\infty) C_{Q,1}\left(\gamma_1,\gamma_\infty\right){C_0}^2C_I\left(\pg_1+\pg_\infty,\gamma_\infty+\frac{1}{2\mu},\pg_1\right),\label{def:C1C}\\
			C_1^D&:=C_\Lcc(\gamma_1,\gamma_1+\gamma_\infty) C_{Q,1}\left(\gamma_1,\gamma_\infty\right){C_0}^2C_I\left(\gamma_\infty+\frac{1}{2\mu},\pg_1+\pg_\infty,\pg_1\right) ,\label{def:C1D}\\
			C_1&:=C_1^A+C_1^B+C_1^C+C_1^D.\label{def:C1}
		\end{align}
	\end{subequations}
	
	We then introduce the constants $C_2^A$ and $C_2^B$ defined by:
	\begin{align}
		C_2^A &:=C_0C_L(\gamma_\infty,2+\pg_\infty)C_\delta C_I(\pg_\infty,2+\pg_\infty,\pg_\infty) ,\label{def:C2A}\\ 
		C_2^B &:= C_0C_L(\gamma_\infty,2+\pg_\infty)C_\delta  C_I(2+\pg_\infty,\pg_\infty,\pg_\infty),\label{def:C2B}
	\end{align} 
	and then consider two positive constants $C_2^C$ and $C_2^D$ large enough so that they verify the following inequalities \eqref{def:C2CD}. We point out that that the conditions stated before the inequalities are linked to the conditions \ref{Th:C3} and \ref{Th:C4} we imposed on the constants $\gamma_1,\gamma_\infty,\pg_1,\pg_\infty$.
	\begin{subequations}\label{def:C2CD}
		\begin{itemize}
			\item If the triplet $\left(\pg_1+\pg_\infty,\gamma_1+\frac{1}{2\mu}+\min\left(\gamma_\infty,\frac{1}{2\mu}\right),\pg_\infty\right)$ verifies \eqref{cond:H}, then:
			\begin{equation}
				C_2^C  \geq C_\Lcc(\gamma_\infty,\gamma_1+\gamma_\infty) C_{Q,1}\left(\gamma_1,\gamma_\infty\right){C_0}^2C_I\left(\pg_1+\pg_\infty,\gamma_1+\frac{1}{2\mu}+\min\left(\gamma_\infty,\frac{1}{2\mu}\right),\pg_\infty\right). \label{def:C2C:1}
			\end{equation}
			\item If the triplet $\left(2\pg_\infty,\gamma_\infty+\min\left(\gamma_\infty,\frac{1}{2\mu}\right),\pg_\infty\right)$ verifies \eqref{cond:H}, then:
			\begin{equation}
				C_2^C  \geq C_\Lcc(\gamma_\infty,2\gamma_\infty) C_{Q,\infty}\left(\gamma_\infty\right){C_0}^2C_I\left(2\pg_\infty,\gamma_\infty+\min\left(\gamma_\infty,\frac{1}{2\mu}\right),\pg_\infty\right). \label{def:C2C:2}
			\end{equation}
			\item If the triplet $\left(\gamma_1+\frac{1}{2\mu}+\min\left(\gamma_\infty,\frac{1}{2\mu}\right),\pg_1+\pg_\infty,\pg_\infty\right)$ verifies \eqref{cond:H}, then:
			\begin{equation}
				C_2^D  \geq C_\Lcc(\gamma_\infty,\gamma_1+\gamma_\infty) C_{Q,1}\left(\gamma_1,\gamma_\infty\right){C_0}^2C_I\left(\gamma_1+\frac{1}{2\mu}+\min\left(\gamma_\infty,\frac{1}{2\mu}\right),\pg_1+\pg_\infty,\pg_\infty\right). \label{def:C2D:1}
			\end{equation}
			\item If the triplet $\left(\gamma_\infty+\min\left(\gamma_\infty,\frac{1}{2\mu}\right),2\pg_\infty,\pg_\infty\right)$ verifies \eqref{cond:H}, then:
			\begin{equation}
				C_2^D  \geq C_\Lcc(\gamma_\infty,2\gamma_\infty) C_{Q,\infty}\left(\gamma_\infty\right){C_0}^2C_I\left(\gamma_\infty+\min\left(\gamma_\infty,\frac{1}{2\mu}\right),2\pg_\infty,\pg_\infty\right).\label{def:C2D:2}
			\end{equation}
		\end{itemize}
	\end{subequations}

	We then finally define the positive constant $C_2$:
	\begin{equation}\label{def:C2}
		C_2:=   C_2^A+C_2^B+C_2^C+C_2^D.
	\end{equation}
	The condition we impose on the parameter $\varepsilon$ is that:
	\begin{equation}\label{cond:C1C2}
		\varepsilon C_1<C_M+1\quad \text{ and } \quad \varepsilon C_2<C_M+1.
	\end{equation}
	In particular, using the definition \eqref{def:C0} of the constant $C_0$, we observe that \eqref{cond:C1C2} implies:
	\begin{subequations}
		\begin{align}
			C_\Lcc(\gamma_1,\Gamma) \left(C_M+1\right) + \varepsilon C_1 & <C_0,\label{in:varepsC1}\\
			C_\Lcc(\gamma_\infty,\Gamma)\left(C_M+1\right)  + \varepsilon C_2 & <C_0.\label{in:varepsC2}
		\end{align}
	\end{subequations}
	
	\subsection{Proof of Theorem \ref{Th} by induction}
	
	Let us now start with the proof of Theorem \ref{Th}. We consider an initial perturbation $\hg\in\ell^1_{\Gamma}$ where $\Gamma$ is defined by \eqref{def:GammaPreuve} such that:
	\begin{equation}\label{in:hg}
		\left\|\hg\right\|_{\ell^1_{\Gamma}}<\varepsilon.
	\end{equation}
	We observe \eqref{in:hg} implies that:
	\begin{equation}\label{in:hgSuite}
		\left|\sum_{j\in\Z}\hg_j\right|\leq\left\|\hg\right\|_{\ell^1}\leq\left\|\hg\right\|_{\ell^1_\Gamma}<\varepsilon.
	\end{equation}
	Using \eqref{set:ImM}, there exists $\delta \in I$ such that:
	\begin{equation}\label{eg:IDdelta}
		\sum_{j\in\Z}\hg_j=M(\delta).
	\end{equation}
	Hypothesis \ref{H:identification} also implies that the function $M$ is injective and thus that the choice of $\delta\in I$ is actually unique. Let us also observe that \eqref{in:M}, \eqref{eg:IDdelta} and \eqref{in:hgSuite} imply:
	\begin{equation}\label{in:delta}
		|\delta|< C_\delta \varepsilon.
	\end{equation}
	
	We then define the initial condition $u^0:=\dsp+\hg$. The proof of \eqref{Th} will be done by induction. We state for $n\in\N$ the following assertion $\Pg(n)$:
	
	\vspace{0.3cm}
	\begin{center}
		\noindent\fbox{\begin{minipage}{0.93\textwidth}
				\noindent \textbf{Assertion }$\Pg(n)$: The sequences $u^m$ constructed using the numerical scheme \eqref{Th:defUn} are well-defined for all $m\in\lbrace0,\hdots,n\rbrace$. Furthermore, if we define the sequences $h^m:=u^m-\dspD$ for $m\in\lbrace0,\hdots,n\rbrace$, then we have that:
				\begin{subequations}
					\begin{align}
						\forall m\in\lbrace0,\hdots,n\rbrace,\quad & \left\|h^m\right\|_{\ell^1_{\gamma_1}}\leq \frac{C_0}{(m+1)^{\pg_1}}\left\|\hg\right\|_{\ell^1_{\Gamma}},\label{Recu:1}\\
						\forall m\in\lbrace0,\hdots,n\rbrace,\quad & \left\|h^m\right\|_{\ell^\infty_{\gamma_\infty}}\leq \frac{C_0}{(m+1)^{\pg_\infty}}\left\|\hg\right\|_{\ell^1_{\Gamma}},\label{Recu:infty}\\
						\forall m\in\lbrace0,\hdots,n\rbrace,\quad & \left\|h^m\right\|_{\ell^\infty}<\Rayon,\label{Recu:delta}
					\end{align}
				\end{subequations}
				where the constant $C_0$ is defined by \eqref{def:C0} and the radius $\Rayon$ by \eqref{def:Rayon}.
		\end{minipage}}
	\end{center}
	
	\vspace{0.3cm}\underline{$\blacktriangleright$ \textbf{Initialization step:}} 
	
	The sequence $u^0$ is obviously well-defined. To prove $\Pg(0)$, we therefore have to prove \eqref{Recu:1}-\eqref{Recu:delta} for $m=0$. We begin by observing that:
	\begin{equation}\label{eg:h0}
		h^0 := u^0- \dspD = \dsp-\dspD+ \hg.
	\end{equation}
	We observe that using \eqref{in:CM} and \eqref{eg:IDdelta}, we have:
	$$\left\|\dsp-\dspD\right\|_{\ell^1_\Gamma} \leq C_M|M(\delta)|\leq C_M\left\|\hg\right\|_{\ell^1}\leq  C_M\left\|\hg\right\|_{\ell^1_\Gamma}. $$
	Therefore, \eqref{eg:h0} implies that the sequence $h^0$ belongs to $\ell^1_{\Gamma}$ and:
	\begin{equation}\label{in:h0}
		\left\|h^0\right\|_{\ell^1_\Gamma}\leq\left(C_M+1\right)\left\|\hg\right\|_{\ell^1_\Gamma}.
	\end{equation}
	Furthermore, we observe that \eqref{eg:IDdelta} and \eqref{eg:h0} imply:
	\begin{equation}\label{eg:Masseh0}
		\sum_{j\in\Z}h^0_j=0.
	\end{equation}
	Thus, the sequence $h^0$ actually belongs to the vector space $\E_\Gamma$ of zero-mass sequences of $\ell^1_\Gamma$ (see \eqref{def:Elamb}).
	
	Let us now start by proving \eqref{Recu:1} and \eqref{Recu:infty} for $m=0$. Using the definition \eqref{def:GammaPreuve} of $\Gamma$, we notice that $\Gamma\geq \gamma_1$ and $\Gamma\geq \gamma_\infty$. We then observe using \eqref{in:h0} as well as the definition \eqref{def:C0} of the constant $C_0$:
	$$\left\|h^0\right\|_{\ell^1_{\gamma_1}}\leq\left\|h^0\right\|_{\ell^1_\Gamma}\leq (C_M+1)\left\|\hg\right\|_{\ell^1_\Gamma}\leq C_0\left\|\hg\right\|_{\ell^1_\Gamma}$$
	and:
	$$\left\|h^0\right\|_{\ell^\infty_{\gamma_\infty}}\leq\left\|h^0\right\|_{\ell^1_\Gamma}\leq (C_M+1)\left\|\hg\right\|_{\ell^1_\Gamma}\leq C_0\left\|\hg\right\|_{\ell^1_\Gamma}.$$
	This implies \eqref{Recu:1} and \eqref{Recu:infty} for $m=0$. 
	
	We finally conclude the initialization by observing that \eqref{Recu:infty} for $m=0$, \eqref{in:hg} and \eqref{in:varepsDelta} imply that:
	$$\left\|h^0\right\|_{\ell^\infty}\leq\left\|h^0\right\|_{\ell^\infty_{\gamma_\infty}}\leq C_0\left\|\hg\right\|_{\ell^1_\Gamma}<C_0\varepsilon<\Rayon.$$
	This implies \eqref{Recu:delta} for $m=0$. This concludes the proof of $\Pg(0)$.
	
	\underline{$\blacktriangleright$ \textbf{Induction step:}}

	We consider $n\in\N$ such that $\Pg(n)$ is true. Let us prove that $\Pg(n+1)$ is also true. First, using \eqref{Recu:delta} for $m=n$ and the definition \eqref{def:Rayon} of the radius $\Rayon$, we have that the sequence $u^n= \dspD+h^n$ belongs to $\Uc^\Z$ and we can thus define:
	$$u^{n+1}:=\Nc(u^n).$$
	Thus, the sequence $u^{n+1}$ is well-defined. Therefore, we can from now on consider the sequence $h^{n+1}:=u^{n+1}-\dspD$. 
	
	To prove that $\Pg(n+1)$ is true, there just remains to prove the inequality \eqref{Recu:1}, \eqref{Recu:infty} and  \eqref{Recu:delta} for $m=n+1$. Before starting with the proofs of \eqref{Recu:1}-\eqref{Recu:delta} for $m=n+1$, we will need to make a slight observation. The inequality \eqref{Recu:delta} for $m\in\lbrace0,\hdots,n\rbrace$ implies that we can use the equality \eqref{decompoNcLcc} to rewrite:
	$$\forall m\in\lbrace0,\hdots,n\rbrace,\quad u^{m+1}=\Nc(u^m)$$
	as:
	$$\forall m\in\lbrace0,\hdots,n\rbrace,\quad  h^{m+1}=\LdspD h^m +(Id-\Tc)\QD(h^m)$$
	and thus:
	$$\forall m\in\lbrace0,\hdots,n\rbrace,\quad  h^{m+1}=\Ldsp h^m +\left(\LdspD-\Ldsp\right)h^m +(Id-\Tc)\QD(h^m)$$
	Using Duhamel's formula, we then have that:
	\begin{equation}\label{expression:hn+1}
		h^{n+1} = \Lcc^{n+1}h^0   + \sum_{m=0}^n\Lcc^{n-m} \left(\LdspD-\Ldsp\right) h^m + \sum_{m=0}^n\Lcc^{n-m} (Id-\Tc) \QD(h^m).
	\end{equation}
	This expression of the sequence $h^{n+1}$ will be central to prove \eqref{Recu:1} and \eqref{Recu:infty} for $m=n+1$.
	
	We will start by proving \eqref{Recu:1} for $m=n+1$. We will then focus on the proof of \eqref{Recu:infty} for $m=n+1$ which will be fairly similar. Finally, we will conclude with the proof of \eqref{Recu:delta} for $m=n+1$ which is actually a consequence of \eqref{Recu:infty} (or even \eqref{Recu:1}) for $m=n+1$.
	
	\vspace{0.2cm}\noindent\underline{$\bullet$ \textbf{Proof of \eqref{Recu:1} for $m=n+1$:}}
	
	We want to find bounds on the sequence $h^{n+1}$ in $\ell^1_{\gamma_1}$. Our goal is to prove the following bounds \eqref{Recu:1:PR} on the different terms appearing in the expression \eqref{expression:hn+1} of the sequence $h^{n+1}$, where the constants $C_1^A$, $C_1^B$, $C_1^C$ and $C^1_D$ are defined by \eqref{def:C1A}-\eqref{def:C1D}:
	\begin{subequations}\label{Recu:1:PR}
		\begin{align}
			\left\|\Ldsp^{n+1}h^0\right\|_{\ell^1_{\gamma_1}} & \leq \frac{C_\Lcc(\gamma_1,\Gamma)(C_M+1)}{(n+2)^{\pg_1}}\left\|\hg\right\|_{\ell^1_{\Gamma}}, \label{Recu:1:PR1}\\
			\left\|\sum_{m=0}^{\lfloor\frac{n+1}{2}\rfloor}\Ldsp^{n-m}\left(\LdspD-\Ldsp\right)h^m\right\|_{\ell^1_{\gamma_1}} &\leq \frac{\varepsilon C_1^{A}}{(n+2)^{\pg_1}}\left\|\hg\right\|_{\ell^1_{\Gamma}},\label{Recu:1:PR2}\\
			\left\|\sum_{m=\lfloor\frac{n+1}{2}\rfloor+1}^{n}\Ldsp^{n-m}\left(\LdspD-\Ldsp\right)h^m\right\|_{\ell^1_{\gamma_1}} &\leq \frac{\varepsilon C_1^{B}}{(n+2)^{\pg_1}}\left\|\hg\right\|_{\ell^1_{\Gamma}},\label{Recu:1:PR3}\\
			\left\|\sum_{m=0}^{\lfloor\frac{n+1}{2}\rfloor}\Ldsp^{n-m}(Id-\Tc)\QD(h^m)\right\|_{\ell^1_{\gamma_1}} &\leq \frac{\varepsilon C_1^{C}}{(n+2)^{\pg_1}}\left\|\hg\right\|_{\ell^1_{\Gamma}},\label{Recu:1:PR4}\\
			\left\|\sum_{m=\lfloor\frac{n+1}{2}\rfloor+1}^{n}\Ldsp^{n-m}(Id-\Tc)\QD(h^m)\right\|_{\ell^1_{\gamma_1}} &\leq \frac{\varepsilon C_1^{D}}{(n+2)^{\pg_1}}\left\|\hg\right\|_{\ell^1_{\Gamma}}.\label{Recu:1:PR5}
		\end{align}
	\end{subequations}
	We observe that once the inequalities \eqref{Recu:1:PR} will have been proved, using the equality \eqref{expression:hn+1} on the sequence $h^{n+1}$, we have that:
	\begin{align*}
		\left\|h^{n+1}\right\|_{\ell^1_{\gamma_1}} \leq& \left\|\Lcc^{n+1}h^0\right\|_{\ell^1_{\gamma_1}}+\left\|\sum_{m=0}^{\lfloor\frac{n+1}{2}\rfloor}\Ldsp^{n-m}\left(\LdspD-\Ldsp\right)h^m\right\|_{\ell^1_{\gamma_1}}+\left\|\sum_{m=\lfloor\frac{n+1}{2}\rfloor+1}^{n}\Ldsp^{n-m}\left(\LdspD-\Ldsp\right)h^m\right\|_{\ell^1_{\gamma_1}}\\
		&+\left\|\sum_{m=0}^{\lfloor\frac{n+1}{2}\rfloor}\Ldsp^{n-m}(Id-\Tc)\QD(h^m)\right\|_{\ell^1_{\gamma_1}}+ \left\|\sum_{m=\lfloor\frac{n+1}{2}\rfloor+1}^{n}\Ldsp^{n-m}(Id-\Tc)\QD(h^m)\right\|_{\ell^1_{\gamma_1}}\\ \leq&\frac{C_\Lcc(\gamma_1,\Gamma)(C_M+1)+ \varepsilon C_1}{(n+2)^{\pg_1}}\left\|\hg\right\|_{\ell^1_{\Gamma}}
	\end{align*}
	where the constant $C_1$ is defined by \eqref{def:C1}. Using the condition \eqref{in:varepsC1}, we will thus obtain \eqref{Recu:1} for $m=n+1$. Let us now prove the inequalities \eqref{Recu:1:PR}.
	
	\underline{$\star$ \textit{Estimate \eqref{Recu:1:PR1} on $\Lcc^{n+1}h^0$ in $\ell^1_{\gamma_1}$:} }
	
	We recall that we proved in the initialization step that the sequence $h^0$ belongs to $\E_\Gamma$. Using the estimate \eqref{prop:Est:L^n:1} of Proposition \ref{prop:Est}, the estimate \eqref{in:h0} on the sequence $h^0$ and using the definition \eqref{def:GammaPreuve} of the constant $\Gamma$ which implies that $\Gamma\geq \pg_1+\gamma_1$, we have that:
	\begin{equation*}
		\left\|\Lcc^{n+1}h^0\right\|_{\ell^1_{\gamma_1}} \leq\frac{C_\Lcc(\gamma_1,\Gamma)}{(n+2)^{\Gamma-\gamma_1}}\left\|h^0\right\|_{\ell^1_{\Gamma}}\leq \frac{C_\Lcc(\gamma_1,\Gamma)}{(n+2)^{\pg_1}}\left\|h^0\right\|_{\ell^1_{\Gamma}}\leq \frac{C_\Lcc(\gamma_1,\Gamma)(C_M+1)}{(n+2)^{\pg_1}}\left\|\hg\right\|_{\ell^1_{\Gamma}}.
	\end{equation*}
	
	\underline{$\star$ \textit{Estimate \eqref{Recu:1:PR2} on $\sum_{m=0}^{\lfloor\frac{n+1}{2}\rfloor}\Ldsp^{n-m}\left(\LdspD-\Ldsp\right)h^m$ in $\ell^1_{\gamma_1}$:} }
	
	Using the estimate \eqref{prop:Est:L^n(LD-L)} of Proposition \ref{prop:Est}, the estimate \eqref{in:delta} on $\delta$ and the estimate \eqref{Recu:1} for $m\in \left\{0,\hdots, \lfloor\frac{n+1}{2}\rfloor\right\}$, we have that:
	\begin{align}
		\begin{split}
			\left\|\sum_{m=0}^{\lfloor\frac{n+1}{2}\rfloor}\Ldsp^{n-m}\left(\LdspD-\Ldsp\right)h^m\right\|_{\ell^1_{\gamma_1}} &\leq \sum_{m=0}^{\lfloor\frac{n+1}{2}\rfloor} \frac{C_L(\gamma_1,2+\pg_1)|\delta|}{(n+1-m)^{2+\pg_1}}\underset{\leq\left\|h^m\right\|_{\ell^1_{\gamma_1}}}{\underbrace{\left\|h^m\right\|_{\ell^\infty}}}\\
			& \leq \varepsilon C_0C_L(\gamma_1,2+\pg_1)C_\delta\left(\sum_{m=0}^{\lfloor\frac{n+1}{2}\rfloor}\frac{1}{(m+1)^{\pg_1}(n+1-m)^{2+\pg_1}}\right)\left\|\hg\right\|_{\ell^1_\Gamma}.
		\end{split}\label{Recu:1:PR2:Interm1}
	\end{align}
	The triplet $\left(\pg_1,2+\pg_1,\pg_1\right)$ verifies the condition \eqref{cond:H}. Thus, using Lemma \ref{lem:InSum}, we have that:
	\begin{equation}\label{Recu:1:PR2:Interm2}
		\sum_{m=0}^{\lfloor\frac{n+1}{2}\rfloor} \frac{1}{(m+1)^{\pg_1}(n+1-m)^{2+\pg_1}} \leq \frac{C_I\left(\pg_1,2+\pg_1,\pg_1\right)}{(n+2)^{\pg_1}}.
	\end{equation}
	Thus, combining  \eqref{Recu:1:PR2:Interm1} and \eqref{Recu:1:PR2:Interm2} and recalling the constant $C_1^A$ is defined by \eqref{def:C1A}, we have proved \eqref{Recu:1:PR2}.
	
	\underline{$\star$ \textit{Estimate \eqref{Recu:1:PR3} on $\sum_{m=\lfloor\frac{n+1}{2}\rfloor+1}^{n}\Ldsp^{n-m}\left(\LdspD-\Ldsp\right)h^m$ in $\ell^1_{\gamma_1}$:} }
	
	The proof of \eqref{Recu:1:PR3} is similar to the proof of \eqref{Recu:1:PR2}. Using a similar proof as for \eqref{Recu:1:PR2:Interm1}, we have that:
	\begin{equation}\label{Recu:1:PR3:Interm1}
		\left\|\sum_{m=\lfloor\frac{n+1}{2}\rfloor+1}^{n}\Ldsp^{n-m}\left(\LdspD-\Ldsp\right)h^m\right\|_{\ell^1_{\gamma_1}}  \leq \varepsilon C_0C_L(\gamma_1,2+\pg_1)C_\delta\left(\sum_{m=\lfloor\frac{n+1}{2}\rfloor+1}^{n}\frac{1}{(m+1)^{\pg_1}(n+1-m)^{2+\pg_1}}\right)\left\|\hg\right\|_{\ell^1_\Gamma}.
	\end{equation}
	The triplet $\left(2+\pg_1,\pg_1,\pg_1\right)$ verifies the condition \eqref{cond:H}. Thus using Lemma \ref{lem:InSum}, we have that:
	\begin{equation}\label{Recu:1:PR3:Interm2}
		\sum_{m=\lfloor\frac{n+1}{2}\rfloor+1}^{n} \frac{1}{(m+1)^{\pg_1}(n+1-m)^{2+\pg_1}} \leq \frac{C_I\left(2+\pg_1,\pg_1,\pg_1\right)}{(n+2)^{\pg_1}}.
	\end{equation}
	Thus, combining  \eqref{Recu:1:PR3:Interm1} and \eqref{Recu:1:PR3:Interm2} and recalling the constant $C_1^B$ is defined by \eqref{def:C1B}, we have proved \eqref{Recu:1:PR3}.
	
	\underline{$\star$ \textit{Estimate \eqref{Recu:1:PR4} on $\sum_{m=0}^{\lfloor\frac{n+1}{2}\rfloor}\Ldsp^{n-m}(Id-\Tc)\QD(h^m)$ in $\ell^1_{\gamma_1}$:} }
	
	Using the estimate \eqref{prop:Est:L^n(id-T):1,1} of Proposition \ref{prop:Est}, we have that:
	\begin{equation}\label{Recu:1:PR4:Interm1}
		\left\|\sum_{m=0}^{\lfloor\frac{n+1}{2}\rfloor}\Ldsp^{n-m}(Id-\Tc)\QD(h^m)\right\|_{\ell^1_{\gamma_1}} \leq \sum_{m=0}^{\lfloor\frac{n+1}{2}\rfloor} \frac{C_\Lcc(\gamma_1,\gamma_1+\gamma_\infty)}{(n+1-m)^{\gamma_\infty+\frac{1}{2\mu}}}\left\|\QD(h^m)\right\|_{\ell^1_{\gamma_1+\gamma_\infty}}.
	\end{equation}
	Furthermore, for $m\in\lbrace0,\hdots,n\rbrace$, using the inequality \eqref{lem:inQ:ell1} on the sequence $\QD(h^m)$, the inequalities \eqref{Recu:1} and \eqref{Recu:infty} on the sequence $h^m$ and finally the inequality \eqref{in:hg}, we have that:
	\begin{equation}\label{Recu:1:PR4:Interm2}
		\left\|\QD(h^m)\right\|_{\ell^1_{\gamma_1+\gamma_\infty}}\leq C_{Q,1}\left(\gamma_1,\gamma_\infty\right)\left\|h^m\right\|_{\ell^1_{\gamma_1}}\left\|h^m\right\|_{\ell^\infty_{\gamma_\infty}}\leq \frac{C_{Q,1}\left(\gamma_1,\gamma_\infty\right){C_0}^2}{(m+1)^{\pg_1+\pg_\infty}}{\left\|\hg\right\|_{\ell^1_\Gamma}}^2\leq \frac{\varepsilon C_{Q,1}\left(\gamma_1,\gamma_\infty\right){C_0}^2}{(m+1)^{\pg_1+\pg_\infty}}\left\|\hg\right\|_{\ell^1_\Gamma}.
	\end{equation}
	Combining \eqref{Recu:1:PR4:Interm1} and \eqref{Recu:1:PR4:Interm2}, we have that:
	\begin{multline}\label{Recu:1:PR4:Interm3}
		\left\|\sum_{m=0}^{\lfloor\frac{n+1}{2}\rfloor}\Ldsp^{n-m}(Id-\Tc)\QD(h^m)\right\|_{\ell^1_{\gamma_1}} \\ \leq \varepsilon C_\Lcc(\gamma_1,\gamma_1+\gamma_\infty) C_{Q,1}\left(\gamma_1,\gamma_\infty\right){C_0}^2\left(\sum_{m=0}^{\lfloor\frac{n+1}{2}\rfloor} \frac{1}{(m+1)^{\pg_1+\pg_\infty}(n+1-m)^{\gamma_\infty+\frac{1}{2\mu}}}\right)\left\|\hg\right\|_{\ell^1_\Gamma}.
	\end{multline}
	Since the condition \ref{Th:C1} presented in the statement of Theorem \ref{Th} states that the triplet $\left(\pg_1+\pg_\infty,\gamma_\infty+\frac{1}{2\mu},\pg_1\right)$ verifies the condition \eqref{cond:H}. Thus, using Lemma \ref{lem:InSum}, we have that:
	\begin{equation}\label{Recu:1:PR4:Interm4}
		\sum_{m=0}^{\lfloor\frac{n+1}{2}\rfloor} \frac{1}{(m+1)^{\pg_1+\pg_\infty}(n+1-m)^{\gamma_\infty+\frac{1}{2\mu}}} \leq \frac{C_I\left(\pg_1+\pg_\infty,\gamma_\infty+\frac{1}{2\mu},\pg_1\right)}{(n+2)^{\pg_1}}.
	\end{equation}
	Thus, combining  \eqref{Recu:1:PR4:Interm3} and \eqref{Recu:1:PR4:Interm4} and recalling the constant $C_1^C$ is defined by \eqref{def:C1C}, we have proved \eqref{Recu:1:PR4}.
	
	\underline{$\star$ \textit{Estimate \eqref{Recu:1:PR5} on $\sum_{m=\lfloor\frac{n+1}{2}\rfloor+1}^{n}\Ldsp^{n-m}(Id-\Tc)\QD(h^m)$ in $\ell^1_{\gamma_1}$:} }
	
	The proof of \eqref{Recu:1:PR5} is similar to the proof of \eqref{Recu:1:PR4}. Using a similar proof as for \eqref{Recu:1:PR4:Interm3}, we have that:
	\begin{multline}\label{Recu:1:PR5:Interm1}
		\left\|\sum_{m=\lfloor\frac{n+1}{2}\rfloor+1}^{n}\Ldsp^{n-m}(Id-\Tc)\QD(h^m)\right\|_{\ell^1_{\gamma_1}} \\ \leq \varepsilon C_\Lcc(\gamma_1,\gamma_1+\gamma_\infty) C_{Q,1}\left(\gamma_1,\gamma_\infty\right){C_0}^2\left(\sum_{m=\lfloor\frac{n+1}{2}\rfloor+1}^{n} \frac{1}{(m+1)^{\pg_1+\pg_\infty}(n+1-m)^{\gamma_\infty+\frac{1}{2\mu}}}\right)\left\|h^0\right\|_{\ell^1_\Gamma}.
	\end{multline}
	Since the condition \ref{Th:C2} presented in the statement of Theorem \ref{Th} state that the triplet $\left(\gamma_\infty+\frac{1}{2\mu},\pg_1+\pg_\infty,\pg_1\right)$ verifies the condition \eqref{cond:H}. Thus using Lemma \ref{lem:InSum}, we have that:
	\begin{equation}\label{Recu:1:PR5:Interm2}
		\sum_{m=\lfloor\frac{n+1}{2}\rfloor+1}^{n} \frac{1}{(m+1)^{\pg_1+\pg_\infty}(n+1-m)^{\gamma_\infty+\frac{1}{2\mu}}} \leq \frac{C_I\left(\gamma_\infty+\frac{1}{2\mu},\pg_1+\pg_\infty,\pg_1\right)}{(n+2)^{\pg_1}}.
	\end{equation}
	Thus, combining  \eqref{Recu:1:PR5:Interm1} and \eqref{Recu:1:PR5:Interm2} and recalling the constant $C_1^D$ is defined by \eqref{def:C1D}, we have proved \eqref{Recu:1:PR5}.
	
	This concludes the proof of the inequalities \eqref{Recu:1:PR}. Just as stated right after those inequalities, this allows us to conclude the proof of \eqref{Recu:1} for $m=n+1$.
	
	\vspace{0.2cm}\noindent\underline{$\bullet$ \textbf{Proof of \eqref{Recu:infty} for $m=n+1$:}}
	
	The proof of \eqref{Recu:infty} for $m=n+1$ will be fairly similar to the proof of \eqref{Recu:1} for $m=n+1$. We need to prove the following bounds where the constants $C_2^A$, $C_2^B$, $C_2^C$ and $C_2^D$ are defined by \eqref{def:C2A}, \eqref{def:C2B} and \eqref{def:C2CD}:
	\begin{subequations}\label{Recu:infty:PR}
		\begin{align}
			\left\|\Ldsp^{n+1}h^0\right\|_{\ell^\infty_{\gamma_\infty}} & \leq \frac{C_\Lcc(\gamma_\infty,\Gamma)(C_M+1)}{(n+2)^{\pg_\infty}}\left\|\hg\right\|_{\ell^1_{\Gamma}}, \label{Recu:infty:PR1}\\
			\left\|\sum_{m=0}^{\lfloor\frac{n+1}{2}\rfloor}\Ldsp^{n-m}\left(\LdspD-\Ldsp\right)h^m\right\|_{\ell^\infty_{\gamma_\infty}} & \leq \frac{\varepsilon C_2^A}{(n+2)^{\pg_\infty}}\left\|\hg\right\|_{\ell^1_{\Gamma}},\label{Recu:infty:PR2}\\
			\left\|\sum_{m=\lfloor\frac{n+1}{2}\rfloor+1}^{n}\Ldsp^{n-m}\left(\LdspD-\Ldsp\right)h^m\right\|_{\ell^\infty_{\gamma_\infty}} & \leq \frac{\varepsilon C_2^B}{(n+2)^{\pg_\infty}}\left\|\hg\right\|_{\ell^1_{\Gamma}},\label{Recu:infty:PR3}\\
			\left\|\sum_{m=0}^{\lfloor\frac{n+1}{2}\rfloor}\Ldsp^{n-m}(Id-\Tc)\QD(h^m)\right\|_{\ell^\infty_{\gamma_\infty}} &\leq \frac{\varepsilon C_2^C}{(n+2)^{\pg_\infty}}\left\|\hg\right\|_{\ell^1_{\Gamma}},\label{Recu:infty:PR4}\\
			\left\|\sum_{m=\lfloor\frac{n+1}{2}\rfloor+1}^{n}\Ldsp^{n-m}(Id-\Tc)\QD(h^m)\right\|_{\ell^\infty_{\gamma_\infty}} &\leq \frac{\varepsilon C_2^D}{(n+2)^{\pg_\infty}}\left\|\hg\right\|_{\ell^1_{\Gamma}}.\label{Recu:infty:PR5}
		\end{align}
	\end{subequations}
	We observe that once the inequalities \eqref{Recu:infty:PR} will have been proved, using the equality \eqref{expression:hn+1} on the sequence $h^{n+1}$, we have that:
	\begin{align*}
		\left\|h^{n+1}\right\|_{\ell^\infty_{\gamma_\infty}} \leq& \left\|\Ldsp^{n+1}h^0\right\|_{\ell^\infty_{\gamma_\infty}}+\left\|\sum_{m=0}^{\lfloor\frac{n+1}{2}\rfloor}\Ldsp^{n-m}\left(\LdspD-\Ldsp\right)h^m\right\|_{\ell^\infty_{\gamma_\infty}}+\left\|\sum_{m=\lfloor\frac{n+1}{2}\rfloor+1}^{n}\Ldsp^{n-m}\left(\LdspD-\Ldsp\right)h^m\right\|_{\ell^\infty_{\gamma_\infty}}\\
		&+\left\|\sum_{m=0}^{\lfloor\frac{n+1}{2}\rfloor}\Ldsp^{n-m}(Id-\Tc)\QD(h^m)\right\|_{\ell^\infty_{\gamma_\infty}}+ \left\|\sum_{m=\lfloor\frac{n+1}{2}\rfloor+1}^{n}\Ldsp^{n-m}(Id-\Tc)\QD(h^m)\right\|_{\ell^\infty_{\gamma_\infty}}\\ \leq&\frac{C_\Lcc(\gamma_\infty,\Gamma)(C_M+1)+ \varepsilon C_2}{(n+2)^{\pg_\infty}}\left\|\hg\right\|_{\ell^1_{\Gamma}}
	\end{align*}
	where the constant $C_2$ is defined by \eqref{def:C2}. Using the condition \eqref{in:varepsC2}, we will thus obtain \eqref{Recu:infty} for $m=n+1$. Let us now prove the inequalities \eqref{Recu:infty:PR}.

	\underline{$\star$ \textit{Estimate \eqref{Recu:infty:PR1} on $\Ldsp^{n+1}h^0$ in $\ell^\infty_{\gamma_\infty}$:} }
	
	The definition \eqref{def:GammaPreuve} of the constant $\Gamma$ implies that $\Gamma\geq \gamma_\infty$. This will allow us below to use the estimate \eqref{prop:Est:L^n:infty} of Proposition \ref{prop:Est}. Furthermore, the definition \eqref{def:GammaPreuve} of the constant $\Gamma$ implies that: 
	$$\Gamma\geq \pg_\infty+\gamma_\infty -\min\left(\gamma_\infty,\frac{1}{2\mu}\right).$$
	We also recall that we proved in the initialization step that the sequence $h^0$ belongs to $\E_\Gamma$. Therefore, combining the observations above with the estimate \eqref{in:h0} on the sequence $h^0$, we have that:
	\begin{equation*}
		\left\|\Ldsp^{n+1}h^0\right\|_{\ell^\infty_{\gamma_\infty}} \leq\frac{C_\Lcc(\gamma_\infty,\Gamma)}{(n+2)^{\Gamma-\gamma_\infty				+\min\left(\gamma_\infty,\frac{1}{2\mu}\right)}}\left\|h^0\right\|_{\ell^1_{\Gamma}}\leq \frac{C_\Lcc(\gamma_\infty,\Gamma)}{(n+2)^{\pg_\infty}}\left\|h^0\right\|_{\ell^1_{\Gamma}}\leq \frac{C_\Lcc(\gamma_\infty,\Gamma)(C_M+1)}{(n+2)^{\pg_\infty}}\left\|\hg\right\|_{\ell^1_{\Gamma}}.
	\end{equation*}
	
	\underline{$\star$ \textit{Estimate \eqref{Recu:infty:PR2} on $\sum_{m=0}^{\lfloor\frac{n+1}{2}\rfloor}\Ldsp^{n-m}\left(\LdspD-\Ldsp\right)h^m$ in $\ell^\infty_{\gamma_\infty}$:} }
	
	Using the estimate \eqref{prop:Est:L^n(LD-L)} of Proposition \ref{prop:Est}, the estimate \eqref{in:delta} on $\delta$ and the estimate \eqref{Recu:infty} for $m\in \left\{0,\hdots, \lfloor\frac{n+1}{2}\rfloor\right\}$, we have that:
	\begin{align}
		\begin{split}
			\left\|\sum_{m=0}^{\lfloor\frac{n+1}{2}\rfloor}\Ldsp^{n-m}\left(\LdspD-\Ldsp\right)h^m\right\|_{\ell^\infty_{\gamma_\infty}} & \leq \left\|\sum_{m=0}^{\lfloor\frac{n+1}{2}\rfloor}\Ldsp^{n-m}\left(\LdspD-\Ldsp\right)h^m\right\|_{\ell^1_{\gamma_\infty}}\\
			&\leq \sum_{m=0}^{\lfloor\frac{n+1}{2}\rfloor} \frac{C_L(\gamma_\infty,2+\pg_\infty)|\delta|}{(n+1-m)^{2+\pg_\infty}}\underset{\leq\left\|h^m\right\|_{\ell^\infty_{\gamma_\infty}}}{\underbrace{\left\|h^m\right\|_{\ell^\infty}}}\\
			& \leq \varepsilon C_0C_L(\gamma_\infty,2+\pg_\infty)C_\delta\left(\sum_{m=0}^{\lfloor\frac{n+1}{2}\rfloor}\frac{1}{(m+1)^{\pg_\infty}(n+1-m)^{2+\pg_\infty}}\right)\left\|\hg\right\|_{\ell^1_\Gamma}.
		\end{split}\label{Recu:infty:PR2:Interm1}
	\end{align}
	The triplet $\left(\pg_\infty,2+\pg_\infty,\pg_\infty\right)$ verifies the condition \eqref{cond:H}. Thus, using Lemma \ref{lem:InSum}, we have that:
	\begin{equation}\label{Recu:infty:PR2:Interm2}
		\sum_{m=0}^{\lfloor\frac{n+1}{2}\rfloor} \frac{1}{(m+1)^{\pg_\infty}(n+1-m)^{2+\pg_\infty}} \leq \frac{C_I\left(\pg_\infty,2+\pg_\infty,\pg_\infty\right)}{(n+2)^{\pg_\infty}}.
	\end{equation}
	Thus, combining  \eqref{Recu:infty:PR2:Interm1} and \eqref{Recu:infty:PR2:Interm2} and recalling the constant $C_2^A$ is defined by \eqref{def:C2A}, we have proved \eqref{Recu:infty:PR2}.
	
	\underline{$\star$ \textit{Estimate \eqref{Recu:infty:PR3} on $\sum_{m=\lfloor\frac{n+1}{2}\rfloor+1}^{n}\Ldsp^{n-m}\left(\LdspD-\Ldsp\right)h^m$ in $\ell^\infty_{\gamma_\infty}$:} }
	
	The proof of \eqref{Recu:infty:PR3} is similar to the proof of \eqref{Recu:infty:PR2}. Using a similar proof as for \eqref{Recu:infty:PR2:Interm1}, we have that:
	\begin{multline}\label{Recu:infty:PR3:Interm1}
		\left\|\sum_{m=\lfloor\frac{n+1}{2}\rfloor+1}^{n}\Ldsp^{n-m}\left(\LdspD-\Ldsp\right)h^m\right\|_{\ell^\infty_{\gamma_\infty}} \\  \leq \varepsilon C_0C_L(\gamma_\infty,2+\pg_\infty)C_\delta\left(\sum_{m=\lfloor\frac{n+1}{2}\rfloor+1}^{n}\frac{1}{(m+1)^{\pg_\infty}(n+1-m)^{2+\pg_\infty}}\right)\left\|\hg\right\|_{\ell^1_\Gamma}.
	\end{multline}
	The triplet $\left(2+\pg_\infty,\pg_\infty,\pg_\infty\right)$ verifies the condition \eqref{cond:H}. Thus using Lemma \ref{lem:InSum}, we have that:
	\begin{equation}\label{Recu:infty:PR3:Interm2}
		\sum_{m=\lfloor\frac{n+1}{2}\rfloor+1}^{n} \frac{1}{(m+1)^{\pg_\infty}(n+1-m)^{2+\pg_\infty}} \leq \frac{C_I\left(2+\pg_\infty,\pg_\infty,\pg_\infty\right)}{(n+2)^{\pg_\infty}}.
	\end{equation}
	Thus, combining  \eqref{Recu:infty:PR3:Interm1} and \eqref{Recu:infty:PR3:Interm2} and recalling the constant $C_2^B$ is defined by \eqref{def:C2B}, we have proved \eqref{Recu:infty:PR3}.

	\underline{$\star$ \textit{Estimate \eqref{Recu:infty:PR4} on $\sum_{m=0}^{\lfloor\frac{n+1}{2}\rfloor}\Ldsp^{n-m}(Id-\Tc)\QD(h^m)$ in $\ell^\infty_{\gamma_\infty}$:} }
	
	To obtain the estimate \eqref{Recu:infty:PR4}, we will use the condition \ref{Th:C3} of the statement of Theorem \ref{Th}. We recall that this condition has two ways of being verified since at least one of the two triplets 
	$$\left(\pg_1+\pg_\infty,\gamma_1+\frac{1}{2\mu}+\min\left(\gamma_\infty,\frac{1}{2\mu}\right),\pg_\infty\right)\text{ and }\left(2\pg_\infty,\gamma_\infty+\min\left(\gamma_\infty,\frac{1}{2\mu}\right),\pg_\infty\right)$$
	verifies the condition \eqref{cond:H}. We will separate both possibilities of satisfying the condition \ref{Th:C3} of the statement of Theorem \ref{Th}.
	
	$\bullet$ The first possibility is that the triplet $\left(\pg_1+\pg_\infty,\gamma_1+\frac{1}{2\mu}+\min\left(\gamma_\infty,\frac{1}{2\mu}\right),\pg_\infty\right)$ verifies condition \eqref{cond:H}. Using the estimate \eqref{prop:Est:L^n(id-T):infty,1} of Proposition \ref{prop:Est}, we have that:
	\begin{equation}\label{Recu:infty:PR4:Interm1}
		\left\|\sum_{m=0}^{\lfloor\frac{n+1}{2}\rfloor}\Ldsp^{n-m}(Id-\Tc)\QD(h^m)\right\|_{\ell^\infty_{\gamma_\infty}} \leq \sum_{m=0}^{\lfloor\frac{n+1}{2}\rfloor} \frac{C_\Lcc(\gamma_\infty,\gamma_1+\gamma_\infty)}{(n+1-m)^{\gamma_1+\frac{1}{2\mu}+\min\left(\gamma_\infty,\frac{1}{2\mu}\right)}}\left\|\QD(h^m)\right\|_{\ell^1_{\gamma_1+\gamma_\infty}}.
	\end{equation}
	Furthermore, combining \eqref{Recu:infty:PR4:Interm1} and \eqref{Recu:1:PR4:Interm2}, we have that:
	\begin{multline}\label{Recu:infty:PR4:Interm2}
		\left\|\sum_{m=0}^{\lfloor\frac{n+1}{2}\rfloor}\Ldsp^{n-m}(Id-\Tc)\QD(h^m)\right\|_{\ell^\infty_{\gamma_\infty}} \\ \leq \varepsilon C_\Lcc(\gamma_\infty,\gamma_1+\gamma_\infty) C_{Q,1}\left(\gamma_1,\gamma_\infty\right){C_0}^2\left(\sum_{m=0}^{\lfloor\frac{n+1}{2}\rfloor} \frac{1}{(m+1)^{\pg_1+\pg_\infty}(n+1-m)^{\gamma_1+\frac{1}{2\mu}+\min\left(\gamma_\infty,\frac{1}{2\mu}\right)}}\right)\left\|\hg\right\|_{\ell^1_\Gamma}.
	\end{multline}
	Since we considered that the triplet $\left(\pg_1+\pg_\infty,\gamma_1+\frac{1}{2\mu}+\min\left(\gamma_\infty,\frac{1}{2\mu}\right),\pg_\infty\right)$ verifies condition \eqref{cond:H}, using Lemma \ref{lem:InSum}, we have that:
	\begin{multline*}
		\left\|\sum_{m=0}^{\lfloor\frac{n+1}{2}\rfloor}\Ldsp^{n-m}(Id-\Tc)\QD(h^m)\right\|_{\ell^\infty_{\gamma_\infty}} \\ \leq \frac{\varepsilon C_\Lcc(\gamma_\infty,\gamma_1+\gamma_\infty) C_{Q,1}\left(\gamma_1,\gamma_\infty\right){C_0}^2C_I\left(\pg_1+\pg_\infty,\gamma_1+\frac{1}{2\mu}\min\left(\gamma_\infty,\frac{1}{2\mu}\right),\pg_\infty\right)}{(n+2)^{\pg_\infty}}\left\|\hg\right\|_{\ell^1_\Gamma}.
	\end{multline*}
	When considering the condition \eqref{def:C2C:1} on the constant $C_2^C$, we thus find \eqref{Recu:infty:PR4}.
	
	$\bullet$ The second possibility is that the triplet $\left(2\pg_\infty,\gamma_\infty+\min\left(\gamma_\infty,\frac{1}{2\mu}\right),\pg_\infty\right)$ verifies condition \eqref{cond:H}. Using the estimate \eqref{prop:Est:L^n(id-T):infty,infty} of Proposition \ref{prop:Est}, we have that:
	\begin{equation}\label{Recu:infty:PR4:Interm3}
		\left\|\sum_{m=0}^{\lfloor\frac{n+1}{2}\rfloor}\Ldsp^{n-m}(Id-\Tc)\QD(h^m)\right\|_{\ell^\infty_{\gamma_\infty}} \leq \sum_{m=0}^{\lfloor\frac{n+1}{2}\rfloor} \frac{C_\Lcc(\gamma_\infty,2\gamma_\infty)}{(n+1-m)^{\gamma_\infty+\min\left(\gamma_\infty,\frac{1}{2\mu}\right)}}\left\|\QD(h^m)\right\|_{\ell^\infty_{2\gamma_\infty}}.
	\end{equation}
	
	Furthermore, for $m\in\lbrace0,\hdots,n\rbrace$, using the inequality \eqref{lem:inQ:ellinfty} on the sequence $\QD(h^m)$, the inequality \eqref{Recu:infty} on the sequence $h^m$ and finally the inequality \eqref{in:hg}, we have that:
	\begin{equation}\label{Recu:infty:PR4:Interm4}
		\left\|\QD(h^m)\right\|_{\ell^\infty_{2\gamma_\infty}}\leq C_{Q,\infty}\left(\gamma_\infty\right){\left\|h^m\right\|_{\ell^\infty_{\gamma_\infty}}}^2 \leq \frac{C_{Q,\infty}\left(\gamma_\infty\right){C_0}^2}{m^{2\pg_\infty}}{\left\|\hg\right\|_{\ell^1_\Gamma}}^2\leq \frac{\varepsilon C_{Q,\infty}\left(\gamma_\infty\right){C_0}^2}{m^{2\pg_\infty}}\left\|\hg\right\|_{\ell^1_\Gamma}.
	\end{equation}
	Thus, combining \eqref{Recu:infty:PR4:Interm3} and \eqref{Recu:infty:PR4:Interm4}, we have that:
	\begin{multline}\label{Recu:infty:PR4:Interm5}
		\left\|\sum_{m=0}^{\lfloor\frac{n+1}{2}\rfloor}\Ldsp^{n-m}(Id-\Tc)\QD(h^m)\right\|_{\ell^\infty_{\gamma_\infty}} \\ \leq \varepsilon C_\Lcc(\gamma_\infty,2\gamma_\infty) C_{Q,\infty}\left(\gamma_\infty\right){C_0}^2\left(\sum_{m=0}^{\lfloor\frac{n+1}{2}\rfloor} \frac{1}{(m+1)^{2\pg_\infty}(n+1-m)^{\gamma_\infty+\min\left(\gamma_\infty,\frac{1}{2\mu}\right)}}\right)\left\|\hg\right\|_{\ell^1_\Gamma}.
	\end{multline}
	Since we considered that the triplet $\left(2\pg_\infty,\gamma_\infty+\min\left(\gamma_\infty,\frac{1}{2\mu}\right),\pg_\infty\right)$ verifies condition \eqref{cond:H}, using Lemma \ref{lem:InSum}, we have that:
	\begin{multline*}
		\left\|\sum_{m=0}^{\lfloor\frac{n+1}{2}\rfloor}\Ldsp^{n-m}(Id-\Tc)\QD(h^m)\right\|_{\ell^\infty_{\gamma_\infty}} \\ \leq \frac{\varepsilon C_\Lcc(\gamma_\infty,2\gamma_\infty) C_{Q,\infty}\left(\gamma_\infty\right){C_0}^2C_I\left(2\pg_\infty,\gamma_\infty+\min\left(\gamma_\infty,\frac{1}{2\mu}\right),\pg_\infty\right)}{(n+2)^{\pg_\infty}}\left\|\hg\right\|_{\ell^1_\Gamma}.
	\end{multline*}
	When considering the condition \eqref{def:C2C:2} on the constant $C_2^C$, we thus find \eqref{Recu:infty:PR4}.

	\underline{$\star$ \textit{Estimate \eqref{Recu:infty:PR5} on $\sum_{m=\lfloor\frac{n+1}{2}\rfloor+1}^{n}\Ldsp^{n-m}(Id-\Tc)\QD(h^m)$ in $\ell^\infty_{\gamma_\infty}$:} }
	
	The proof is similar to the proof of the estimate \eqref{Recu:infty:PR4} on the term $\sum_{m=0}^{\lfloor\frac{n+1}{2}\rfloor}\Ldsp^{n-m}(Id-\Tc)\QD(h^m)$. There are two possibilities to separate with regards to condition \ref{Th:C4} of Theorem \ref{Th}.
	
	$\bullet$ The first possibility is that the triplet $\left(\gamma_1+\frac{1}{2\mu}+\min\left(\gamma_\infty,\frac{1}{2\mu}\right),\pg_1+\pg_\infty,\pg_\infty\right)$ verifies condition \eqref{cond:H}. In this case, using the exact same calculations as for the estimate \eqref{Recu:infty:PR4:Interm2}, we have:
	\begin{multline*}
		\left\|\sum_{m=\lfloor\frac{n+1}{2}\rfloor+1}^{n}\Ldsp^{n-m}(Id-\Tc)\QD(h^m)\right\|_{\ell^\infty_{\gamma_\infty}} \\ \leq \varepsilon C_\Lcc(\gamma_\infty,\gamma_1+\gamma_\infty) C_{Q,1}\left(\gamma_1,\gamma_\infty\right){C_0}^2\left(\sum_{m=\lfloor\frac{n+1}{2}\rfloor+1}^{n} \frac{1}{(m+1)^{\pg_1+\pg_\infty}(n+1-m)^{\gamma_1+\frac{1}{2\mu}+\min\left(\gamma_\infty,\frac{1}{2\mu}\right)}}\right)\left\|\hg\right\|_{\ell^1_\Gamma}.
	\end{multline*}
	Thus, using Lemma \ref{lem:InSum}, we have that:
	\begin{multline*}
		\left\|\sum_{m=\lfloor\frac{n+1}{2}\rfloor+1}^{n}\Ldsp^{n-m}(Id-\Tc)\QD(h^m)\right\|_{\ell^\infty_{\gamma_\infty}} \\ \leq \frac{\varepsilon C_\Lcc(\gamma_\infty,\gamma_1+\gamma_\infty) C_{Q,1}\left(\gamma_1,\gamma_\infty\right){C_0}^2C_I\left(\gamma_1+\frac{1}{2\mu}+\min\left(\gamma_\infty,\frac{1}{2\mu}\right),\pg_1+\pg_\infty,\pg_\infty\right)}{(n+2)^{\pg_\infty}}\left\|\hg\right\|_{\ell^1_\Gamma}.
	\end{multline*}
	When considering the condition \eqref{def:C2D:1} on the constant $C_2^D$, we thus find \eqref{Recu:infty:PR5}.
	
	$\bullet$ The second possibility is that the triplet $\left(\gamma_\infty+\min\left(\gamma_\infty,\frac{1}{2\mu}\right),2\pg_\infty,\pg_\infty\right)$ verifies condition \eqref{cond:H}. In this case, using the exact same calculations as for the estimate \eqref{Recu:infty:PR4:Interm5}, we have:
	\begin{multline*}
		\left\|\sum_{m=\lfloor\frac{n+1}{2}\rfloor+1}^{n}\Ldsp^{n-m}(Id-\Tc)\QD(h^m)\right\|_{\ell^\infty_{\gamma_\infty}} \\ \leq \varepsilon C_\Lcc(\gamma_\infty,2\gamma_\infty) C_{Q,\infty}\left(\gamma_\infty\right){C_0}^2\left(\sum_{m=\lfloor\frac{n+1}{2}\rfloor+1}^n \frac{1}{(m+1)^{2\pg_\infty}(n+1-m)^{\gamma_\infty+\min\left(\gamma_\infty,\frac{1}{2\mu}\right)}}\right)\left\|\hg\right\|_{\ell^1_\Gamma}.
	\end{multline*}
	Thus, using Lemma \ref{lem:InSum}, we have that:
	\begin{multline*}
		\left\|\sum_{m=\lfloor\frac{n+1}{2}\rfloor+1}^{n}\Ldsp^{n-m}(Id-\Tc)\QD(h^m)\right\|_{\ell^\infty_{\gamma_\infty}} \\ \leq  \frac{\varepsilon C_\Lcc(\gamma_\infty,2\gamma_\infty) C_{Q,\infty}\left(\gamma_\infty\right){C_0}^2C_I\left(\gamma_\infty+\min\left(\gamma_\infty,\frac{1}{2\mu}\right),2\pg_\infty,\pg_\infty\right)}{(n+2)^{\pg_\infty}}\left\|\hg\right\|_{\ell^1_\Gamma}.
	\end{multline*}
	When considering the condition \eqref{def:C2D:2} on the constant $C_2^D$, we thus find \eqref{Recu:infty:PR5}.
	
	This concludes the proof of the inequalities \eqref{Recu:infty:PR}. Just as stated right after those inequalities, this allows us to conclude the proof of \eqref{Recu:infty} for $m=n+1$.
	
	\noindent \underline{$\bullet$ \textbf{Proof of \eqref{Recu:delta} for $m=n+1$:}}
	
	We observe that the equality \eqref{Recu:infty} for $m=n+1$, the condition \eqref{in:hg} on the initial perturbation $\hg$ and the condition \eqref{in:varepsDelta} on $\varepsilon$ imply that:
	$$\left\|h^{n+1}\right\|_{\ell^\infty} \leq \left\|h^{n+1}\right\|_{\ell^\infty_{\gamma_\infty}} \leq \frac{C_0}{(n+2)^{\pg_\infty}}\left\|\hg\right\|_{\ell^1_{\Gamma}} <C_0\varepsilon<\Rayon.$$
	Therefore, we have \eqref{Recu:delta} for $m=n+1$. 
	
	We have thus proved that $\Pg(n+1)$ is true and have concluded the induction step. Therefore, by induction, $\Pg(n)$ is true for all $n\in\N$ and this allows us to conclude the proof of Theorem \ref{Th}.

	\section{Estimates on the operators \texorpdfstring{$\Ldsp^n$}{Ln}, \texorpdfstring{$\Ldsp^n(Id-\Tc)$}{Ln(I-T)} and \texorpdfstring{$\Ldsp^n(\LdspD-\Ldsp)$}{Ln(Ld-L)}}\label{sec:Preuve_prop:Est}

	The proof of Proposition \ref{prop:Est} can be seen as a fairly longer version of the proof of \cite[Theorem 2]{Coeuret2024d}. It will be separated in three sections: Section \ref{subsec:prop:Est:1} will be dedicated to proving the estimates \eqref{prop:Est:L^n:1} and \eqref{prop:Est:L^n:infty} on the operator $\Ldsp^n$, Section \ref{subsec:prop:Est:2} will tackle the estimates \eqref{prop:Est:L^n(id-T):1,1}, \eqref{prop:Est:L^n(id-T):infty,1} and \eqref{prop:Est:L^n(id-T):infty,infty} on the operator $\Ldsp^n(Id-\Tc)$ and finally Section \ref{subsec:prop:Est:3} will tackle the estimate \eqref{prop:Est:L^n(LD-L)} on the operator $\Ldsp^n(\LdspD-\Ldsp)$. Before beginning with the proofs, let us make some useful observations.
	
	\begin{itemize}
		\item We will fix the constant $c$ that appears in the decompositions \eqref{decompoGreen} and \eqref{decompoDerGreen} of the Green's function and its discrete derivative. This will allow us to introduce the constant $\tilde{c}>0$ defined by:
		\begin{equation}\label{def:ctilde}
			\tilde{c} := \frac{c}{2^\frac{2\mu}{\mu-1}}.
		\end{equation}
		Then, since the constant $\alpha^+$ is negative, we observe that :
		\begin{multline}\label{in:ctilde}
			\forall n\in\N,\forall j_0\in\N\cap \left[0,-\frac{n\alpha^+}{2}\right] ,\forall j\in\N,\\ \exp\left(-c\left(\frac{\left|n-\frac{j-j_0}{\alpha^+}\right|}{n^\frac{1}{2\mu}}\right)^{\frac{2\mu}{2\mu-1}}\right)\leq\exp\left(-c\left(\frac{\left|n+\frac{j_0}{\alpha^+}\right|}{n^\frac{1}{2\mu}}\right)^{\frac{2\mu}{2\mu-1}}\right)\leq \exp(-\tilde{c}n).
		\end{multline}
		where the constant $\tilde{c}$ is defined by \eqref{def:ctilde}.
		
		\item There exists a constant $C>0$ such that:
		\begin{equation}\label{in:sommeGauss}
			\forall n\in\N\backslash\lbrace0\rbrace,\quad \sum_{j\in\Z}\frac{1}{n^\frac{1}{2\mu}}\exp\left(-c\left(\frac{\left|n-\frac{j}{\alpha^+}\right|}{n^\frac{1}{2\mu}}\right)^\frac{2\mu}{2\mu-1}\right)<C
		\end{equation}
		where the constant $c$ has been fixed as stated above.
		
		\item We also observe that, for any given parameter $\gamma\in[0,+\infty[$, there exists a positive constant $C$ such that:
		\begin{equation}\label{in:cnj0j}
			\forall n\in\N,\forall j_0\in\Z,\forall j\in\Z,\quad j-j_0\in\lbrace-nq,\hdots,np\rbrace \quad \Rightarrow \quad \left(1+|j|\right)^\gamma\leq C\left(1+|j_0|\right)^\gamma\left(1+n\right)^\gamma.
		\end{equation}
	\end{itemize}
	
	\subsection{Proof of the estimates \texorpdfstring{\eqref{prop:Est:L^n:1}}{} and \texorpdfstring{\eqref{prop:Est:L^n:infty}}{} on the operator \texorpdfstring{$\Ldsp^n$}{Ln}}\label{subsec:prop:Est:1}
	
	We first observe that \eqref{lienLccGreenV1} and Lemma \ref{lemGreenTempo} imply that:
	\begin{equation}\label{lienLccGreen}
		\forall n\in\N,\forall h\in\ell^\infty(\Z),\quad \Ldsp^nh= \left(\sum_{j_0\in\Z}\ind_{j-j_0\in\lbrace-nq,\hdots,np\rbrace}\Gcc(n,j_0,j)h_{j_0}\right)_{j\in\Z}.
	\end{equation}
	
	Let us recall here that, in the introduction of the present paper, we introduced a decomposition \eqref{decompoGreen} of the Green's function in the scalar case. We have that there exists a constant $c>0$ such that for $j_0\in\Z$, $n\in\N\backslash\lbrace0\rbrace$ and $j\in\Z$ such that $j-j_0\in\lbrace-nq,\hdots,np\rbrace$, the Green's function can be decomposed in 4 parts:
	\begin{subequations}\label{decompoGreenImproved}
		\begin{align}
			\forall j_0\in\N,\quad  &\Gcc(n,j_0,j) = E^+(n,j_0)V(j)+ D^+(n,j_0,j) + R^+(n,j_0,j)+O(e^{-cn}),\\
			\forall j_0\in-\N\backslash\lbrace0\rbrace,\quad  &\Gcc(n,j_0,j) = E^-(n,j_0)V(j)+ D^-(n,j_0,j) + R^-(n,j_0,j)+O(e^{-cn}),
		\end{align}
	\end{subequations}
	where the sequence $V$ defined in \eqref{decompoGreen} is an eigenvector of the operator $\Lcc$ associated with the eigenvalue $1$ and the other terms are defined as follows for $n\in\N$, $j_0,j\in\Z$:
	\begin{itemize}
		\item The terms $E^\pm(n,j_0)$ correspond to the activation of the eigenvalue $1$ of the operator $\Lcc$ and is defined by:
		\begin{subequations}\label{ExpE}
			\begin{align}
				\forall j_0\in\N, \quad &  E^+(n,j_0) := C_EE_{2\mu}\left(\beta^+;\frac{n\alpha^+ +j_0}{n^\frac{1}{2\mu}}\right),\\
				\forall j_0\in-\N\backslash\lbrace0\rbrace, \quad & E^-(n,j_0) := C_EE_{2\mu}\left(\beta^-;\frac{-n\alpha^- -j_0}{n^\frac{1}{2\mu}}\right).
			\end{align}
		\end{subequations}
		\item The term $D^+(n,j_0,j)$ (resp. $D^-(n,j_0,j)$) corresponds to the diffusion wave which is incoming with respect to the shock and which is associated with the characteristic field of the right (resp. left) state :
		\begin{subequations}\label{ExpDpm}
			\begin{align}
				\forall j_0\in\N, \quad & D^+(n,j_0,j) = \ind_{j\geq0}O\left(\frac{1}{n^\frac{1}{2\mu}}\exp\left(-c\left(\frac{\left|n-\left(\frac{j-j_0}{\alpha^+}\right)\right|}{n^\frac{1}{2\mu}}\right)^\frac{2\mu}{2\mu-1}\right)\right),\label{ExpD+}\\
				\forall j_0\in-\N\backslash\lbrace0\rbrace, \quad & D^-(n,j_0,j) = \ind_{j\leq0}O\left(\frac{1}{n^\frac{1}{2\mu}}\exp\left(-c\left(\frac{\left|n-\left(\frac{j-j_0}{\alpha^-}\right)\right|}{n^\frac{1}{2\mu}}\right)^\frac{2\mu}{2\mu-1}\right)\right),
			\end{align}
			where the term $O$ is uniform with respect to $n$,$j_0$ and $j$.
		\end{subequations}
		
		\item The terms $R^\pm(n,j_0,j)$ correspond to a remainder term for the activation of the eigenvector associated with the eigenvalue $1$ of the operator $\Ldsp$:
		\begin{subequations}
			\begin{align}
				\forall j_0\in\N, \quad & R^+(n,j_0,j) =O\left(\frac{e^{-c|j|}}{n^\frac{1}{2\mu}}\exp\left(-c\left(\frac{\left|n+\frac{j_0}{\alpha^+}\right|}{n^\frac{1}{2\mu}}\right)^\frac{2\mu}{2\mu-1}\right)\right),\label{ExpR+}\\
				\forall j_0\in-\N\backslash\lbrace0\rbrace, \quad & R^-(n,j_0,j) =O\left(\frac{e^{-c|j|}}{n^\frac{1}{2\mu}}\exp\left(-c\left(\frac{\left|n+\frac{j_0}{\alpha^-}\right|}{n^\frac{1}{2\mu}}\right)^\frac{2\mu}{2\mu-1}\right)\right),
			\end{align}
			where the term $O$ is uniform with respect to $n$, $j_0$ and $j$.
		\end{subequations}
	\end{itemize} 
	
	Applying the decompositions \eqref{decompoGreenImproved} of the Green's function in the right-hand side term of \eqref{lienLccGreen}, we have that for $h\in\ell^\infty(\Z)$, $n\in\N\backslash\lbrace0\rbrace$ and $j\in\Z$:
	\begin{multline}\label{lienLccGreenImproved}
		(\Ldsp^nh)_j = T^+_D(h,n,j)+ T^-_D(h,n,j)+T^+_R(h,n,j)+ T^-_R(h,n,j)+ T_E(h,n,j) \\+\sum_{j_0\in\Z} \ind_{j-j_0\in\lbrace-nq,\hdots,np\rbrace} O(e^{-cn}) h_{j_0}
	\end{multline}
	where the terms $T^\pm_D(h,n,j)$, $T^\pm_R(h,n,j)$ and $T_E(h,n,j)$ are respectively defined by:
	\begin{subequations}
		\begin{align}
			T^+_D(h,n,j)&:= \sum_{j_0\in\N} \ind_{j-j_0\in\lbrace-nq,\hdots,np\rbrace} D^+(n,j_0,j)h_{j_0},\label{ExpTD+}\\
			T^-_D(h,n,j)&:= \sum_{j_0\in-\N\backslash\lbrace0\rbrace} \ind_{j-j_0\in\lbrace-nq,\hdots,np\rbrace} D^-(n,j_0,j)h_{j_0},\\
			T^+_R(h,n,j)&:= \sum_{j_0\in\N} \ind_{j-j_0\in\lbrace-nq,\hdots,np\rbrace} R^+(n,j_0,j)h_{j_0},\label{ExpTR+}\\
			T^-_R(h,n,j)&:= \sum_{j_0\in-\N\backslash\lbrace0\rbrace} \ind_{j-j_0\in\lbrace-nq,\hdots,np\rbrace} R^-(n,j_0,j)h_{j_0},\\
			T_E(h,n,j)& := \left(\sum_{j_0\in\N}\ind_{j-j_0\in\lbrace-nq,\hdots,np\rbrace}E^+(n,j_0)h_{j_0}+\sum_{j_0\in-\N\backslash\lbrace0\rbrace}\ind_{j-j_0\in\lbrace-nq,\hdots,np\rbrace}E^-(n,j_0)h_{j_0}\right)V(j).\label{ExpTE}
		\end{align}
	\end{subequations}
	
	In the following lemma, we will prove sharp estimates on the terms that appear in the right-hand side of \eqref{lienLccGreenImproved}. 
	\begin{lemma}\label{prop:Est:lem1}
		For $0\leq \gamma\leq \Gamma$, there exists a constant $C>0$ such that:
		
		$\bullet$ \textbf{\underline{First terms of the decomposition (the diffusion waves):}}
		\begin{subequations}\label{prop:Est:lem1:1}
			\begin{align}
				\begin{split}
					\forall n\in\N\backslash\lbrace0\rbrace,\forall h\in\ell^1_{\Gamma},\quad &\left\|\left(T^\pm_D(h,n,j)\right)_{j\in\Z}\right\|_{\ell^1_{\gamma}} \leq \frac{C}{n^{\Gamma-\gamma}} \left\|h\right\|_{\ell^1_{\Gamma}},
				\end{split}\label{prop:Est:lem1:1a}\\
				\begin{split}
					\forall n\in\N\backslash\lbrace0\rbrace,\forall h\in\ell^\infty_{\Gamma},\quad &\left\|\left(T^\pm_D(h,n,j)\right)_{j\in\Z}\right\|_{\ell^\infty_{\gamma}} \leq \frac{C}{n^{\Gamma-\gamma}} \left\|h\right\|_{\ell^\infty_{\Gamma}}
				\end{split}\label{prop:Est:lem1:1b}\\
				\begin{split}
					\forall n\in\N\backslash\lbrace0\rbrace,\forall h\in\ell^1_{\Gamma},\quad &\left\|\left(T^\pm_D(h,n,j)\right)_{j\in\Z}\right\|_{\ell^\infty_{\gamma}} \leq \frac{C}{n^{\Gamma-\gamma+\frac{1}{2\mu}}} \left\|h\right\|_{\ell^1_{\Gamma}}.
				\end{split}\label{prop:Est:lem1:1c}
			\end{align}
		\end{subequations}
		
		$\bullet$ \textbf{\underline{Second term of the decomposition (the activation remainder):}}
		\begin{subequations}\label{prop:Est:lem1:2}
			\begin{align}
				\begin{split}
					\forall n\in\N\backslash\lbrace0\rbrace,\forall h\in\ell^\infty_{\Gamma},\quad \left\|\left(T^\pm_R(h,n,j)\right)_{j\in\Z}\right\|_{\ell^1_{\gamma}}\leq \frac{C}{n^{\Gamma}} \left\|h\right\|_{\ell^\infty_{\Gamma}},
				\end{split}\label{prop:Est:lem1:2a}\\
				\begin{split}
					\forall n\in\N\backslash\lbrace0\rbrace,\forall h\in\ell^1_{\Gamma},\quad \left\|\left(T^\pm_R(h,n,j)\right)_{j\in\Z}\right\|_{\ell^1_{\gamma}}\leq \frac{C}{n^{\Gamma+\frac{1}{2\mu}}} \left\|h\right\|_{\ell^1_{\Gamma}}.
				\end{split}\label{prop:Est:lem1:2b}
			\end{align}
		\end{subequations}
		
		$\bullet$ \textbf{\underline{Third term of the decomposition:}}
		\begin{subequations}\label{prop:Est:lem1:3}
			\begin{align}
				\forall n\in\N\backslash\lbrace0\rbrace,\forall h\in\ell^1_{\Gamma},\quad &\left\|\left(\sum_{j_0\in\Z}\ind_{j-j_0\in\lbrace-nq,\hdots,np\rbrace}O(e^{-cn})h_{j_0}\right)_{j\in\Z}\right\|_{\ell^1_{\gamma}}\leq C e^{-\frac{c}{2}n}\left\|h\right\|_{\ell^1_{\Gamma}},\label{prop:Est:lem1:3a}\\
				\forall n\in\N\backslash\lbrace0\rbrace,\forall h\in\ell^\infty_{\Gamma}, \quad & \left\|\left(\sum_{j_0\in\Z}\ind_{j-j_0\in\lbrace-nq,\hdots,np\rbrace}O(e^{-cn})h_{j_0}\right)_{j\in\Z}\right\|_{\ell^\infty_{\gamma}}\leq C e^{-\frac{c}{2}n}\left\|h\right\|_{\ell^\infty_{\Gamma}}.\label{prop:Est:lem1:3b}
			\end{align}
		\end{subequations}	
		
		$\bullet$ \textbf{\underline{Fourth term of the decomposition (associated with the eigenvalue $1$ of $\Lcc$):}}
		\begin{equation}\label{prop:Est:lem1:4}
			\forall n\in\N\backslash\lbrace0\rbrace,\forall h\in\E_{\Gamma},\quad \left\|\left( T_E(h,n,j) \right)_{j\in\Z}\right\|_{\ell^1_{\gamma}} \leq \frac{C}{n^{\Gamma}}\left\|h\right\|_{\ell^1_{\Gamma}}.
		\end{equation}
	\end{lemma}
	
	We point out that we exhibited estimates of some of the terms (for instance the second or fourth term) only in the $\ell^1_{\gamma}$-norm. Indeed, we can obtain estimates in the $\ell^\infty_{\gamma}$-norm by observing that:
	$$\forall h\in\ell^1_{\gamma},\quad \left\|h\right\|_{\ell^\infty_{\gamma}}\leq \left\|h\right\|_{\ell^1_{\gamma}}.$$ 
	Combining the results of Lemma \ref{prop:Est:lem1} and the equality \eqref{lienLccGreenImproved}, we can immediately obtain the inequalities \eqref{prop:Est:L^n:1} and \eqref{prop:Est:L^n:infty} on the semi-group $(\Ldsp^n)_{n\in\N}$. The result of Lemma \ref{prop:Est:lem1} is actually slightly more complete than what is necessary to prove \eqref{prop:Est:L^n:1} and \eqref{prop:Est:L^n:infty}. This completeness will be useful in the Section \ref{subsec:prop:Est:2} tackling the proofs of \eqref{prop:Est:L^n(id-T):1,1}, \eqref{prop:Est:L^n(id-T):infty,1} and \eqref{prop:Est:L^n(id-T):infty,infty}.
	
	\begin{proof}\textbf{of Lemma \ref{prop:Est:lem1}}
		
		We fix $0\leq \gamma\leq \Gamma$. We will separate the proofs of the different estimates claimed in Lemma \ref{prop:Est:lem1}.
		
		\underline{$\blacktriangleright$ \textbf{Proof of the estimates \eqref{prop:Est:lem1:1} on the first term (diffusion waves):}}
		
		In this part, we will only prove the estimates for the sequence $\left(T^+_D(h,n,j)\right)_{j\in\Z}$. Everything can then easily be extended for the sequence $\left(T^-_D(h,n,j)\right)_{j\in\Z}$. 
		
		First, using the expressions \eqref{ExpTD+} of $T_D^+(h,n,j)$ and \eqref{ExpD+} of $D^+(n,j_0,j)$, we observe that for $n\in\N\backslash\lbrace0\rbrace$, $h\in\ell^\infty_{\Gamma}$ and $j\in\Z$, we have:
		\begin{align}\label{prop:Est:lem1:1a:Overall}
			\begin{split}
				&(1+|j|)^{\gamma} \left|T_D^+(h,n,j)\right|  \\
				= & (1+|j|)^{\gamma} \left|\sum_{j_0\in\N}\ind_{j-j_0\in\lbrace-nq,\hdots,np\rbrace}\ind_{j\geq0}O\left(\frac{1}{n^\frac{1}{2\mu}}\exp\left(-c\left(\frac{\left|n-\frac{j-j_0}{\alpha^+}\right|}{n^\frac{1}{2\mu}}\right)^{\frac{2\mu}{2\mu-1}}\right)\right)h_{j_0}\right| \\ 
				\lesssim& \ind_{j\geq0} \sum_{j_0\in\N}\ind_{j-j_0\in\lbrace-nq,\hdots,np\rbrace}(1+|j|)^{\gamma}\frac{1}{n^\frac{1}{2\mu}}\exp\left(-c\left(\frac{\left|n-\frac{j-j_0}{\alpha^+}\right|}{n^\frac{1}{2\mu}}\right)^{\frac{2\mu}{2\mu-1}}\right)\left|h_{j_0}\right|
			\end{split}
		\end{align}
		where the notations $\lesssim$, introduced in the "Notations" paragraph of the introduction, is used to ease the reading and describes inequalities up to a multiplicative constant independent from the parameters $h$, $n$ and $j$.
		
		\underline{$\bullet$ \textit{Proof of the estimate \eqref{prop:Est:lem1:1a}}}
		
		\begin{subequations}
			We consider $h\in\ell^1_{\Gamma}$ and $n\in\N\backslash\lbrace0\rbrace$. To prove \eqref{prop:Est:lem1:1a}, we want to find estimates on the sum:
			$$\left\|\left(T_D^+(h,n,j)\right)_{j\in\Z}\right\|_{\ell^1_{\gamma}}=\sum_{j\in\Z}(1+|j|)^{\gamma} \left|T_D^+(h,n,j)\right|.$$
			We observe that the inequality \eqref{prop:Est:lem1:1a:Overall} implies that we only need to find bounds on 
			\begin{equation}\label{prop:Est:lem1:1a:1}
				\sum_{j\in\Z}\ind_{j\geq0} \sum_{j_0\in\N}\ind_{j-j_0\in\lbrace-nq,\hdots,np\rbrace}(1+|j|)^{\gamma}\frac{1}{n^\frac{1}{2\mu}}\exp\left(-c\left(\frac{\left|n-\frac{j-j_0}{\alpha^+}\right|}{n^\frac{1}{2\mu}}\right)^{\frac{2\mu}{2\mu-1}}\right)\left|h_{j_0}\right|.
			\end{equation}
			We will decompose the sum \eqref{prop:Est:lem1:1a:1} with respect to $j$ according to various regimes of $j-j_0$.
			
			$\star$ For $n\in\N\backslash\lbrace0\rbrace$ and $h\in\ell^1_{\Gamma}$, we have using the inequality \eqref{in:cnj0j} and the fact that $\alpha^+<0$:
			\begin{align}
				\begin{split}
					\hspace{3cm} \mathclap{\hspace{8cm}\sum_{j\in\Z}\ind_{j\geq0}\sum_{j_0=0}^j\ind_{j-j_0\in\lbrace-nq,\hdots,np\rbrace}(1+|j|)^{\gamma}\frac{1}{n^\frac{1}{2\mu}}\underset{\leq\exp(-cn)}{\underbrace{\exp\left(-c\left(\frac{\left|n-\frac{j-j_0}{\alpha^+}\right|}{n^\frac{1}{2\mu}}\right)^{\frac{2\mu}{2\mu-1}}\right)}}\left|h_{j_0}\right|}&\\
					&\lesssim \sum_{j_0\in\N} \underset{\lesssim n}{\underbrace{\sum_{j\geq j_0}\ind_{j-j_0\in\lbrace-nq,\hdots,np\rbrace}}}\frac{(1+n)^{\gamma}}{n^\frac{1}{2\mu}}  \exp(-cn)(1+|j_0|)^{\gamma}|h_{j_0}|\\
					&\lesssim \exp\left(-\frac{c}{2}n\right)\left\|h\right\|_{\ell^1_{\gamma}}\\
					&\lesssim \exp\left(-\frac{c}{2}n\right)\left\|h\right\|_{\ell^1_{\Gamma}}.
				\end{split}\label{prop:Est:lem1:1a:2}
			\end{align}
			
			$\star$ For $n\in\N\backslash\lbrace0\rbrace$ and $h\in\ell^1_{\Gamma}$, we have using the inequality \eqref{in:ctilde} and the fact that $\gamma\leq \Gamma$:
			\begin{align}
				\begin{split}
					\hspace{3cm} \mathclap{\hspace{8cm}\sum_{j\in\Z}\ind_{j\geq0}\sum_{j_0\geq j+1}\ind_{j_0\in\left[0,-\frac{n\alpha^+}{2}\right[}\ind_{j-j_0\in\lbrace-nq,\hdots,np\rbrace}\underset{\leq(1+|j_0|)^{\gamma}}{\underbrace{(1+|j|)^{\gamma}}}\frac{1}{n^\frac{1}{2\mu}}\exp\left(-c\left(\frac{\left|n-\frac{j-j_0}{\alpha^+}\right|}{n^\frac{1}{2\mu}}\right)^{\frac{2\mu}{2\mu-1}}\right)\left|h_{j_0}\right|} & \\
					&\lesssim \sum_{j_0\in\N}\ind_{j_0\in\left[0,-\frac{n\alpha^+}{2}\right[} \underset{\lesssim n}{\underbrace{\sum_{j=0}^{j_0-1}\ind_{j-j_0\in\lbrace-nq,\hdots,np\rbrace}}}\frac{1}{n^\frac{1}{2\mu}}  \exp(-\tilde{c}n)(1+|j_0|)^{\gamma}|h_{j_0}|\\
					&\lesssim \exp\left(-\frac{\tilde{c}}{2}n\right)\left\|h\right\|_{\ell^1_{\gamma}}\\
					&\lesssim \exp\left(-\frac{\tilde{c}}{2}n\right)\left\|h\right\|_{\ell^1_{\Gamma}}.
				\end{split}\label{prop:Est:lem1:1a:3}
			\end{align}
			
			$\star$ For $n\in\N\backslash\lbrace0\rbrace$ and $h\in\ell^1_{\Gamma}$, we have using the inequality \eqref{in:sommeGauss}:
			\begin{align}
				\begin{split}
					\hspace{2cm}\mathclap{\hspace{10cm}\sum_{j\in\Z}\ind_{j\geq0}\sum_{j_0\geq j+1}\ind_{j_0\in\left[-\frac{n\alpha^+}{2},+\infty\right[}\ind_{j-j_0\in\lbrace-nq,\hdots,np\rbrace}\underset{\leq(1+|j_0|)^{\gamma}}{\underbrace{(1+|j|)^{\gamma}}}\frac{1}{n^\frac{1}{2\mu}}\exp\left(-c\left(\frac{\left|n-\frac{j-j_0}{\alpha^+}\right|}{n^\frac{1}{2\mu}}\right)^{\frac{2\mu}{2\mu-1}}\right)\left|h_{j_0}\right|} & \\
					&\lesssim \sum_{j_0\in\N}\ind_{j_0\in\left[-\frac{n\alpha^+}{2},+\infty\right[} \underset{\lesssim1}{\underbrace{\left(\sum_{j=0}^{j_0-1}\frac{1}{n^\frac{1}{2\mu}} \exp\left(-c\left(\frac{\left|n-\frac{j-j_0}{\alpha^+}\right|}{n^\frac{1}{2\mu}}\right)^{\frac{2\mu}{2\mu-1}}\right)\right)}} \frac{(1+|j_0|)^{\Gamma}}{|j_0|^{\Gamma-\gamma}}|h_{j_0}|\\
					&\lesssim \frac{1}{n^{\Gamma-\gamma}}\left\|h\right\|_{\ell^1_{\Gamma}}.
				\end{split}\label{prop:Est:lem1:1a:4}
			\end{align}
			
			Using \eqref{prop:Est:lem1:1a:Overall} and \eqref{prop:Est:lem1:1a:2}-\eqref{prop:Est:lem1:1a:4} allows us to conclude the proof of \eqref{prop:Est:lem1:1a}.
		\end{subequations}
		
		\underline{$\bullet$ \textit{Proof of the estimate \eqref{prop:Est:lem1:1b}}}
		
		\begin{subequations}
			The proof of \eqref{prop:Est:lem1:1b} is fairly similar to \eqref{prop:Est:lem1:1a} with slight modifications made precise below. For $n\in\N\backslash\lbrace0\rbrace$ and $h\in \ell^\infty_{\Gamma}$, we want to prove bounds 
			$$\left\|\left(T_D^+(h,n,j)\right)_{j\in\Z}\right\|_{\ell^\infty_{\gamma}} = \sup_{j\in\Z}(1+|j|)^{\gamma}\left|T_D^+(h,n,j)\right| . $$
			The inequality \eqref{prop:Est:lem1:1a:Overall} thus leads us to decompose and study the sum	for $j\in\Z$
			\begin{equation}\label{prop:Est:lem1:1b:1}
				\ind_{j\geq0} \sum_{j_0\in\N}\ind_{j-j_0\in\lbrace-nq,\hdots,np\rbrace}(1+|j|)^{\gamma}\frac{1}{n^\frac{1}{2\mu}}\exp\left(-c\left(\frac{\left|n-\frac{j-j_0}{\alpha^+}\right|}{n^\frac{1}{2\mu}}\right)^{\frac{2\mu}{2\mu-1}}\right)\left|h_{j_0}\right|.
			\end{equation}
			
			$\star$ For $n\in\N\backslash\lbrace0\rbrace$, $h\in\ell^\infty_{\Gamma}$ and $j\in\N$, we have using \eqref{in:cnj0j}:
			\begin{align}
				\begin{split}
					&\sum_{j_0=0}^j\ind_{j-j_0\in\lbrace-nq,\hdots,np\rbrace}(1+|j|)^{\gamma}\frac{1}{n^\frac{1}{2\mu}}\underset{\leq\exp(-cn)}{\underbrace{\exp\left(-c\left(\frac{\left|n-\frac{j-j_0}{\alpha^+}\right|}{n^\frac{1}{2\mu}}\right)^{\frac{2\mu}{2\mu-1}}\right)}}\left|h_{j_0}\right|\\
					\lesssim& \underset{\lesssim n}{\underbrace{\sum_{j_0=0}^j\ind_{j-j_0\in\lbrace-nq,\hdots,np\rbrace}}}\frac{(1+n)^{\gamma}}{n^\frac{1}{2\mu}}  \exp(-cn) \underset{\leq \left\|h\right\|_{\ell^\infty_{\gamma}}\leq \left\|h\right\|_{\ell^\infty_{\Gamma}}}{\underbrace{(1+|j_0|)^{\gamma}|h_{j_0}|}}\\
					\lesssim &\exp\left(-\frac{c}{2}n\right)\left\|h\right\|_{\ell^\infty_{\Gamma}}.
				\end{split}\label{prop:Est:lem1:1b:2}
			\end{align}
			
			$\star$ For $n\in\N\backslash\lbrace0\rbrace$, $h\in\ell^\infty_{\Gamma}$ and $j\in\N$, we have using the inequality \eqref{in:ctilde}:
			\begin{align}
				\begin{split}
					\hspace{3cm}\mathclap{\hspace{8cm}\sum_{j_0\geq j+1}\ind_{j_0\in\left[0,-\frac{n\alpha^+}{2}\right[}\ind_{j-j_0\in\lbrace-nq,\hdots,np\rbrace}\underset{\leq(1+|j_0|)^{\gamma}}{\underbrace{(1+|j|)^{\gamma}}}\frac{1}{n^\frac{1}{2\mu}}\exp\left(-c\left(\frac{\left|n-\frac{j-j_0}{\alpha^+}\right|}{n^\frac{1}{2\mu}}\right)^{\frac{2\mu}{2\mu-1}}\right)\left|h_{j_0}\right|} & \\
					&\lesssim \underset{\lesssim n}{\underbrace{\sum_{j_0\geq j+1}\ind_{j_0\in\left[0,-\frac{n\alpha^+}{2}\right[}\ind_{j-j_0\in\lbrace-nq,\hdots,np\rbrace}}}\frac{1}{n^\frac{1}{2\mu}}  \exp(-\tilde{c}n)\underset{\leq\left\|h\right\|_{\ell^\infty_{\Gamma}}}{\underbrace{(1+|j_0|)^{\gamma}|h_{j_0}|}}\\
					&\lesssim \exp\left(-\frac{\tilde{c}}{2}n\right)\left\|h\right\|_{\ell^\infty_{\Gamma}}.
				\end{split}\label{prop:Est:lem1:1b:3}
			\end{align}
			
			$\star$ For $n\in\N\backslash\lbrace0\rbrace$, $h\in\ell^\infty_{\Gamma}$ and $j\in\N$, we have using \eqref{in:sommeGauss}:
			\begin{align}
				\begin{split}
					\hspace{2cm}\mathclap{\hspace{10cm}\sum_{j_0\geq j+1}\ind_{j_0\in\left[-\frac{n\alpha^+}{2},+\infty\right[}\ind_{j-j_0\in\lbrace-nq,\hdots,np\rbrace}\underset{\leq(1+|j_0|)^{\gamma}}{\underbrace{(1+|j|)^{\gamma}}}\frac{1}{n^\frac{1}{2\mu}}\exp\left(-c\left(\frac{\left|n-\frac{j-j_0}{\alpha^+}\right|}{n^\frac{1}{2\mu}}\right)^{\frac{2\mu}{2\mu-1}}\right)\left|h_{j_0}\right|} & \\
					&\lesssim \sum_{j_0\in\N}\ind_{j_0\in\left[-\frac{n\alpha^+}{2},+\infty\right[} \left(\frac{1}{n^\frac{1}{2\mu}}\exp\left(-c\left(\frac{\left|n-\frac{j-j_0}{\alpha^+}\right|}{n^\frac{1}{2\mu}}\right)^{\frac{2\mu}{2\mu-1}}\right)\right) \frac{(1+|j_0|)^{\Gamma}}{|j_0|^{\Gamma-\gamma}}|h_{j_0}|\\
					&\lesssim \underset{\lesssim1}{\underbrace{\sum_{j_0\in\N}\ind_{j_0\in\left[-\frac{n\alpha^+}{2},+\infty\right[} \left(\frac{1}{n^\frac{1}{2\mu}}\exp\left(-c\left(\frac{\left|n-\frac{j-j_0}{\alpha^+}\right|}{n^\frac{1}{2\mu}}\right)^{\frac{2\mu}{2\mu-1}}\right)\right)}}\frac{1}{n^{\Gamma-\gamma}}\left\|h\right\|_{\ell^\infty_{\Gamma}}\\
					&\lesssim \frac{1}{n^{\Gamma-\gamma}}\left\|h\right\|_{\ell^\infty_{\Gamma}}.
				\end{split}\label{prop:Est:lem1:1b:4}
			\end{align}
			
			Using \eqref{prop:Est:lem1:1a:Overall} and \eqref{prop:Est:lem1:1b:2}-\eqref{prop:Est:lem1:1b:4} allows us to conclude the proof of \eqref{prop:Est:lem1:1b}.
		\end{subequations}

		\underline{$\bullet$ \textit{Proof of the estimate \eqref{prop:Est:lem1:1c}}}
		
		\begin{subequations}
			The proof of \eqref{prop:Est:lem1:1c} is essentially the same one as for \eqref{prop:Est:lem1:1b}. Indeed, we still need to find estimates for \eqref{prop:Est:lem1:1b:1} but this time for $h\in\ell^1_{\Gamma}$ and $n\in\N\backslash\lbrace0\rbrace$. Let us observe that we have the following inequality:
			\begin{equation}\label{prop:Est:lem1:1c:1}
				\forall h\in  \ell^1_{\Gamma}, \quad \left\|h\right\|_{\ell^\infty_{\Gamma}}\leq \left\|h\right\|_{\ell^1_{\Gamma}}.
			\end{equation}
			This inequality allows us to immediately adapt some intermediate results in the proof of \eqref{prop:Est:lem1:1b}, which applied to the case where $h$ belonged to $\ell^\infty_{\Gamma}$, for the proof of \eqref{prop:Est:lem1:1c}, where $h$ belonged to $\ell^1_{\Gamma}$. This will be explained below.
			
			$\star$ For $n\in\N\backslash\lbrace0\rbrace$, $h\in\ell^1_{\Gamma}$ and $j\in\N$, we have using \eqref{prop:Est:lem1:1b:2} and \eqref{prop:Est:lem1:1c:1}:
			\begin{multline}
				\sum_{j_0=0}^j\ind_{j-j_0\in\lbrace-nq,\hdots,np\rbrace}(1+|j|)^{\gamma}\frac{1}{n^\frac{1}{2\mu}}\underset{\leq\exp(-cn)}{\underbrace{\exp\left(-c\left(\frac{\left|n-\frac{j-j_0}{\alpha^+}\right|}{n^\frac{1}{2\mu}}\right)^{\frac{2\mu}{2\mu-1}}\right)}}\left|h_{j_0}\right| \\ \leq \exp\left(-cn\right)\left\|h\right\|_{\ell^1_{\gamma}}\leq \exp\left(-cn\right)\left\|h\right\|_{\ell^1_{\Gamma}}.
				\label{prop:Est:lem1:1c:2}
			\end{multline}
			
			$\star$ For $n\in\N\backslash\lbrace0\rbrace$, $h\in\ell^1_{\Gamma}$ and $j\in\N$, we have using \eqref{prop:Est:lem1:1b:3} and \eqref{prop:Est:lem1:1c:1}:
			\begin{equation}
				\sum_{j_0\geq j+1}\ind_{j_0\in\left[0,-\frac{n\alpha^+}{2}\right[}\ind_{j-j_0\in\lbrace-nq,\hdots,np\rbrace}\underset{\leq(1+|j_0|)^{\gamma}}{\underbrace{(1+|j|)^{\gamma}}}\frac{1}{n^\frac{1}{2\mu}}\exp\left(-c\left(\frac{\left|n-\frac{j-j_0}{\alpha^+}\right|}{n^\frac{1}{2\mu}}\right)^{\frac{2\mu}{2\mu-1}}\right)\left|h_{j_0}\right| \lesssim \exp\left(-\frac{\tilde{c}}{2}n\right)\left\|h\right\|_{\ell^1_{\Gamma}}.
				\label{prop:Est:lem1:1c:3}
			\end{equation}
			
			$\star$ For $n\in\N\backslash\lbrace0\rbrace$, $h\in\ell^1_{\Gamma}$ and $j\in\N$, we have:
			\begin{align}
				\begin{split}
					\hspace{2cm}\mathclap{\hspace{4cm}\sum_{j_0\geq j+1}\ind_{j_0\in\left[-\frac{n\alpha^+}{2},+\infty\right[}\ind_{j-j_0\in\lbrace-nq,\hdots,np\rbrace}\underset{\leq(1+|j_0|)^{\gamma}}{\underbrace{(1+|j|)^{\gamma}}}\frac{1}{n^\frac{1}{2\mu}}\exp\left(-c\left(\frac{\left|n-\frac{j-j_0}{\alpha^+}\right|}{n^\frac{1}{2\mu}}\right)^{\frac{2\mu}{2\mu-1}}\right)\left|h_{j_0}\right|} & \\
					&\lesssim \sum_{j_0\in\N}\ind_{j_0\in\left[-\frac{n\alpha^+}{2},+\infty\right[} \frac{1}{n^\frac{1}{2\mu}} \frac{(1+|j_0|)^{\Gamma}}{|j_0|^{\Gamma-\gamma}}|h_{j_0}|\\
					&\lesssim \frac{1}{n^{\Gamma-\gamma+\frac{1}{2\mu}}}\left\|h\right\|_{\ell^1_{\Gamma}}.
				\end{split}\label{prop:Est:lem1:1c:4}
			\end{align}
			
			Using \eqref{prop:Est:lem1:1a:Overall} and \eqref{prop:Est:lem1:1c:2}-\eqref{prop:Est:lem1:1c:4} allows us to conclude the proof of \eqref{prop:Est:lem1:1c}.
		\end{subequations}

		\underline{$\blacktriangleright$ \textbf{Proof of the estimates \eqref{prop:Est:lem1:2} on the second term:}}
		
		In this part, we will only prove the estimates for the sequence $\left(T^+_R(h,n,j)\right)_{j\in\Z}$. Everything can then easily be extended for the sequence $\left(T^-_R(h,n,j)\right)_{j\in\Z}$. 
		
		First, using the expressions \eqref{ExpTR+} of $T_R^+(h,n,j)$ and \eqref{ExpR+} of $R^+(n,j_0,j)$, we observe that for $n\in\N\backslash\lbrace0\rbrace$ and $h\in\ell^\infty_{\Gamma}$, we have:
		\begin{align}\label{prop:Est:lem1:2:Overall}
			\begin{split}
				\left\|\left(T_R^+(h,n,j)\right)_{j\in\Z}\right\|_{\ell^1_{\gamma}}&=\sum_{j\in\Z} (1+|j|)^{\gamma} \left|T_R^+(h,n,j)\right|  \\
				&=  \sum_{j\in\Z}(1+|j|)^{\gamma}\left|\sum_{j_0\in\N}\ind_{j-j_0\in\lbrace-nq,\hdots,np\rbrace}O\left(\frac{e^{-c|j|}}{n^\frac{1}{2\mu}}\exp\left(-c\left(\frac{\left|n+\frac{j_0}{\alpha^+}\right|}{n^\frac{1}{2\mu}}\right)^{\frac{2\mu}{2\mu-1}}\right)\right)h_{j_0}\right| \\ 
				&\lesssim  \sum_{j_0\in\N} \frac{1}{n^\frac{1}{2\mu}}\exp\left(-c\left(\frac{\left|n+\frac{j_0}{\alpha^+}\right|}{n^\frac{1}{2\mu}}\right)^{\frac{2\mu}{2\mu-1}}\right)|h_{j_0}|.
			\end{split}
		\end{align}
		Thus, to prove the estimates \eqref{prop:Est:lem1:2}, we just need to find bounds for $n\in\N\backslash\lbrace0\rbrace$ on:
		$$ \sum_{j_0\in\N} \frac{1}{n^\frac{1}{2\mu}}\exp\left(-c\left(\frac{\left|n+\frac{j_0}{\alpha^+}\right|}{n^\frac{1}{2\mu}}\right)^{\frac{2\mu}{2\mu-1}}\right)|h_{j_0}|$$
		when $h\in\ell^\infty_{\Gamma}$ and when $h\in\ell^1_{\Gamma}$.
		
		\underline{$\bullet$ \textit{Proof of the estimate \eqref{prop:Est:lem1:2a}}}
		
		\begin{subequations}
			
			$\star$ We observe using \eqref{in:ctilde} that for $n\in\N\backslash\lbrace0\rbrace$ and $h\in\ell^\infty_{\Gamma}$:
			\begin{multline}\label{prop:Est:lem1:2a:1}
				\sum_{j_0\in\N}\ind_{j_0\in\left[0,-\frac{n\alpha^+}{2}\right[} \frac{1}{n^\frac{1}{2\mu}}\exp\left(-c\left(\frac{\left|n+\frac{j_0}{\alpha^+}\right|}{n^\frac{1}{2\mu}}\right)^{\frac{2\mu}{2\mu-1}}\right)|h_{j_0}| \\ \lesssim \frac{n}{n^\frac{1}{2\mu}}e^{-\tilde{c}n} \left\|h\right\|_{\ell^\infty}\lesssim e^{-\frac{\tilde{c}}{2}n} \left\|h\right\|_{\ell^\infty}\lesssim e^{-\frac{\tilde{c}}{2}n} \left\|h\right\|_{\ell^\infty_{\Gamma}} .
			\end{multline} 
			
			$\star$ We observe that for $n\in\N\backslash\lbrace0\rbrace$ and $h\in\ell^\infty_{\Gamma}$, we have using \eqref{in:sommeGauss} that:
			\begin{align}
				\begin{split}
					&\sum_{j_0\in\N}\ind_{j_0\in\left[-\frac{n\alpha^+}{2},+\infty\right[} \frac{1}{n^\frac{1}{2\mu}}\exp\left(-c\left(\frac{\left|n+\frac{j_0}{\alpha^+}\right|}{n^\frac{1}{2\mu}}\right)^{\frac{2\mu}{2\mu-1}}\right)|h_{j_0}| \\
					\leq& \sum_{j_0\in\N}\ind_{j_0\in\left[-\frac{n\alpha^+}{2},+\infty\right[} \frac{1}{n^\frac{1}{2\mu}}\exp\left(-c\left(\frac{\left|n+\frac{j_0}{\alpha^+}\right|}{n^\frac{1}{2\mu}}\right)^{\frac{2\mu}{2\mu-1}}\right)\frac{(1+|j_0|)^{\Gamma}}{|j_0|^{\Gamma}}|h_{j_0}|\\
					\lesssim& \frac{\left\|h\right\|_{\ell^\infty_{\Gamma}}}{n^{\Gamma}}.
				\end{split}\label{prop:Est:lem1:2a:2}
			\end{align}
			
			Thus, using \eqref{prop:Est:lem1:2:Overall},\eqref{prop:Est:lem1:2a:1} and \eqref{prop:Est:lem1:2a:2}, we can prove \eqref{prop:Est:lem1:2a}. 
		\end{subequations}
		
		\underline{$\bullet$ \textit{Proof of the estimate \eqref{prop:Est:lem1:2b}}}
		
		\begin{subequations}
			$\star$ First, using \eqref{prop:Est:lem1:1c:1} and \eqref{prop:Est:lem1:2a:1}, we observe that  for $n\in\N\backslash\lbrace0\rbrace$ and $h\in\ell^1_{\Gamma}$:
			\begin{equation}\label{prop:Est:lem1:2b:1}
				\sum_{j_0\in\N}\ind_{j_0\in\left[0,-\frac{n\alpha^+}{2}\right[} \frac{1}{n^\frac{1}{2\mu}}\exp\left(-c\left(\frac{\left|n+\frac{j_0}{\alpha^+}\right|}{n^\frac{1}{2\mu}}\right)^{\frac{2\mu}{2\mu-1}}\right)|h_{j_0}| \lesssim e^{-\frac{\tilde{c}}{2}n} \left\|h\right\|_{\ell^\infty_{\Gamma}} \lesssim e^{-\frac{\tilde{c}}{2}n} \left\|h\right\|_{\ell^1_{\Gamma}} .
			\end{equation} 
			
			$\star$ For $n\in\N\backslash\lbrace0\rbrace$ and $h\in\ell^1_{\Gamma}$, we have that:
			\begin{align}
				\begin{split}
					\sum_{j_0\in\N}\ind_{j_0\in\left[-\frac{n\alpha^+}{2},+\infty\right[} \frac{1}{n^\frac{1}{2\mu}}\exp\left(-c\left(\frac{\left|n+\frac{j_0}{\alpha^+}\right|}{n^\frac{1}{2\mu}}\right)^{\frac{2\mu}{2\mu-1}}\right)|h_{j_0}| & \leq \sum_{j_0\in\N}\ind_{j_0\in\left[-\frac{n\alpha^+}{2},+\infty\right[} \frac{1}{n^\frac{1}{2\mu}}\frac{(1+|j_0|)^{\Gamma}}{|j_0|^{\Gamma}}|h_{j_0}|\\
					&\lesssim \frac{\left\|h\right\|_{\ell^1_{\Gamma}}}{n^{\Gamma+\frac{1}{2\mu}}}.
				\end{split}\label{prop:Est:lem1:2b:2}
			\end{align}
			Thus, using \eqref{prop:Est:lem1:2:Overall},\eqref{prop:Est:lem1:2b:1} and \eqref{prop:Est:lem1:2b:2}, we can prove \eqref{prop:Est:lem1:2b}.
		\end{subequations}
		
		\underline{$\blacktriangleright$ \textbf{Proof of the estimates \eqref{prop:Est:lem1:3} on the third term:}}
		
		For $h\in\ell^1_{\Gamma}$ and $n\in\N\backslash\lbrace0\rbrace$, using \eqref{in:cnj0j}, we have that:
		\begin{align*}
			\sum_{j\in\Z}(1+|j|)^{\gamma}\left|\sum_{j_0\in\Z}\ind_{j-j_0\in\lbrace-nq,\hdots,np\rbrace}O(e^{-cn})h_{j_0}\right|& \lesssim (1+n)^{\gamma}e^{-cn}\sum_{j_0\in\Z}\left(\sum_{j\in\Z}\ind_{j-j_0\in\lbrace-nq,\hdots,np\rbrace}\right)(1+|j_0|)^{\gamma}|h_{j_0}|\\
			&\lesssim n(1+n)^{\gamma}e^{-cn} \left\|h\right\|_{\ell^1_{\Gamma}}\\
			&\lesssim e^{-\frac{c}{2}n}\left\|h\right\|_{\ell^1_{\Gamma}}.
		\end{align*}
		We thus obtain \eqref{prop:Est:lem1:3a}.
		
		Similarly, for $h\in\ell^\infty_{\Gamma}$ and $n\in\N\backslash\lbrace0\rbrace$ and $j\in\Z$, using \eqref{in:cnj0j}, we have that:
		\begin{align*}
			(1+|j|)^{\gamma}\left|\sum_{j_0\in\Z}\ind_{j-j_0\in\lbrace-nq,\hdots,np\rbrace}O(e^{-cn})h_{j_0}\right|& \lesssim (1+n)^{\gamma}e^{-cn}\underset{\lesssim n}{\underbrace{\sum_{j_0\in\Z}\ind_{j-j_0\in\lbrace-nq,\hdots,np\rbrace}}}\underset{\leq\left\|h\right\|_{\ell^\infty_{\gamma}}\leq\left\|h\right\|_{\ell^\infty_{\Gamma}}}{\underbrace{(1+|j_0|)^{\gamma}|h_{j_0}|}}\\
			&\lesssim n(1+n)^{\gamma}e^{-cn} \left\|h\right\|_{\ell^\infty_{\Gamma}}\\
			&\lesssim e^{-\frac{c}{2}n}\left\|h\right\|_{\ell^\infty_{\Gamma}}.
		\end{align*}
		We can then deduce  \eqref{prop:Est:lem1:3b}.
		
		\underline{$\blacktriangleright$ \textbf{Proof of the estimates \eqref{prop:Est:lem1:4} on the fourth term:}}

		We start the proof of \eqref{prop:Est:lem1:4} with an observation. Using the definition \eqref{ExpTE} of $T_E(h,n,j)$, for $n\in\N\backslash\lbrace0\rbrace$ and $h\in\E_{\Gamma}$, we have:
		\begin{align*}
			&\left\|\left(T_E(h,n,j)\right)_{j\in\Z}\right\|_{\ell^1_{\gamma}} \\
			= &\left\|\left(\left(\sum_{j_0\in\N}\ind_{j-j_0\in\lbrace-nq,\hdots,np\rbrace}E^+(n,j_0)h_{j_0}+\sum_{j_0\in-\N\backslash\lbrace0\rbrace}\ind_{j-j_0\in\lbrace-nq,\hdots,np\rbrace}E^-(n,j_0)h_{j_0}\right)V(j)\right)_{j\in\Z}\right\|_{\ell^1_{\gamma}} \\
			\leq & \left\|\left(\left(\sum_{j_0\in\N}\ind_{j-j_0\notin\lbrace-nq,\hdots,np\rbrace}E^+(n,j_0)h_{j_0}+\sum_{j_0\in-\N\backslash\lbrace0\rbrace}\ind_{j-j_0\notin\lbrace-nq,\hdots,np\rbrace}E^-(n,j_0)h_{j_0}\right)V(j)\right)_{j\in\Z}\right\|_{\ell^1_{\gamma}} \\
			& + \left\|\left(\left(\sum_{j_0\in\N}E^+(n,j_0)h_{j_0}+\sum_{j_0\in-\N\backslash\lbrace0\rbrace}E^-(n,j_0)h_{j_0}\right)V(j)\right)_{j\in\Z}\right\|_{\ell^1_{\gamma}}
		\end{align*}
		Thus, we have that for all $n\in\N\backslash\lbrace0\rbrace$ and $h\in\E_{\Gamma}$:
		\begin{align}
			\begin{split}
				&\left\|\left(T_E(h,n,j)\right)_{j\in\Z}\right\|_{\ell^1_{\gamma}} \\
				\leq & \left\|\left(\left(\sum_{j_0\in\N}\ind_{j-j_0\notin\lbrace-nq,\hdots,np\rbrace}E^+(n,j_0)h_{j_0}+\sum_{j_0\in-\N\backslash\lbrace0\rbrace}\ind_{j-j_0\notin\lbrace-nq,\hdots,np\rbrace}E^-(n,j_0)h_{j_0}\right)V(j)\right)_{j\in\Z}\right\|_{\ell^1_{\gamma}} \\
				& + \left|\sum_{j_0\in\N}E^+(n,j_0)h_{j_0}+\sum_{j_0\in-\N\backslash\lbrace0\rbrace}E^-(n,j_0)h_{j_0}\right|\left\|V\right\|_{\ell^1_{\gamma}}.
			\end{split}\label{prop:Est:lem1:4:1}
		\end{align}
		Furthermore, using a similar proof as for \cite[(2.10)]{Coeuret2024d}, we can prove that there exist two positive constants $C,c>0$ such that:
		\begin{multline}\label{prop:Est:lem1:4:2}
			\forall n\in\N\backslash\lbrace0\rbrace,\forall h\in\ell^\infty(\Z),\\
			\left\|\left(\left(\sum_{j_0\in\N}\ind_{j-j_0\notin\lbrace-nq,\hdots,np\rbrace}E^+(n,j_0)h_{j_0}+\sum_{j_0\in-\N\backslash\lbrace0\rbrace}\ind_{j-j_0\notin\lbrace-nq,\hdots,np\rbrace}E^-(n,j_0)h_{j_0}\right)V(j)\right)_{j\in\Z}\right\|_{\ell^1_{\gamma}} \\ \leq Ce^{-cn}\left\|h\right\|_{\ell^\infty}.
		\end{multline}
		
		Combining \eqref{prop:Est:lem1:4:1} and \eqref{prop:Est:lem1:4:2}, we have that for all $n\in\N\backslash\lbrace0\rbrace$ and $h\in\E_{\Gamma}$:
		\begin{equation}\label{prop:Est:lem1:4:3}
			\left\|\left(T_E(h,n,j)\right)_{j\in\Z}\right\|_{\ell^1_{\gamma}} \leq Ce^{-cn}\left\|h\right\|_{\ell^1_{\Gamma}} + \left|\sum_{j_0\in\N}E^+(n,j_0)h_{j_0}+\sum_{j_0\in-\N\backslash\lbrace0\rbrace}E^-(n,j_0)h_{j_0}\right|\left\|V\right\|_{\ell^1_{\gamma}}.
		\end{equation}
		Therefore, if we prove that there exists a constant $\widetilde{C}>0$ such that:
		\begin{equation}\label{prop:Est:lem1:4:4}
			\forall h\in\E_{\Gamma},\forall n \in\N\backslash\lbrace0\rbrace,\quad \left|\sum_{j_0\in\N}E^+(n,j_0)h_{j_0}+\sum_{j_0\in-\N\backslash\lbrace0\rbrace}E^-(n,j_0)h_{j_0}\right|\leq \frac{\widetilde{C}}{n^{\Gamma}}\left\|h\right\|_{\ell^1_{\Gamma}},
		\end{equation}
		then, combining \eqref{prop:Est:lem1:4:3} and \eqref{prop:Est:lem1:4:4} allows us to conclude on the proof of \eqref{prop:Est:lem1:4}. Therefore, there just remains to prove \eqref{prop:Est:lem1:4:4}. 
		\begin{subequations}
			
			Using the definition \eqref{ExpE} of the terms $E^\pm$ and the zero-mass assumption on $h$, we have that:
			\begin{align}\label{prop:Est:lem1:4:5}
				\begin{split}
					&\left|\sum_{j_0\in\N}E^+(n,j_0)h_{j_0}+\sum_{j_0\in-\N\backslash\lbrace0\rbrace}E^-(n,j_0)h_{j_0}\right|\\
					=& |C_E|\left|\sum_{j_0\in\N}E_{2\mu}\left(\beta^+;\frac{n\alpha^++j_0}{n^\frac{1}{2\mu}}\right)h_{j_0}+\sum_{j_0\in-\N\backslash\lbrace0\rbrace}E_{2\mu}\left(\beta^-;\frac{-n\alpha^--j_0}{n^\frac{1}{2\mu}}\right)h_{j_0} -\underset{=0}{\underbrace{\sum_{j_0\in\Z}h_{j_0}}}\right| \\ 
					\leq&  |C_E|\left(\sum_{j_0\in\N}\left|E_{2\mu}\left(\beta^+;\frac{n\alpha^++j_0}{n^\frac{1}{2\mu}}\right)-1\right|\left|h_{j_0}\right|+\sum_{j_0\in-\N\backslash\lbrace0\rbrace}\left|E_{2\mu}\left(\beta^-;\frac{-n\alpha^--j_0}{n^\frac{1}{2\mu}}\right)-1\right|\left|h_{j_0}\right|\right).
				\end{split}
			\end{align}
			We now need to prove estimates on the terms on the right-hand side of \eqref{prop:Est:lem1:4:5}. We will decompose the sums with respect to $j_0$ according to various regimes depending on $n$.
			
			First, using the inequality \eqref{inE-} on the function $E_{2\mu}$, we have that there exists a positive constant $c$ such that for all $n\in\N\backslash\lbrace0\rbrace$ and $h\in\E_{\Gamma}$:
			\begin{align}\label{prop:Est:lem1:4:6}
				\begin{split}
					\hspace{3cm}\mathclap{\sum_{j_0\in\N}\ind_{j_0\in\left[0,-\frac{n\alpha^+}{2}\right[}\left|E_{2\mu}\left(\beta^+;\frac{n\alpha^++j_0}{n^\frac{1}{2\mu}}\right)-1\right|\left|h_{j_0}\right| }&\\
					&\lesssim \sum_{j_0\in\N}\ind_{j_0\in\left[0,-\frac{n\alpha^+}{2}\right[} \frac{1}{n^\frac{1}{2\mu}}\exp\left(-c\left(\frac{\left|n+\frac{j_0}{\alpha^+}\right|}{n^\frac{1}{2\mu}}\right)^\frac{2\mu}{2\mu-1}\right)|h_{j_0}|\\
					&\lesssim \underset{\lesssim n}{\underbrace{\sum_{j_0\in\N}\ind_{j_0\in\left[0,-\frac{n\alpha^+}{2}\right[}}} \frac{1}{n^\frac{1}{2\mu}}\exp\left(-\frac{c}{2^\frac{2\mu}{2\mu-1}}n\right)\left\|h\right\|_{\ell^\infty}\\
					& \lesssim e^{-\frac{c}{4}n}\left\|h\right\|_{\ell^1_{\Gamma}}
				\end{split}
			\end{align}
			and similarly:
			\begin{equation}\label{prop:Est:lem1:4:7}
				\sum_{j_0\in-\N\backslash\lbrace0\rbrace}\ind_{j_0\in\left]-\frac{n\alpha^-}{2},0\right]}\left|E_{2\mu}\left(\beta^-;\frac{-n\alpha^--j_0}{n^\frac{1}{2\mu}}\right)-1\right|\left|h_{j_0}\right|\lesssim e^{-\frac{c}{4}n}\left\|h\right\|_{\ell^1_{\Gamma}}.
			\end{equation}
			Furthermore, for all $n\in\N\backslash\lbrace0\rbrace$ and $h\in\E_{\Gamma}$:
			\begin{multline}\label{prop:Est:lem1:4:8}
				\sum_{j_0\in\N}\ind_{j_0\in\left[-\frac{n\alpha^+}{2},+\infty\right[}\underset{\leq 2}{\underbrace{\left|E_{2\mu}\left(\beta^+;\frac{n\alpha^++j_0}{n^\frac{1}{2\mu}}\right)-1\right|}}\left|h_{j_0}\right|\\ \lesssim \sum_{j_0\in\N}\ind_{j_0\in\left[-\frac{n\alpha^+}{2},+\infty\right[}\frac{(1+|j_0|)^{\Gamma}}{|j_0|^{\Gamma}} |h_{j_0}|\lesssim \frac{\left\|h\right\|_{\ell^1_{\Gamma}}}{n^{\Gamma}}\hspace{1cm}
			\end{multline} 
			and similarly:
			\begin{equation}\label{prop:Est:lem1:4:9}
				\sum_{j_0\in-\N\backslash\lbrace0\rbrace}\ind_{j_0\in\left]-\infty,-\frac{n\alpha^-}{2}\right]}\underset{\leq 2}{\underbrace{\left|E_{2\mu}\left(\beta^-;\frac{-n\alpha^--j_0}{n^\frac{1}{2\mu}}\right)-1\right|}}\left|h_{j_0}\right|\\ \lesssim\frac{\left\|h\right\|_{\ell^1_{\Gamma}}}{n^{\Gamma}}.
			\end{equation} 
		\end{subequations}
		Thus, combining \eqref{prop:Est:lem1:4:5}-\eqref{prop:Est:lem1:4:9}, we conclude the proof of \eqref{prop:Est:lem1:4:4} and thus of \eqref{prop:Est:lem1:4}. Thus, this concludes the proof of Lemma \ref{prop:Est:lem1}.
	\end{proof}
	
	\subsection{Proof of the estimates \texorpdfstring{\eqref{prop:Est:L^n(id-T):1,1}}{}, \texorpdfstring{\eqref{prop:Est:L^n(id-T):infty,1}}{} and \texorpdfstring{\eqref{prop:Est:L^n(id-T):infty,infty}}{} on the operator \texorpdfstring{$\Ldsp^n(Id-\Tc)$}{Ln(Id-T)}}\label{subsec:prop:Est:2}
	
	The methodology to prove the estimates \eqref{prop:Est:L^n(id-T):1,1}, \eqref{prop:Est:L^n(id-T):infty,1} and \eqref{prop:Est:L^n(id-T):infty,infty} on the operators $\Ldsp^n(Id-\Tc)$ will be exactly the same as for the proof of  the estimates \eqref{prop:Est:L^n:1} and \eqref{prop:Est:L^n:infty} on the operators $\Ldsp^n$
	
	Using the definition \eqref{def:Tc} of the operator $\Tc$ and the equality \eqref{lienLccGreen} above, we have that for $n\in\N$, $h\in\ell^\infty(\Z)$ and for $j\in\Z$:
	$$\left(\Ldsp^n(Id-\Tc)h\right)_j = \sum_{j_0\in\Z} \ind_{j-j_0\in\lbrace-nq,\hdots,np\rbrace}\Gcc(n,j_0,j)(h_{j_0}-h_{j_0+1}).$$
	Thus, we have that:
	\begin{multline}\label{lienLcc(Id-Tc)DerGreen}
		\forall n\in\N,\forall h\in\ell^\infty(\Z),\\ \Ldsp^n(Id-\Tc)h= \left(\sum_{j_0\in\Z}\ind_{j-j_0\in\lbrace-nq-1,\hdots,np\rbrace}\left(\Gcc(n,j_0,j)-\Gcc(n,j_0-1,j)\right)h_{j_0}\right)_{j\in\Z}.
	\end{multline}
	We recall that, just like for the Green's function, in the introduction of this paper, we have presented a decomposition \eqref{decompoDerGreen} of the discrete derivative of the Green's function in the scalar case. We have that there exists a constant $c>0$ such that for $j_0\in\Z$, $n\in\N\backslash\lbrace0\rbrace$ and $j\in\Z$ such that $j-j_0\in\lbrace-nq-1,\hdots,np\rbrace$, the discrete derivative of the Green's function verifies:
	\begin{subequations}\label{decompoDerGreenImproved}
		\begin{align}
			\begin{split}
				\forall j_0\in\N,\quad \Gcc(n,j_0,j)-\Gcc(n,j_0-1,j)  = D_1^+(n,j_0,j) &+ D_2^+(n,j_0,j) \\ &+ D_3^+(n,j_0,j) + R^+(n,j_0,j) +O(e^{-cn}),
			\end{split}\\
			\begin{split}
				\forall j_0\in-\N\backslash\lbrace0\rbrace,\quad \Gcc(n,j_0,j)-\Gcc(n,j_0-1,j)  = D_1^-(n,j_0,j) &+ D_2^-(n,j_0,j)\\ & + D_3^-(n,j_0,j) + R^-(n,j_0,j) +O(e^{-cn}),
			\end{split}
		\end{align}
	\end{subequations}
	where:
	\begin{itemize}
		\item The terms $D^\pm_1(n,j_0,j)$,  $D^\pm_2(n,j_0,j)$ and  $D^\pm_3(n,j_0,j)$ correspond to the remainder of the diffusion waves:
		\begin{subequations}\label{ExpDpm1}
			\begin{align}
				\forall j_0\in\N, \quad & D^+_1(n,j_0,j) = \ind_{j\geq0}O\left(\frac{1}{n^\frac{1}{\mu}}\exp\left(-c\left(\frac{\left|n-\left(\frac{j-j_0}{\alpha^+}\right)\right|}{n^\frac{1}{2\mu}}\right)^\frac{2\mu}{2\mu-1}\right)\right),\label{ExpD+1}\\
				\forall j_0\in-\N\backslash\lbrace0\rbrace, \quad & D^-_1(n,j_0,j) = \ind_{j\leq0}O\left(\frac{1}{n^\frac{1}{\mu}}\exp\left(-c\left(\frac{\left|n-\left(\frac{j-j_0}{\alpha^-}\right)\right|}{n^\frac{1}{2\mu}}\right)^\frac{2\mu}{2\mu-1}\right)\right),
			\end{align}
		\end{subequations}
		\begin{subequations}
			\begin{align}
				\forall j_0\in\N, \quad & D^+_2(n,j_0,j) = \ind_{j\geq0}O\left(\frac{e^{-c|j|}}{n^\frac{1}{2\mu}}\exp\left(-c\left(\frac{\left|n-\left(\frac{j-j_0}{\alpha^+}\right)\right|}{n^\frac{1}{2\mu}}\right)^\frac{2\mu}{2\mu-1}\right)\right),\label{ExpD+2}\\
				\forall j_0\in-\N\backslash\lbrace0\rbrace, \quad & D^-_2(n,j_0,j) = \ind_{j\leq0}O\left(\frac{e^{-c|j|}}{n^\frac{1}{2\mu}}\exp\left(-c\left(\frac{\left|n-\left(\frac{j-j_0}{\alpha^-}\right)\right|}{n^\frac{1}{2\mu}}\right)^\frac{2\mu}{2\mu-1}\right)\right),
			\end{align}
		\end{subequations}
		\begin{subequations}
			\begin{align}
				\forall j_0\in\N, \quad & D^+_3(n,j_0,j) = \ind_{j\geq0}O\left(\frac{e^{-c|j_0|}}{n^\frac{1}{2\mu}}\exp\left(-c\left(\frac{\left|n-\left(\frac{j-j_0}{\alpha^+}\right)\right|}{n^\frac{1}{2\mu}}\right)^\frac{2\mu}{2\mu-1}\right)\right),\label{ExpD+3}\\
				\forall j_0\in-\N\backslash\lbrace0\rbrace, \quad & D^-_3(n,j_0,j) = \ind_{j\leq0}O\left(\frac{e^{-c|j_0|}}{n^\frac{1}{2\mu}}\exp\left(-c\left(\frac{\left|n-\left(\frac{j-j_0}{\alpha^-}\right)\right|}{n^\frac{1}{2\mu}}\right)^\frac{2\mu}{2\mu-1}\right)\right).
			\end{align}
		\end{subequations}
		
		\item The terms $R^\pm(n,j_0,j)$ correspond to a remainder term for the activation of the eigenvector associated with the eigenvalue $1$ of the operator $\Ldsp$:
		\begin{subequations}
			\begin{align}
				\forall j_0\in\N, \quad & R^+(n,j_0,j) =O\left(\frac{e^{-c|j|}}{n^\frac{1}{2\mu}}\exp\left(-c\left(\frac{\left|n+\frac{j_0}{\alpha^+}\right|}{n^\frac{1}{2\mu}}\right)^\frac{2\mu}{2\mu-1}\right)\right),\label{ExpR+bis}\\
				\forall j_0\in-\N\backslash\lbrace0\rbrace, \quad & R^-(n,j_0,j) =O\left(\frac{e^{-c|j|}}{n^\frac{1}{2\mu}}\exp\left(-c\left(\frac{\left|n+\frac{j_0}{\alpha^-}\right|}{n^\frac{1}{2\mu}}\right)^\frac{2\mu}{2\mu-1}\right)\right).
			\end{align}
		\end{subequations}
	\end{itemize} 
	
	Applying the decompositions \eqref{decompoDerGreenImproved} of the discrete derivative of the Green's function in the right-hand side term of \eqref{lienLcc(Id-Tc)DerGreen}, we have that for $h\in\ell^\infty(\Z)$, $n\in\N\backslash\lbrace0\rbrace$ and $j\in\Z$:
	\begin{multline}\label{lienLcc(Id-Tc)DerGreenImproved}
		(\Ldsp^n(Id-\Tc)h)_j = T^+_{D,1}(h,n,j)+T^+_{D,2}(h,n,j)+T^+_{D,3}(h,n,j)\\
		+T^-_{D,1}(h,n,j)+T^-_{D,2}(h,n,j)+T^-_{D,3}(h,n,j)+	T^+_R(h,n,j)+ T^-_R(h,n,j) \\+\sum_{j_0\in\Z} \ind_{j-j_0\in\lbrace-nq-1,\hdots,np\rbrace} O(e^{-cn}) h_{j_0}
	\end{multline}
	where the terms $T^\pm_{D,k}(h,n,j)$ for $k\in\lbrace1,2,3\rbrace$ and $T^\pm_R(h,n,j)$ are respectively defined by:
	\begin{subequations}
		\begin{align}
			T^+_{D,k}(h,n,j)&:= \sum_{j_0\in\N} \ind_{j-j_0\in\lbrace-nq-1,\hdots,np\rbrace} D^+_k(n,j_0,j)h_{j_0},\label{ExpTD+bis}\\
			T^-_{D,k}(h,n,j)&:= \sum_{j_0\in-\N\backslash\lbrace0\rbrace} \ind_{j-j_0\in\lbrace-nq-1,\hdots,np\rbrace} D^-_k(n,j_0,j)h_{j_0},\\
			T^+_R(h,n,j)&:= \sum_{j_0\in\N} \ind_{j-j_0\in\lbrace-nq-1,\hdots,np\rbrace} R^+(n,j_0,j)h_{j_0},\label{ExpTR+bis}\\
			T^-_R(h,n,j)&:= \sum_{j_0\in-\N\backslash\lbrace0\rbrace} \ind_{j-j_0\in\lbrace-nq-1,\hdots,np\rbrace} R^-(n,j_0,j)h_{j_0}.
		\end{align}
	\end{subequations}
	
	The following lemma is analogous to Lemma \ref{prop:Est:lem1}. We will prove sharp estimates on the terms that appear in the right-hand side of \eqref{lienLcc(Id-Tc)DerGreenImproved}.
	
	\begin{lemma}\label{prop:Est:lem2}
		For $0\leq \gamma\leq \Gamma$, there exists a constant $C>0$ such that:\newline
		$\bullet$ \textbf{\underline{First term of the decomposition:}}
		\begin{subequations}\label{prop:Est:lem2:1}
			\begin{align}
				\begin{split}
					\forall n\in\N\backslash\lbrace0\rbrace,\forall h\in\ell^1_{\Gamma},\quad &\left\|\left(T^\pm_{D,1}(h,n,j)\right)_{j\in\Z}\right\|_{\ell^1_{\gamma}} \leq \frac{C}{n^{\Gamma-\gamma+\frac{1}{2\mu}}} \left\|h\right\|_{\ell^1_{\Gamma}},
				\end{split}\label{prop:Est:lem2:1a}\\
				\begin{split}
					\forall n\in\N\backslash\lbrace0\rbrace,\forall h\in\ell^1_{\Gamma},\quad&\left\|\left(T^\pm_{D,1}(h,n,j)\right)_{j\in\Z}\right\|_{\ell^\infty_{\gamma}} \leq \frac{C}{n^{\Gamma-\gamma+\frac{1}{\mu}}} \left\|h\right\|_{\ell^1_{\Gamma}},
				\end{split}\label{prop:Est:lem2:1b}\\
				\begin{split}
					\forall n\in\N\backslash\lbrace0\rbrace,\forall h\in\ell^\infty_{\Gamma},\quad&\left\|\left(T^\pm_{D,1}(h,n,j)\right)_{j\in\Z}\right\|_{\ell^\infty_{\gamma}} \leq \frac{C}{n^{\Gamma-\gamma+\frac{1}{2\mu}}} \left\|h\right\|_{\ell^\infty_{\Gamma}}.
				\end{split}\label{prop:Est:lem2:1c}
			\end{align}
		\end{subequations}	
		$\bullet$ \textbf{\underline{Second term of the decomposition:}}
		\begin{subequations}\label{prop:Est:lem2:2}
			\begin{align}
				\begin{split}
					\forall n\in\N\backslash\lbrace0\rbrace,\forall h\in\ell^1_{\Gamma},\quad&\left\|\left(T^\pm_{D,2}(h,n,j)\right)_{j\in\Z}\right\|_{\ell^1_{\gamma}} \leq  \frac{C}{n^{\Gamma+\frac{1}{2\mu}}}\left\|h\right\|_{\ell^1_{\Gamma}},
				\end{split}\label{prop:Est:lem2:2a}\\
				\begin{split}
					\forall n\in\N\backslash\lbrace0\rbrace,\forall h\in\ell^\infty_{\Gamma},\quad&\left\|\left(T^\pm_{D,2}(h,n,j)\right)_{j\in\Z}\right\|_{\ell^1_{\gamma}} \leq \frac{C}{n^{\Gamma}} \left\|h\right\|_{\ell^\infty_{\Gamma}}.
				\end{split}\label{prop:Est:lem2:2b}
			\end{align}
		\end{subequations}	
		$\bullet$ \textbf{\underline{Third term of the decomposition:}} There exists a constant $\hat{c}>0$ such that:
		\begin{subequations}\label{prop:Est:lem2:3}
			\begin{align}
				\begin{split}
					\forall n\in\N\backslash\lbrace0\rbrace,\forall h\in\ell^1_{\Gamma},\quad&\left\|\left(T^\pm_{D,3}(h,n,j)\right)_{j\in\Z}\right\|_{\ell^1_{\gamma}} \leq Ce^{-\hat{c}n} \left\|h\right\|_{\ell^1_{\Gamma}},
				\end{split}\label{prop:Est:lem2:3a}\\
				\begin{split}
					\forall n\in\N\backslash\lbrace0\rbrace,\forall h\in\ell^\infty_{\Gamma},\quad&\left\|\left(T^\pm_{D,3}(h,n,j)\right)_{j\in\Z}\right\|_{\ell^\infty_{\gamma}} \leq Ce^{-\hat{c}n} \left\|h\right\|_{\ell^\infty_{\Gamma}}.
				\end{split}\label{prop:Est:lem2:3b}
			\end{align}
		\end{subequations}	
		$\bullet$ \textbf{\underline{Fourth term of the decomposition:}}
		\begin{subequations}\label{prop:Est:lem2:4}
			\begin{align}
				\begin{split}
					\forall n\in\N\backslash\lbrace0\rbrace,\forall h\in\ell^1_{\Gamma},\quad&\left\|\left(T^\pm_R(h,n,j)\right)_{j\in\Z}\right\|_{\ell^1_{\gamma}}\leq \frac{C}{n^{\Gamma+\frac{1}{2\mu}}} \left\|h\right\|_{\ell^1_{\Gamma}},
				\end{split}\label{prop:Est:lem2:4a}\\
				\begin{split}
					\forall n\in\N\backslash\lbrace0\rbrace,\forall h\in\ell^\infty_{\Gamma},\quad&\left\|\left(T^\pm_R(h,n,j)\right)_{j\in\Z}\right\|_{\ell^1_{\gamma}}\leq \frac{C}{n^{\Gamma}} \left\|h\right\|_{\ell^\infty_{\Gamma}}.
				\end{split}\label{prop:Est:lem2:4b}
			\end{align}
		\end{subequations}	
		$\bullet$ \textbf{\underline{Fifth term of the decomposition:}}
		\begin{subequations}\label{prop:Est:lem2:5}
			\begin{align}
				\forall n\in\N\backslash\lbrace0\rbrace,\forall h\in\ell^1_{\Gamma},\quad &\left\|\left(\sum_{j_0\in\Z}\ind_{j-j_0\in\lbrace-nq-1,\hdots,np\rbrace}O(e^{-cn})h_{j_0}\right)_{j\in\Z}\right\|_{\ell^1_{\gamma}}\leq C e^{-\frac{c}{2}n}\left\|h\right\|_{\ell^1_{\Gamma}},\label{prop:Est:lem2:5a}\\
				\forall n\in\N\backslash\lbrace0\rbrace,\forall h\in\ell^\infty_{\Gamma}, \quad & \left\|\left(\sum_{j_0\in\Z}\ind_{j-j_0\in\lbrace-nq-1,\hdots,np\rbrace}O(e^{-cn})h_{j_0}\right)_{j\in\Z}\right\|_{\ell^\infty_{\gamma}}\leq C e^{-\frac{c}{2}n}\left\|h\right\|_{\ell^\infty_{\Gamma}}.\label{prop:Est:lem2:5b}
			\end{align}
		\end{subequations}
	\end{lemma}
	
	Combining the results of Lemma \ref{prop:Est:lem2} and the equality \eqref{lienLcc(Id-Tc)DerGreenImproved}, we can immediately obtain the inequalities \eqref{prop:Est:L^n(id-T):1,1}, \eqref{prop:Est:L^n(id-T):infty,1} and \eqref{prop:Est:L^n(id-T):infty,infty} on the family of operators $(\Ldsp^n(Id-\Tc))_{n\in\N}$. There just remains to conclude on the proof of Lemma \ref{prop:Est:lem2}.
	
	\begin{proof}\textbf{of Lemma \ref{prop:Est:lem2}}
		
		Let us observe that the estimates \eqref{prop:Est:lem2:4} for the fourth term and \eqref{prop:Est:lem2:5} for the fifth term have already been proved in Lemma \ref{prop:Est:lem1}. Furthermore, the term $D^\pm_1$ defined by \eqref{ExpDpm1} corresponds to the term $D^\pm$ defined by \eqref{ExpDpm} with a factor $\frac{1}{n^\frac{1}{\mu}}$ rather than $\frac{1}{n^\frac{1}{2\mu}}$. Thus, the estimates \eqref{prop:Est:lem2:1} for the first term of Lemma \ref{prop:Est:lem2} corresponds to the estimates \eqref{prop:Est:lem1:1} for the first term of Lemma \ref{prop:Est:lem1}. There will only remain to prove \eqref{prop:Est:lem2:2} and \eqref{prop:Est:lem2:3}.
		
		\underline{$\blacktriangleright$ \textbf{Proof of the estimates \eqref{prop:Est:lem2:2} on the second term:}}
		
		\begin{subequations}
			The proof of \eqref{prop:Est:lem2:2} is fairly similar to \eqref{prop:Est:lem1:2}. For $n\in\N\backslash\lbrace0\rbrace$ and $h\in\ell^\infty_{\Gamma}$, using the expressions \eqref{ExpTD+bis} of $T_{D,2}^+(h,n,j)$ and \eqref{ExpD+2} of $D_2^+(n,j_0,j)$, we have that:
			\begin{align}\label{prop:Est:lem2:2:1}
				\begin{split}
					& \left\|\left(T^+_{D,2}(h,n,j)\right)_{j\in\Z}\right\|_{\ell^1_{\gamma}}\\
					=&\sum_{j\in\Z}(1+|j|)^{\gamma}\left|\sum_{j_0\in\N}\ind_{j-j_0\in\lbrace-nq-1,\hdots,np\rbrace}\ind_{j\geq0}O\left(\frac{e^{-c|j|}}{n^\frac{1}{2\mu}}\exp\left(-c\left(\frac{\left|n-\frac{j-j_0}{\alpha^+}\right|}{n^\frac{1}{2\mu}}\right)^{\frac{2\mu}{2\mu-1}}\right)\right)h_{j_0}\right| \\ 
					\lesssim& \sum_{j\in\N}\sum_{j_0\in\N}\ind_{j-j_0\in\lbrace-nq-1,\hdots,np\rbrace}e^{-\frac{c}{2}|j|} \frac{1}{n^\frac{1}{2\mu}}\exp\left(-c\left(\frac{\left|n-\frac{j-j_0}{\alpha^+}\right|}{n^\frac{1}{2\mu}}\right)^{\frac{2\mu}{2\mu-1}}\right)|h_{j_0}|.
				\end{split}
			\end{align}
			We will separate the sum on the right-hand side of \eqref{prop:Est:lem2:2:1}.
			
			$\bullet$ For $n\in\N\backslash\lbrace0\rbrace$ and $h\in\ell^\infty_{\Gamma}$, we have that:
			\begin{multline}\label{prop:Est:lem2:2:2}
				\sum_{j\in\N}\sum_{j_0=0}^j\ind_{j-j_0\in\lbrace-nq-1,\hdots,np\rbrace}e^{-\frac{c}{2}|j|} \frac{1}{n^\frac{1}{2\mu}}\exp\left(-c\left(\frac{\left|n-\frac{j-j_0}{\alpha^+}\right|}{n^\frac{1}{2\mu}}\right)^{\frac{2\mu}{2\mu-1}}\right)|h_{j_0}|\\ \leq\sum_{j\in\N}e^{-\frac{c}{2}|j|}\underset{\lesssim n}{\underbrace{\left(\sum_{j_0=0}^j\ind_{j-j_0\in\lbrace-nq-1,\hdots,np\rbrace}\right)}}\frac{1}{n^\frac{1}{2\mu}}e^{-cn}\left\|h\right\|_{\ell^\infty}			\lesssim e^{-\frac{c}{2}n}\left\|h\right\|_{\ell^\infty_{\Gamma}}.
			\end{multline}
			
			$\bullet$ For $n\in\N\backslash\lbrace0\rbrace$ and $h\in\ell^\infty_{\Gamma}$, we have using \eqref{in:ctilde} that:
			\begin{multline}\label{prop:Est:lem2:2:3}
				\sum_{j\in\N}\sum_{j_0\geq j+1} \ind_{j_0\in\left[0,-\frac{n\alpha^+}{2}\right[}\ind_{j-j_0\in\lbrace-nq-1,\hdots,np\rbrace}e^{-\frac{c}{2}|j|} \frac{1}{n^\frac{1}{2\mu}}\exp\left(-c\left(\frac{\left|n-\frac{j-j_0}{\alpha^+}\right|}{n^\frac{1}{2\mu}}\right)^{\frac{2\mu}{2\mu-1}}\right)|h_{j_0}|\\ \leq\sum_{j\in\N}e^{-\frac{c}{2}|j|}\underset{\lesssim n}{\underbrace{\left(\sum_{j_0\geq j+1}\ind_{j_0\in\left[0,-\frac{n\alpha^+}{2}\right[}\ind_{j-j_0\in\lbrace-nq-1,\hdots,np\rbrace}\right)}}\frac{1}{n^\frac{1}{2\mu}}e^{-\tilde{c}n}\left\|h\right\|_{\ell^\infty}	\lesssim e^{-\frac{\tilde{c}}{2}n}\left\|h\right\|_{\ell^\infty_{\Gamma}}.
			\end{multline}
			
			$\bullet$ For $n\in\N\backslash\lbrace0\rbrace$ and $h\in\ell^\infty_{\Gamma}$, we have using \eqref{in:sommeGauss} that:
			\begin{align}\label{prop:Est:lem2:2:4}
				\begin{split}
					&\sum_{j\in\N}\sum_{j_0\geq j+1} \ind_{j_0\in\left[-\frac{n\alpha^+}{2},+\infty\right[}\ind_{j-j_0\in\lbrace-nq-1,\hdots,np\rbrace}e^{-\frac{c}{2}|j|} \frac{1}{n^\frac{1}{2\mu}}\exp\left(-c\left(\frac{\left|n-\frac{j-j_0}{\alpha^+}\right|}{n^\frac{1}{2\mu}}\right)^{\frac{2\mu}{2\mu-1}}\right)|h_{j_0}|\\
					\leq&\sum_{j\in\N}e^{-\frac{c}{2}|j|}\sum_{j_0\geq j+1}\ind_{j_0\in\left[-\frac{n\alpha^+}{2},+\infty\right[}\frac{1}{n^\frac{1}{2\mu}}\exp\left(-c\left(\frac{\left|n-\frac{j-j_0}{\alpha^+}\right|}{n^\frac{1}{2\mu}}\right)^{\frac{2\mu}{2\mu-1}}\right) \underset{\lesssim \frac{\left\|h\right\|_{\ell^\infty_{\Gamma}}}{n^{\Gamma}}}{\underbrace{\frac{(1+|j_0|)^{\Gamma}}{|j_0|^{\Gamma}}|h_{j_0}|}}\\
					\lesssim &\frac{\left\|h\right\|_{\ell^\infty_{\Gamma}}}{n^{\Gamma}}.
				\end{split}
			\end{align}
			
			Thus, combining \eqref{prop:Est:lem2:2:1}-\eqref{prop:Est:lem2:2:4}, we obtain \eqref{prop:Est:lem2:2a}. There remains to prove \eqref{prop:Est:lem2:2b}. To prove it, we observe that for $n\in\N\backslash\lbrace0\rbrace$ and $h\in\ell^1_{\Gamma}$, we have that:
			\begin{multline}\label{prop:Est:lem2:2:5}
				\sum_{j\in\N}\sum_{j_0\geq j+1} \ind_{j_0\in\left[-\frac{n\alpha^+}{2},+\infty\right[}\ind_{j-j_0\in\lbrace-nq-1,\hdots,np\rbrace}e^{-\frac{c}{2}|j|} \frac{1}{n^\frac{1}{2\mu}}\exp\left(-c\left(\frac{\left|n-\frac{j-j_0}{\alpha^+}\right|}{n^\frac{1}{2\mu}}\right)^{\frac{2\mu}{2\mu-1}}\right)|h_{j_0}|\\
				\leq\frac{1}{n^\frac{1}{2\mu}}\sum_{j\in\N}e^{-\frac{c}{2}|j|}\underset{\lesssim \frac{\left\|h\right\|_{\ell^1_{\Gamma}}}{n^{\Gamma}}}{\underbrace{\sum_{j_0\geq j+1}\ind_{j_0\in\left[-\frac{n\alpha^+}{2},+\infty\right[}\frac{(1+|j_0|)^{\Gamma}}{|j_0|^{\Gamma}}|h_{j_0}|}}\lesssim \frac{\left\|h\right\|_{\ell^1_{\Gamma}}}{n^{\Gamma+\frac{1}{2\mu}}}.
			\end{multline}
			Thus, combining \eqref{prop:Est:lem2:2:1}-\eqref{prop:Est:lem2:2:3} and \eqref{prop:Est:lem2:2:5}, we obtain \eqref{prop:Est:lem2:2b}.

		\end{subequations}

		\underline{$\blacktriangleright$ \textbf{Proof of the estimates \eqref{prop:Est:lem2:3} on the third term:}}
		
		Using the expressions \eqref{ExpTD+bis} of $T_{D,3}^+(h,n,j)$ and \eqref{ExpD+3} of $D^+_3(n,j_0,j)$, we observe that for $n\in\N\backslash\lbrace0\rbrace$, $h\in\ell^\infty_{\Gamma}$ and $j\in\Z$, we have:
		\begin{align}\label{prop:Est:lem2:3:Overall}
			\begin{split}
				&(1+|j|)^{\gamma} \left|T_{D,3}^+(h,n,j)\right|  \\
				= & (1+|j|)^{\gamma} \left|\sum_{j_0\in\N}\ind_{j-j_0\in\lbrace-nq-1,\hdots,np\rbrace}\ind_{j\geq0}O\left(\frac{e^{-c|j_0|}}{n^\frac{1}{2\mu}}\exp\left(-c\left(\frac{\left|n-\frac{j-j_0}{\alpha^+}\right|}{n^\frac{1}{2\mu}}\right)^{\frac{2\mu}{2\mu-1}}\right)\right)h_{j_0}\right| \\ 
				\lesssim& \ind_{j\geq0} \sum_{j_0\in\N}\ind_{j-j_0\in\lbrace-nq-1,\hdots,np\rbrace}(1+|j|)^{\gamma}\frac{e^{-c|j_0|}}{n^\frac{1}{2\mu}}\exp\left(-c\left(\frac{\left|n-\frac{j-j_0}{\alpha^+}\right|}{n^\frac{1}{2\mu}}\right)^{\frac{2\mu}{2\mu-1}}\right)\left|h_{j_0}\right|
			\end{split}
		\end{align}

		\underline{$\bullet$ \textit{Proof of the estimate \eqref{prop:Est:lem2:3a}}}
		
		\begin{subequations}
			We consider $h\in\ell^1_{\Gamma}$ and $n\in\N\backslash\lbrace0\rbrace$. To prove \eqref{prop:Est:lem2:3a}, we want to find estimates on the sum:
			$$\sum_{j\in\Z}(1+|j|)^{\gamma} \left|T_{D,3}^+(h,n,j)\right|.$$
			We observe that the inequality \eqref{prop:Est:lem2:3:Overall} implies that we only need to find bounds on 
			\begin{equation}\label{prop:Est:lem2:3a:1}
				\sum_{j\in\Z}\ind_{j\geq0} \sum_{j_0\in\N}\ind_{j-j_0\in\lbrace-nq-1,\hdots,np\rbrace}(1+|j|)^{\gamma}\frac{e^{-c|j_0|}}{n^\frac{1}{2\mu}}\exp\left(-c\left(\frac{\left|n-\frac{j-j_0}{\alpha^+}\right|}{n^\frac{1}{2\mu}}\right)^{\frac{2\mu}{2\mu-1}}\right)\left|h_{j_0}\right|.
			\end{equation}
			We will decompose the sum \eqref{prop:Est:lem2:3a:1} with respect to $j$ according to various regimes of $j-j_0$.
			
			$\star$ For $n\in\N\backslash\lbrace0\rbrace$ and $h\in\ell^1_{\Gamma}$, we have using the inequality \eqref{in:cnj0j} and the fact that $\alpha^+<0$:
			\begin{align}
				\begin{split}
					\hspace{3cm} \mathclap{\hspace{8cm}\sum_{j\in\Z}\ind_{j\geq0}\sum_{j_0=0}^j\ind_{j-j_0\in\lbrace-nq-1,\hdots,np\rbrace}(1+|j|)^{\gamma}\frac{e^{-c|j_0|}}{n^\frac{1}{2\mu}}\underset{\leq\exp(-cn)}{\underbrace{\exp\left(-c\left(\frac{\left|n-\frac{j-j_0}{\alpha^+}\right|}{n^\frac{1}{2\mu}}\right)^{\frac{2\mu}{2\mu-1}}\right)}}\left|h_{j_0}\right|}&\\
					&\lesssim \sum_{j_0\in\N} (1+|j_0|)^{\gamma}e^{-c|j_0|}\underset{\lesssim n}{\underbrace{\sum_{j\geq j_0}\ind_{j-j_0\in\lbrace-nq-1,\hdots,np\rbrace}}}\frac{(1+n)^{\gamma}}{n^\frac{1}{2\mu}}  \exp(-cn)\underset{\leq\left\|h\right\|_{\ell^1_{\Gamma}}}{\underbrace{\left\|h\right\|_{\ell^\infty}}}\\
					&\lesssim \exp\left(-\frac{c}{2}n\right)\left\|h\right\|_{\ell^1_{\Gamma}}.
				\end{split}\label{prop:Est:lem2:3a:2}
			\end{align}
			
			$\star$ For $n\in\N\backslash\lbrace0\rbrace$ and $h\in\ell^1_{\Gamma}$, we have using the inequality \eqref{in:ctilde} and the fact that $\gamma\leq \Gamma$:
			\begin{align}
				\begin{split}
					\hspace{3cm} \mathclap{\hspace{8cm}\sum_{j\in\Z}\ind_{j\geq0}\sum_{j_0\geq j+1}\ind_{j_0\in\left[0,-\frac{n\alpha^+}{2}\right[}\ind_{j-j_0\in\lbrace-nq-1,\hdots,np\rbrace}\underset{\leq(1+|j_0|)^{\gamma}}{\underbrace{(1+|j|)^{\gamma}}}\frac{e^{-c|j_0|}}{n^\frac{1}{2\mu}}\exp\left(-c\left(\frac{\left|n-\frac{j-j_0}{\alpha^+}\right|}{n^\frac{1}{2\mu}}\right)^{\frac{2\mu}{2\mu-1}}\right)\left|h_{j_0}\right|} & \\
					&\lesssim \sum_{j_0\in\N}\ind_{j_0\in\left[0,-\frac{n\alpha^+}{2}\right[} (1+|j_0|)^{\gamma}e^{-c|j_0|}\underset{\lesssim n}{\underbrace{\sum_{j=0}^{j_0-1}\ind_{j-j_0\in\lbrace-nq-1,\hdots,np\rbrace}}}\frac{1}{n^\frac{1}{2\mu}}  \exp(-\tilde{c}n)\underset{\leq\left\|h\right\|_{\ell^1_{\Gamma}}}{\underbrace{\left\|h\right\|_{\ell^\infty}}}\\
					&\lesssim \exp\left(-\frac{\tilde{c}}{2}n\right)\left\|h\right\|_{\ell^1_{\Gamma}}.
				\end{split}\label{prop:Est:lem2:3a:3}
			\end{align}
			
			$\star$ There exists a constant $\hat{c}>0$ such that, for $n\in\N\backslash\lbrace0\rbrace$ and $h\in\ell^1_{\Gamma}$, we have using the inequality \eqref{in:sommeGauss}:
			\begin{align}
				\begin{split}
					&\sum_{j\in\Z}\ind_{j\geq0}\sum_{j_0\geq j+1}\ind_{j_0\in\left[-\frac{n\alpha^+}{2},+\infty\right[}\ind_{j-j_0\in\lbrace-nq-1,\hdots,np\rbrace}\underset{\leq(1+|j_0|)^{\gamma}}{\underbrace{(1+|j|)^{\gamma}}}\frac{e^{-c|j_0|}}{n^\frac{1}{2\mu}}\exp\left(-c\left(\frac{\left|n-\frac{j-j_0}{\alpha^+}\right|}{n^\frac{1}{2\mu}}\right)^{\frac{2\mu}{2\mu-1}}\right)\left|h_{j_0}\right|\\
					\leq& \sum_{j_0\in\N}\ind_{j_0\in\left[-\frac{n\alpha^+}{2},+\infty\right[} (1+|j_0|)^{\gamma}e^{-\frac{c}{2}|j_0|}\underset{\leq e^{-\hat{c}n}}{\underbrace{e^{-\frac{c}{2}|j_0|}}} \underset{\lesssim1}{\underbrace{\left(\sum_{j=0}^{j_0-1}\frac{1}{n^\frac{1}{2\mu}} \exp\left(-c\left(\frac{\left|n-\frac{j-j_0}{\alpha^+}\right|}{n^\frac{1}{2\mu}}\right)^{\frac{2\mu}{2\mu-1}}\right)\right)}} \underset{\leq\left\|h\right\|_{\ell^1_{\Gamma}}}{\underbrace{\left\|h\right\|_{\ell^\infty}}}\\
					\lesssim &e^{-\hat{c}n}\left\|h\right\|_{\ell^1_{\Gamma}}.
				\end{split}\label{prop:Est:lem2:3a:4}
			\end{align}
			
			Using \eqref{prop:Est:lem2:3:Overall} and \eqref{prop:Est:lem2:3a:2}-\eqref{prop:Est:lem2:3a:4} allows us to conclude the proof of \eqref{prop:Est:lem2:3a}.
		\end{subequations}
		
		\underline{$\bullet$ \textit{Proof of the estimate \eqref{prop:Est:lem2:3b}}}
		
		\begin{subequations}
			The proof of \eqref{prop:Est:lem2:3b} is fairly similar to \eqref{prop:Est:lem2:3a} with slight modifications made precise below. For $n\in\N\backslash\lbrace0\rbrace$ and $h\in \ell^\infty_{\Gamma}$, we want to prove bounds on:
			$$(1+|j|)^{\gamma}\left|T_{D,3}^+(h,n,j)\right|  $$
			for $j\in\Z$. The inequality \eqref{prop:Est:lem2:3:Overall} thus leads us to decompose and study the sum
			\begin{equation}\label{prop:Est:lem2:3b:1}
				\ind_{j\geq0} \sum_{j_0\in\N}\ind_{j-j_0\in\lbrace-nq-1,\hdots,np\rbrace}(1+|j|)^{\gamma}\frac{e^{-c|j_0|}}{n^\frac{1}{2\mu}}\exp\left(-c\left(\frac{\left|n-\frac{j-j_0}{\alpha^+}\right|}{n^\frac{1}{2\mu}}\right)^{\frac{2\mu}{2\mu-1}}\right)\left|h_{j_0}\right|.
			\end{equation}
			
			$\star$ For $n\in\N\backslash\lbrace0\rbrace$, $h\in\ell^\infty_{\Gamma}$ and $j\in\N$, we have using \eqref{in:cnj0j}:
			\begin{align}
				\begin{split}
					&\sum_{j_0=0}^j\ind_{j-j_0\in\lbrace-nq-1,\hdots,np\rbrace}(1+|j|)^{\gamma}\frac{e^{-c|j_0|}}{n^\frac{1}{2\mu}}\underset{\leq\exp(-cn)}{\underbrace{\exp\left(-c\left(\frac{\left|n-\frac{j-j_0}{\alpha^+}\right|}{n^\frac{1}{2\mu}}\right)^{\frac{2\mu}{2\mu-1}}\right)}}\left|h_{j_0}\right|\\
					\lesssim& \underset{\lesssim 1}{\underbrace{\sum_{j_0=0}^j\ind_{j-j_0\in\lbrace-nq-1,\hdots,np\rbrace}(1+|j_0|)^\gamma e^{-c|j_0|}}}\frac{(1+n)^{\gamma}}{n^\frac{1}{2\mu}}  \exp(-cn) \underset{\leq \left\|h\right\|_{\ell^\infty_{\Gamma}}}{\underbrace{\left\|h\right\|_{\ell^\infty}}}\\
					\lesssim &\exp\left(-\frac{c}{2}n\right)\left\|h\right\|_{\ell^\infty_{\Gamma}}.
				\end{split}\label{prop:Est:lem2:3b:2}
			\end{align}
			
			$\star$ For $n\in\N\backslash\lbrace0\rbrace$, $h\in\ell^\infty_{\Gamma}$ and $j\in\N$, we have using the inequality \eqref{in:ctilde}:
			\begin{align}
				\begin{split}
					\hspace{3cm}\mathclap{\hspace{8cm}\sum_{j_0\geq j+1}\ind_{j_0\in\left[0,-\frac{n\alpha^+}{2}\right[}\ind_{j-j_0\in\lbrace-nq-1,\hdots,np\rbrace}\underset{\leq(1+|j_0|)^{\gamma}}{\underbrace{(1+|j|)^{\gamma}}}\frac{e^{-c|j_0|}}{n^\frac{1}{2\mu}}\exp\left(-c\left(\frac{\left|n-\frac{j-j_0}{\alpha^+}\right|}{n^\frac{1}{2\mu}}\right)^{\frac{2\mu}{2\mu-1}}\right)\left|h_{j_0}\right|} & \\
					&\lesssim \underset{\lesssim 1}{\underbrace{\sum_{j_0\geq j+1}\ind_{j_0\in\left[0,-\frac{n\alpha^+}{2}\right[}\ind_{j-j_0\in\lbrace-nq-1,\hdots,np\rbrace}(1+|j_0|)^\gamma e^{-c|j_0|}}}\frac{1}{n^\frac{1}{2\mu}}  \exp(-\tilde{c}n)\underset{\leq\left\|h\right\|_{\ell^\infty_{\Gamma}}}{\underbrace{\left\|h\right\|_{\ell^\infty}}}\\
					&\lesssim \exp\left(-\tilde{c}n\right)\left\|h\right\|_{\ell^\infty_{\Gamma}}.
				\end{split}\label{prop:Est:lem2:3b:3}
			\end{align}
			
			$\star$ There exists a constant $\hat{c}>0$ such that, for $n\in\N\backslash\lbrace0\rbrace$, $h\in\ell^\infty_{\Gamma}$ and $j\in\N$, we have using \eqref{in:sommeGauss}:
			\begin{align}
				\begin{split}
					&\sum_{j_0\geq j+1}\ind_{j_0\in\left[-\frac{n\alpha^+}{2},+\infty\right[}\ind_{j-j_0\in\lbrace-nq-1,\hdots,np\rbrace}\underset{\leq(1+|j_0|)^{\gamma}}{\underbrace{(1+|j|)^{\gamma}}}\frac{e^{-c|j_0|}}{n^\frac{1}{2\mu}}\exp\left(-c\left(\frac{\left|n-\frac{j-j_0}{\alpha^+}\right|}{n^\frac{1}{2\mu}}\right)^{\frac{2\mu}{2\mu-1}}\right)\left|h_{j_0}\right|\\
					\leq& \sum_{j_0\in\N}\ind_{j_0\in\left[-\frac{n\alpha^+}{2},+\infty\right[} \underset{\lesssim1}{\underbrace{(1+|j_0|)^{\gamma}e^{-\frac{c}{2}|j_0|}}}\underset{\leq e^{-\hat{c}n}}{\underbrace{e^{-\frac{c}{2}|j_0|}}} \frac{1}{n^\frac{1}{2\mu}} \exp\left(-c\left(\frac{\left|n-\frac{j-j_0}{\alpha^+}\right|}{n^\frac{1}{2\mu}}\right)^{\frac{2\mu}{2\mu-1}}\right) \underset{\leq\left\|h\right\|_{\ell^\infty_{\Gamma}}}{\underbrace{\left\|h\right\|_{\ell^\infty}}}\\
					\lesssim &e^{-\hat{c}n}\underset{\lesssim1}{\underbrace{\sum_{j_0\in\N}\ind_{j_0\in\left[-\frac{n\alpha^+}{2},+\infty\right[} \left(\frac{1}{n^\frac{1}{2\mu}}\exp\left(-c\left(\frac{\left|n-\frac{j-j_0}{\alpha^+}\right|}{n^\frac{1}{2\mu}}\right)^{\frac{2\mu}{2\mu-1}}\right)\right)}}\left\|h\right\|_{\ell^\infty_{\Gamma}}\\
					\lesssim &e^{-\hat{c}n}\left\|h\right\|_{\ell^\infty_{\Gamma}}.
				\end{split}\label{prop:Est:lem2:3b:4}
			\end{align}
			
			Using \eqref{prop:Est:lem2:3:Overall} and \eqref{prop:Est:lem2:3b:2}-\eqref{prop:Est:lem2:3b:4} allows us to conclude the proof of \eqref{prop:Est:lem2:3b}.
		\end{subequations}
		
		\subsection{Proof of the estimate \texorpdfstring{\eqref{prop:Est:L^n(LD-L)}}{} on the operator \texorpdfstring{$\Ldsp^n(\LdspD-\Ldsp)$}{Ln(LD-L)}}\label{subsec:prop:Est:3}
		
		Let us fix the constants $\gamma,p\in[0,+\infty[$ that intervene in \eqref{prop:Est:L^n(LD-L)}. Let us start by studying the operator $\LdspD-\Ldsp$. We recall that the operator $\LdspD$ is defined by \eqref{def:linearizedScheme} and has the form:
		\begin{equation}\label{prop:Est:3:PR1}
			\forall \delta \in I,\forall h\in\ell^\infty(\Z),\forall j\in\Z,\quad (\LdspD h)_j:= \sum_{k=-p}^qa^\delta_{j,k}h_{j+k}
		\end{equation}
		where the scalars $a^\delta_{j,k}$ are defined by \eqref{def:Ajk}. As a consequence, we have:
		\begin{equation}\label{prop:Est:3:PR2}
			\forall \delta \in I,\forall h\in\ell^\infty(\Z),\forall j\in\Z,\quad ((\LdspD-\Ldsp) h)_j= \sum_{k=-p}^{q-1}\left(b^\delta_{j,k}-b^0_{j,k}\right)h_{j+k} - \sum_{k=-p}^{q-1}\left(b^\delta_{j+1,k}-b^0_{j+1,k}\right)h_{j+1+k}
		\end{equation}
		where the scalars $b^\delta_{j,k}$ are defined by \eqref{def:Bjk}. We make the following two observations:
		
		$\bullet$ Using the inequality \eqref{DSP_CV_ExpoDelta}, the fact that the set \eqref{set:CT} is relatively compact and the regularity of the numerical flux $F$, we observe that there exist two constants $C,c>0$ such that the scalars $b^\delta_{j,k}$ defined by \eqref{def:Bjk} satisfy:
		\begin{equation}\label{prop:Est:3:PR3}
			\forall \delta \in I, \forall j\in\Z,\forall k\in\left\{-p,\hdots,q-1\right\},\quad \left|b^\delta_{j,k}-b^0_{j,k}\right|\leq C|\delta|e^{-c|j|}.
		\end{equation}
		This implies that, for $\delta\in I$ and $h\in\ell^\infty(\Z)$, the series:
		$$\sum_{j\in\Z} \left(\sum_{k=-p}^{q-1}\left(b^\delta_{j,k}-b^0_{j,k}\right)h_{j+k}\right)$$
		converges. Therefore, the equality \eqref{prop:Est:3:PR2} implies that:
		\begin{equation}\label{prop:Est:3:MasseNulle}
			\forall \delta \in I,\forall h\in\ell^\infty(\Z),\quad \sum_{j\in\Z}((\LdspD-\Ldsp) h)_j=0.
		\end{equation}

		$\bullet$ Using the definition \eqref{def:Ajk} of the scalars $a^\delta_{j,k}$ and the inequality \eqref{prop:Est:3:PR3}, we have that there exist two constants $C,c>0$ such that:
		\begin{equation}\label{prop:Est:3:PR4}
			\forall \delta \in I,\forall j\in\Z,\forall k\in\left\{-p,\hdots,q\right\}\quad \left|a^\delta_{j,k}-a^0_{j,k}\right|\leq C|\delta|e^{-c|j|}.
		\end{equation}
		Combining \eqref{prop:Est:3:PR1} and \eqref{prop:Est:3:PR4}, this implies that for $h\in\ell^1(\Z)$, $\delta\in I$ and $j\in \Z$:
		$$\left|\left(\left(\LdspD-\Ldsp\right)h\right)_j\right| = \left|\sum_{k=-p}^q\left(a^\delta_{j,k}-a^0_{j,k}\right)h_{j+k}\right|\leq C(p+q+1)e^{-c|j|}|\delta|\left\|h\right\|_{\ell^\infty}.$$
		We can then conclude that there exists a constant $C>0$ such that:
		\begin{equation}\label{prop:Est:3:in}
			\forall \delta \in I,\forall h\in \ell^1(\Z), \quad \left\|\left(\LdspD-\Ldsp\right)h\right\|_{\ell^1_{\gamma+p}}\leq C|\delta|\left\|h\right\|_{\ell^\infty}.
		\end{equation}
		
		Combining \eqref{prop:Est:3:MasseNulle} and \eqref{prop:Est:3:in}, we have that the sequence $\left(\LdspD-\Ldsp\right)h$ belongs to $\E_{\gamma+p}$ for all $h\in\ell^\infty(\Z)$ and $\delta \in I$. Finally, using \eqref{prop:Est:L^n:1} and \eqref{prop:Est:3:in}, we easily conclude the proof of \eqref{prop:Est:L^n(LD-L)}.
	\end{proof}
	
	\section{Appendix}\label{sec:Appendix}

	\noindent\textbf{\large Proof of Lemma \ref{lem:inQ}}
	
	\vspace{0.3cm}Let us recall the statement of Lemma \ref{lem:inQ}:
	\begin{lemma*}
		Let us consider two constants $\gamma_1,\gamma_\infty\in[0,+\infty[$. 
		\begin{subequations}
			\begin{itemize}
				\item There exists a constant $C_{Q,1}(\gamma_1,\gamma_\infty)>0$ such that for any $\delta \in I$ and sequence $h\in \ell^1_{\gamma_1}\cap\ell^\infty_{\gamma_\infty}$ which verifies:
				$$\left\|h\right\|_{\ell^\infty}<\Rayon,$$
				then the sequence $\QD(h)$ belongs to $\ell^1_{\gamma_1+\gamma_\infty}$ and:
				\begin{equation}
					\left\|\QD(h)\right\|_{\ell^1_{\gamma_1+\gamma_\infty}} \leq C_{Q,1}(\gamma_1,\gamma_\infty)\left\|h\right\|_{\ell^1_{\gamma_1}}\left\|h\right\|_{\ell^\infty_{\gamma_\infty}}.\label{lem:inQ:ell1PR}
				\end{equation}
				\item There exists a constant $C_{Q,\infty}(\gamma_\infty)>0$ such that for any $\delta \in I$ and sequence $h\in \ell^\infty_{\gamma_\infty}$ which verifies:
				$$\left\|h\right\|_{\ell^\infty}<\Rayon,$$
				then the sequence $\QD(h)$ belongs to $\ell^\infty_{2\gamma_\infty}$ and:
				\begin{equation}
					\left\|\QD(h)\right\|_{\ell^\infty_{2\gamma_\infty}}\leq C_{Q,\infty}(\gamma_\infty){\left\|h\right\|_{\ell^\infty_{\gamma_\infty}}}^2.\label{lem:inQ:ellinftyPR}
				\end{equation}
			\end{itemize}
		\end{subequations}
	\end{lemma*}
	
	\begin{proof}
		For the rest of the proof, we fix the two positive constants $\gamma_1$ and $\gamma_\infty$ that will appear throughout the proof. We recall that the set on the left side of \eqref{def:Rayon} is closed and contained in $\Uc$. Therefore, using the regularity of the numerical flux $F$ and Taylor's Theorem, there exists a constant $C>0$ such that:
		\begin{multline}\label{lem:inQ:1}
			\forall \delta \in I,\forall j \in\Z,\forall h_{-p},\hdots,h_{q-1}\in B(0,\Rayon),\\
			\left|\nug F\left(\nug;\dspD_{j-p}+h_{-p},\hdots,\dspD_{j+q-1}+h_{q-1}\right)-\nug F\left(\nug;\dspD_{j-p},\hdots,\dspD_{j+q-1}\right) -\sum_{k=-p}^{q-1}b^\delta_{j,k} h_k \right| \leq C\left(\sum_{k=-p}^{q-1}|h_k|\right)^2
		\end{multline}
		where the scalars $b^\delta_{j,k}$ are defined by \eqref{def:Bjk}. We also observe that there exists a positive constant $\widetilde{C}$:
		\begin{equation}\label{lem:inQ:2}
			\forall j\in\Z,\forall k\in\lbrace-p,\hdots,q-1\rbrace, \quad \begin{array}{c}
				(1+|j|)^{\gamma_1}\leq \widetilde{C} (1+|j+k|)^{\gamma_1},\\
				(1+|j|)^{\gamma_\infty}\leq \widetilde{C} (1+|j+k|)^{\gamma_\infty}.
			\end{array}
		\end{equation}
		
		Let us now start by proving \eqref{lem:inQ:ell1PR}. We consider $\delta \in I$ and $h\in\ell^1_{\gamma_1}\cap\ell^\infty_{\gamma_\infty}$ such that:
		$$\left\|h\right\|_{\ell^\infty}<\Rayon.$$
		For all $j\in\Z$, using \eqref{lem:inQ:1} and \eqref{lem:inQ:2}, we have that:
		\begin{align*}
			\left(1+|j|\right)^{\gamma_1+\gamma_\infty}\left|\QD(h)_j\right|&\leq C\widetilde{C}^2\left(\sum_{k=-p}^{q-1}\left(1+|j+k|\right)^{\gamma_1}|h_{j+k}|\right)\left(\sum_{k=-p}^{q-1}\left(1+|j+k|\right)^{\gamma_\infty}|h_{j+k}|\right)\\
			& \leq C\widetilde{C}^2(p+q)\left\|h\right\|_{\ell^\infty_{\gamma_\infty}}\left(\sum_{k=-p}^{q-1}\left(1+|j+k|\right)^{\gamma_1}|h_{j+k}|\right).
		\end{align*}
		Thus, by summing this inequality, we obtain \eqref{lem:inQ:ell1PR}.
		
		The proof of \eqref{lem:inQ:ellinftyPR} is fairly similar. We consider $\delta \in I$ and a sequence $h\in\ell^\infty_{\gamma_\infty}$ such that:
		$$\left\|h\right\|_{\ell^\infty}<\Rayon.$$
		For all $j\in\Z$, using \eqref{lem:inQ:1} and \eqref{lem:inQ:2}, we have that:
		\begin{align*}
			\left(1+|j|\right)^{2\gamma_\infty}\left|\QD(h)_j\right|&\leq C\widetilde{C}^2\left(\sum_{k=-p}^{q-1}\left(1+|j+k|\right)^{\gamma_\infty}|h_{j+k}|\right)^2\\
			& \leq C\widetilde{C}^2(p+q)^2{\left\|h\right\|_{\ell^\infty_{\gamma_\infty}}}^2.
		\end{align*}
		Thus, we obtain \eqref{lem:inQ:ellinftyPR}.
	\end{proof}
	
	\noindent\textbf{\large Proof of Lemma \ref{lem:InSum}}
	
	\vspace{0.3cm}Let us recall the statement of Lemma \ref{lem:InSum}.
	
	\begin{lemma*}
		We consider a triplet of positive constants $(a,b,c)\in[0,+\infty[^3$ that satisfies the condition \eqref{cond:H}. There exists a constant $C_I(a,b,c)>0$ such that, for all $n\in\N$, we have that:
		$$\sum_{m=0}^{\lfloor\frac{n+1}{2}\rfloor }\frac{1}{(m+1)^a(n+1-m)^b} \leq\frac{C_I(a,b,c)}{(n+2)^c}. $$
	\end{lemma*}	
	
	\begin{proof}
		The proof is fairly immediate. First, we observe that for $n\in\N$:
		$$\sum_{m=0}^{\lfloor\frac{n+1}{2}\rfloor}\frac{1}{(m+1)^a(n+1-m)^b} = \frac{1}{(n+1)^{a+b-1}}\left(\frac{1}{n+1}\sum_{m=0}^{\lfloor\frac{n+1}{2}\rfloor}\frac{1}{\left(\frac{m+1}{n+1}\right)^a\left(1-\left(\frac{m}{n+1}\right)\right)^b}\right). $$
		Since for $m\in\lbrace0,\hdots,\lfloor\frac{n+1}{2}\rfloor\rbrace$, we have that for $n\in\N$:
		$$\frac{1}{\left(1-\left(\frac{m}{n+1}\right)\right)^b}\leq 2^b,$$
		we thus have
		$$\sum_{m=0}^{\lfloor\frac{n+1}{2}\rfloor}\frac{1}{(m+1)^a(n+1-m)^b} \leq  \frac{2^b}{(n+1)^{a+b-1}}\left(\frac{1}{n+1}\sum_{m=0}^{\lfloor\frac{n+1}{2}\rfloor}\frac{1}{\left(\frac{m+1}{n+1}\right)^a}\right). $$
		We observe that:
		$$\forall n\in\N,\forall m \in\left\{1,\hdots,\lfloor\frac{n+1}{2}\rfloor\right\},\quad \frac{1}{n+1}\frac{1}{\left(\frac{m+1}{n+1}\right)^a}\leq \int_{\frac{m}{n+1}}^{\frac{m+1}{n+1}}\frac{dt}{t^a}.$$
		Thus, we have :
		$$\forall n\in\N,\quad \frac{1}{n+1}\sum_{m=0}^{\lfloor\frac{n+1}{2}\rfloor}\frac{1}{\left(\frac{m+1}{n+1}\right)^a} \leq (n+1)^{a-1}+\int_{\frac{1}{n+1}}^{1}\frac{dt}{t^a}.$$
		We can then easily conclude the proof by separating the cases $a\in[0,1[$, $a=1$ and $a>1$.
	\end{proof}
	
	\textbf{Acknowledgments:} The author wishes to thank Jean-François Coulombel and Grégory Faye for numerous discussions throughout the research surrounding this paper as well as for leading him towards a simpler redaction of the proof of Theorem \ref{Th}. He also warmly thanks Frédéric Rousset for bringing to his attention the main idea to generalize Theorem \ref{Th} for nonzero-mass initial perturbations.
	
	\printbibliography
\end{document}